\documentclass[12pt]{amsart}
\usepackage{amsmath,amsthm,amsfonts,amscd,amssymb,eucal,latexsym,mathrsfs,enumerate}
\usepackage[all]{xy}

\input{xy}  
\xyoption{all} 
\CompileMatrices 
\newcommand{\potimes}{\widehat{\otimes}} 
\newcommand{\Bu}{B^\dagger }
\newcommand{\Au}{A^\dagger }
\newcommand{\ovl}{\overline}

\newcommand{\vp}{\varepsilon}

\newcommand{\bb}[1]{{\mathbb{#1}}}

\newcommand{\ad}{\mathrm{Ad}}
\newcommand{\wt}{\widetilde}
\newcommand{\id}{\text{id}}
\numberwithin{equation}{section}

\numberwithin{equation}{section}

\theoremstyle{plain}
\newtheorem{lemma}{Lemma}[section]
\newtheorem{proposition}[lemma]{Proposition}
\newtheorem{theorem}[lemma]{Theorem}
\newtheorem{corollary}[lemma]{Corollary}

\newtheorem{introtheorem}{Theorem}

\theoremstyle{definition}

\newtheorem{notation}[lemma]{Notation}

\newtheorem{definition}[lemma]{Definition}

\newtheorem*{definition*}{Definition}

\theoremstyle{remark}
\newtheorem{remark}[lemma]{Remark}

\begin{document}

\title{Perturbations of nuclear C$^*$-algebras}

\author[E.~Christensen]{Erik Christensen}
\address{\hskip-\parindent
Erik Christensen, Institute for Mathematiske Fag, University of Copenhagen, Copenhagen, Demark.}
\email{echris@math.ku.dk}
\author[A.~M.~Sinclair]{Allan M.~Sinclair}
\address{\hskip-\parindent
Allan M.~Sinclair, School of Mathematics, University of Edinburgh, JCMB, King's Buildings, Mayfield Road, Edinburgh, EH9 3JZ, UK.}
\email{a.sinclair@ed.ac.uk}
\author[R.~R.~Smith]{Roger R.~Smith}
\address{\hskip-\parindent
Roger R.~Smith, Department of Mathematics, Texas A{\&}M University,
College Station TX 77843-3368, U.S.A.}
\email{rsmith@math.tamu.edu}
\author[S.~A.~White]{Stuart A.~White}
\address{\hskip-\parindent
Stuart A.~White, Department of Mathematics, University of Glasgow, 
University Gardens, Glasgow Q12 8QW, UK.}
\email{s.white@maths.gla.ac.uk}
\author[W.~ Winter]{Wilhelm Winter}
\address{\hskip-\parindent
Wilhelm Winter, School of Mathematical Sciences, University of Nottingham, Nottingham, NG7 2RD, UK.}
\email{wilhelm.winter@maths.nottingham.ac.uk}

\date{October 26, 2009}
\maketitle

\begin{abstract}
Kadison and Kastler introduced a natural metric on the collection of all ${\mathrm C}^*$-subalgebras of  the bounded operators on a separable Hilbert space. They 
conjectured that sufficiently close algebras are unitarily conjugate.  We establish this 
conjecture when one algebra is separable and nuclear. We also consider
one-sided versions of these notions, and we obtain embeddings from certain
near inclusions involving separable nuclear ${\mathrm C}^*$-algebras.  At the end of the paper we demonstrate how our methods lead to improved characterisations of some of the types of algebras that are of current interest in the classification programme.
\end{abstract}

\section{Introduction}

Kadison and Kastler initiated the uniform perturbation theory of operator algebras in 
\cite{Kadison.Kastler}.  They considered the collection of all operator algebras 
acting on a fixed separable Hilbert space and equipped this set with a metric 
induced by  the Hausdorff distance between the unit balls. In general terms, this 
means that two algebras $A$ and $B$ are close if each element in the unit ball of 
one algebra can be closely approximated by an element in the unit ball of the other 
algebra. They conjectured that suitably close operator algebras must be unitarily 
conjugate and this  has been verified in various situations. For von Neumann 
algebras, the injective case was settled in \cite{Christensen.Perturbations1} and 
\cite{Raeburn.CohomologyPerturbation}. It is also known to be true   for certain 
special classes of separable nuclear ${\mathrm C}^*$-algebras, including the separable AF algebras 
\cite{Christensen.NearInclusions,Khoshkam.UnitaryEquivalence,Phillips.PerturbationAF,Phillips.Perturbations2}. In this paper, our primary purpose is to give an 
affirmative answer to Kadison and Kastler's question when one algebra is a 
separable nuclear ${\mathrm C}^*$-algebra, a hypothesis which automatically implies the 
same property for nearby algebras. In this introduction we discuss these results in 
qualitative terms; precise estimates will be given in the main text.

The original question of Kadison and Kastler leads naturally to the more general 
one-sided situation of near inclusions of two ${\mathrm C}^*$-algebras $A$ and $B$ 
introduced by the first named author in \cite{Christensen.NearInclusions}. 
Heuristically, this means that every element of the unit ball of $A$ can be 
approximated closely by some element of the unit ball of $B$, but we do not require 
a reverse approximation.  The reformulation of Kadison and Kastler's question in this context is to  
ask whether an  embedding of $A$ into $B$ can be found whenever $A$ is very 
nearly included in $B$.  We are also able to  resolve this problem positively in three 
situations:
\begin{itemize}
\item[\rm (i)] when $A$ is  separable, has finite nuclear dimension and no 
hypotheses are imposed on $B$;
\item[\rm (ii)] when $A$ is separable, unital and has approximately inner half flip and 
no hypotheses are imposed on $B$;
\item[\rm (iii)] when $A$ is separable and both algebras are nuclear.
\end{itemize}
  The notion of  {\emph {nuclear dimension}} was recently introduced by the last 
named author and Zacharias in \cite{Zacharias.NuclearDimension} and  extends 
the {\emph {decomposition rank}} defined by Kirchberg and the last named author 
in  \cite{Winter.CoveringDimension}. Noncommutative topological dimension is 
particularly relevant in Elliott's programme to classify nuclear C$^{*}$-algebras by 
$K$-theoretic invariants; in fact, all separable simple nuclear C$^{*}$-algebras 
presently covered by known classification theorems have finite nuclear dimension.
   
The first positive answer to Kadison and Kastler's question was given independently 
by the first named author \cite{Christensen.PerturbationsType1} and  Phillips 
\cite{Phi74} when one of $A$ or $B$ is a type $\mathrm{I}$ von 
Neumann algebra. Combining the results of  \cite{Christensen.Perturbations1} with 
those of Raeburn and Taylor \cite{Raeburn.CohomologyPerturbation}, the question 
was subsequently answered positively when one of $A$ or $B$ is an injective von 
Neumann algebra. This  work was formulated variously in terms of injectivity, 
hyperfiniteness and Property $P$, since it predates Connes'  work on the 
equivalence of these notions \cite{Connes.InjectiveFactors}. The most general 
statement of these results was later given in \cite[Corollary 4.4]
{Christensen.NearInclusions}, where it was shown that if $A$ is an injective von 
Neumann algebra acting non-degenerately on some Hilbert space $H$ and $B$ is a 
${\mathrm C}^*$-subalgebra of $\mathbb B(H)$, which is sufficiently close to $A$, then there 
exists a unitary operator $u\in(A\cup B)''$ with $uAu^*=B$. Furthermore, this unitary 
can be taken to be close to the identity, by which we mean that  $\|u-I_H\|$ can be controlled in terms of 
the distance between $A$ and $B$.  A continuous path of unitaries $u_t$ 
connecting $u_0=I_H$ to $u_1=u$ then leads to a continuous deformation  
$A_t=u_tAu_t^*$, $0\leq t \leq 1$, of $A_0=A$ into $A_1=B$. In this way we can 
regard $B$ as a \emph{perturbation} of $A$.   The other situation in which Kadison 
and Kastler's question has been resolved for von Neumann algebras is for two 
close von Neumann subalgebras $A$ and $B$ of a common \emph{finite} von 
Neumann algebra \cite{Christensen.Perturbations2}, where again there is a unitary  
$u\in (A\cup B)''$ close to the identity with $uAu^*=B$.  Perturbation problems have 
also been studied in the ultraweakly closed non-self-adjoint setting 
\cite{Davidson.PerturbationReflexive,Lance.PerturbationNest,Pitts}.

We now turn to perturbation results for ${\mathrm C}^*$-algebras. A consequence of 
the results for injective von Neumann algebras described above and Connes' work 
\cite{Connes.InjectiveFactors} is that any ${\mathrm C}^*$-algebra which is sufficiently close to 
a nuclear ${\mathrm C}^*$-algebra is also nuclear, \cite[Theorem 6.5]
{Christensen.NearInclusions}.  Perturbation results of the form that two sufficiently 
close ${\mathrm C}^*$-algebras $A$ and $B$ must be unitarily conjugate by a unitary $u\in(A
\cup B)''$ have been established under the following sets of hypotheses:
\begin{itemize}
\item[\rm (i)] Either $A$ or $B$ is separable and AF 
\cite{Christensen.NearInclusions} (see also \cite{Phillips.PerturbationAF}).
\item[\rm (ii)] Either $A$ or $B$ is continuous trace and either unital or separable 
\cite{Phillips.Perturbations2}.
\end{itemize}
Perturbation results have also been established for certain extensions of the 
${\mathrm C}^*$-algebras in the classes above by Khoshkam in \cite{Khoshkam.UnitaryEquivalence}.  
On the near inclusion side, Johnson has obtained embeddings from near inclusions 
$A\subset_\gamma B$ when $B$ is a separable subhomogeneous ${\mathrm C}^*$-algebra, 
\cite{Johnson.NearInclusions}. In contrast to the injective von Neumann algebra 
case, we cannot always expect to obtain an estimate which controls  $\|u-I_H\|$ in 
terms of the distance between $A$ and $B$.  Indeed, Johnson
\cite{Johnson.PerturbationExample} has
 constructed two faithful representations of $C[0,1]\otimes\mathbb K$  on some 
separable Hilbert space whose images $A$ and $B$ are unitarily conjugate and can 
be taken to be arbitrarily close, but for which there is no isomorphism $\theta:A
\rightarrow B$ with $\|\theta(x)-x\|\leq\|x\|/70$ for $x\in A$. Here, $\mathbb K$ denotes the algebra of compact operators. 

Kadison and Kastler's original paper \cite{Kadison.Kastler} shows that sufficiently close von Neumann algebras have the same type decompositions. Inspired by this work, various authors have examined properties of close operator algebras: \cite{Phillips.Perturbation} shows that close C$^*$-algebras have isomorphic lattices of ideals and homeomorphic spectra, while the work of Khoshkam, \cite{Kho84,Kho90} shows that sufficiently close nuclear C$^*$-algebras have isomorphic $K$-groups and so are $KK$-equivalent if both satisfy the UCT. The first four authors consider problems of this nature related to the similarity length in \cite{Saw.PerturbLength}.

In full generality, Kadison and Kastler's question has a negative answer for close nuclear 
${\mathrm C}^*$-algebras. In \cite{Christensen.CounterExamples} Choi and the first named 
author gave examples of arbitarily close nuclear ${\mathrm C}^*$-algebras which are not 
$*$-isomorphic.  These examples are even approximately finite dimensional, but they 
are not separable.  Our first main theorem, stated qualitatively below and 
quantitatively as Theorem \ref{Close.Iso.Main}, shows that separability is the only 
obstruction to an isomorphism result for nuclear ${\mathrm C}^*$-algebras.
\begin{introtheorem}\label{Intro.Iso}
Let $H$ be a Hilbert space and let $A$ and $B$ be ${\mathrm C}^*$-subalgebras of $\mathbb 
B(H)$. If $A$ is separable and nuclear and $B$ is sufficiently close to $A$, then 
there is a surjective $*$-isomorphism $\alpha:  A\to B$.
\end{introtheorem}
Johnson's examples \cite{Johnson.PerturbationExample} show that we cannot 
demand that the isomorphism of Theorem \ref{Intro.Iso} is uniformly close to the 
inclusion of $A$ into $\mathbb B(H)$. However, our methods do allow us to specify 
a finite set $X$ of the unit ball of $A$ and construct an isomorphism $\alpha:A
\rightarrow B$ which almost fixes $X$. This additional  control is crucial in obtaining 
a unitary in the von Neumann algebra ${\mathrm W}^*(A, B,I_H)$ which conjugates $A$ into 
$B$ (provided the underlying Hilbert space is separable).  Our second main 
theorem accomplishes  this and so gives a complete answer to Kadison and 
Kastler's question when one algebra is a separable nuclear ${\mathrm C}^*$-algebra.  The 
quantitative version of this theorem is Theorem \ref{thm10.4} in the text.
\begin{introtheorem}\label{Intro.Unitary}
Let $H$ be a separable Hilbert space and let $A$ and $B$ be ${\mathrm C}^*$-subalgebras of 
$\mathbb B(H)$. If $A$ is separable and nuclear and $B$ is sufficiently close to $A
$, then there exists a unitary operator $u\in (A\cup B)''$ with $uAu^*=B$.
\end{introtheorem}

To put the techniques involved in Theorem \ref{Intro.Iso} into context, it is helpful to 
first discuss the perturbation results for injective von Neumann algebras from 
\cite{Christensen.Perturbations1}. Given two close injective von Neumann algebras 
$M,N\subseteq\mathbb B(H)$, take a conditional expectation $\Phi:
\mathbb B(H)\rightarrow N$.  If we restrict $\Phi$ to $M$ we obtain a completely 
positive and contractive map (cpc map) $\Phi|_M$ from $M$ into $N$ which is 
uniformly close to the inclusion of $M$ into $\mathbb B(H)$. Then $\Phi|_M$ is 
almost a multiplicative map, in that 
$$
\sup_{u\in\mathcal U(M)}\|\Phi(u)\Phi(u^*)-1\|
$$
is small.  The main idea behind \cite{Christensen.Perturbations1} is to use injectivity 
of $M$ to show that a  completely positive contractive normal map $\Psi:M\rightarrow P$ which is almost multiplicative must be 
uniformly close to a $*$-homomorphism.  The $*$-homomorphism is obtained 
from $\Psi$ by integrating the Stinespring projection for $\Psi$ over the unitary groups of an increasing family of dense 
finite dimensional subalgebras of $M$ and so this can be regarded as an averaging 
result.  Modulo technicalities regarding normality, this procedure can be applied to $
\Phi|_M$ to obtain a $*$-homomorphism from $M$ onto $N$ which is close to the 
inclusion of $M$ into $\mathbb B(H)$.  A second averaging argument is then used 
to show that such maps are spatially implemented. The general question of
when approximately multiplicative maps between Banach algebras are close to multiplicative
maps has been studied by
Johnson, \cite{Johnson.ANMN}.

Now suppose that we have two close ${\mathrm C}^*$-algebras $A$ and $B$ on some Hilbert 
space $H$ and that one of these algebras is separable and nuclear.  Results from 
\cite{Christensen.NearInclusions} show that both algebras must then be separable 
and nuclear.  We do not have a conditional expectation onto $B$ which we can use 
to obtain a cpc map $A\rightarrow B$ uniformly close to the inclusion of $A
$ into $\mathbb B(H)$, but we can use Arveson's extension theorem, \cite{Arveson.Subalgebras}, to produce 
completely positive maps from $A\rightarrow B$ which approximate this inclusion on 
finite subsets of the unit ball of $A$ (see Proposition \ref{Prelim.PtArveson}).  
Accordingly we look to develop point-norm versions of the averaging techniques 
from \cite{Christensen.Perturbations1} for nuclear ${\mathrm C}^*$-algebras.  This is the 
subject of Section \ref{Avg} and the critical ingredient is the \emph{amenability} of a 
nuclear ${\mathrm C}^*$-algebra established by Haagerup in 
\cite{Haagerup.NuclearAmenable}.  The first of these lemmas (Lemma \ref{Avg.1}) 
enables us to obtain cpc maps $A\rightarrow B$ which are almost multiplicative on a 
a finite set $Y$ of the unit ball of $A$ up to a specified tolerance $\varepsilon$ from 
cpc maps $A\rightarrow B$ which are almost multiplicative on some finite set $X$ of 
the unit ball of $A$ up to a fixed tolerance of $1/17$. The set $X$ can be thought of 
as a F{\o}lner set for $Y$ and $\varepsilon$.  We can apply this lemma to the cpc 
maps arising from Arveson's extension theorem to produce cpc maps $A\rightarrow 
B$ which are almost multiplicative on arbitrary finite sets of the unit ball of $A$ up to 
an an arbitrary small tolerance. The next stage is to construct a $*$-homomorphism 
$A\rightarrow B$ as a point norm limit of these maps. Our second lemma (Lemma 
\ref{Avg.2}) enables us to conjugate these maps by unitaries to ensure this point 
norm convergence.  An intertwining argument (Lemma \ref{Close.Iso.Tech}) 
inspired by \cite[Theorem 6.1]{Christensen.NearInclusions} and those in the 
classification programme gives a $*$-isomorphism $A\rightarrow B$.  At this point 
separability of $A$ is crucial.

It is perhaps worth noting that we use a variety of characterisations of nuclearity in 
the course of the proof of Theorem \ref{Intro.Iso}. The equivalence between 
nuclearity of $A$ and injectivity of $A^{**}$ is used in 
\cite{Christensen.NearInclusions} to show that nuclearity transfers to close 
${\mathrm C}^*$-algebras. The characterization by the completely positive approximation 
property, due to Choi and Effros \cite{Choi.EffrosCPAP}, allows us to find cpc maps $A
\rightarrow B$, and the amenability of nuclear ${\mathrm C}^*$-algebras, 
\cite{Haagerup.NuclearAmenable}, is an essential ingredient for   converting these maps into a $*$-isomorphism 
$A\rightarrow B$.

We now turn to Theorem \ref{Intro.Unitary}. Earlier results from 
\cite{Christensen.NearInclusions} and elementary techniques enable us to reduce 
to the situation of two close separable nuclear ${\mathrm C}^*$-algebras $A$ and $B$ which 
are non-degenerately represented on some separable Hilbert space and have the 
same ultraweak closure. By repeatedly applying Theorem \ref{Intro.Iso} and Lemma 
\ref{Avg.2}, we can construct a sequence $\{u_n\}_{n=1}^\infty$ of unitaries so 
that $\lim\ad(u_n)$ converges in point norm topology to an isomorphism between  
$A$ and $B$.   If this sequence was $^*$-strongly convergent to a unitary, then this 
unitary would implement a spatial isomorphism between $A$ and $B$.  However 
there is no reason why this should be so; indeed the sequence $\{u_n\}_{n=1}^\infty$ 
is not even guaranteed to have a non-zero ultraweak accumulation point. Instead 
we explicitly modify the sequence $\{u_n\}_{n=1}^\infty$ to force the required 
$^*$-strong convergence while still retaining control over the point norm limit of $\lim
\ad(u_n)$.  This adjustment procedure is inspired by Bratteli's classification of 
representations of AF algebras \cite[Section 4]{Bratteli.AF}, which in turn builds 
upon the  work of Powers for UHF algebras in \cite{Powers.Representations}.  A key 
observation in this work is that if $A$ is a ${\mathrm C}^*$-algebra non-degenerately 
represented on $H$ with ultraweak operator closure $M=A''$ and $F$ is a finite 
dimensional ${\mathrm C}^*$-subalgebra of $A$, then $(F'\cap A)''=F'\cap M$. The Kaplansky density theorem then enables unitaries in $M$ commuting with $F$ to be approximated in $^*$-strong topology by unitaries in $A$ commuting with $F$.  In Lemma \ref{lem3.9} we prove a Kaplansky density theorem for unitaries with a uniform spectral gap which approximately commute with suitable finite sets, again using amenability of nuclear ${\mathrm C}^*$-algebras. Using this result, a technical argument enables us to make the suitable adjustments to the sequence of unitaries $\{u_n\}_{n=1}^\infty$ described above to prove Theorem \ref{Intro.Unitary}. This is the subject of Section \ref{Unitary}.

The procedure used to obtain Theorem \ref{Intro.Iso} forms the basis of our near inclusion results.  Given a sufficiently small near inclusion of $A$ in $B$ with $A$ separable and nuclear, our intertwining argument gives an embedding $A\hookrightarrow B$ whenever we can produce cpc maps $A\rightarrow B$ which almost fix finite sets in the unit ball of $A$ (up to a tolerance depending on the near inclusion constant).  In two situations the existence of these maps is immediate: when $B$ is nuclear the maps are given by Arveson's extension theorem, while if $A$ is unital and has  approximately inner half flip the maps are given by \cite[Proposition 6.7]{Christensen.NearInclusions}.  The third, and least restrictive, hypothesis under which we can produce these maps is when $A$ has finite nuclear dimension.  Here further work is required to construct our cpc maps.  Using the completely positive approximation property for nuclear ${\mathrm C}^*$-algebras, \cite{Choi.EffrosCPAP}, we can approximately factorise the identity map on a nuclear ${\mathrm C}^*$-algebra $A$ through matrix algebras $\mathbb M_k$ using cp maps.  When $A$ has finite nuclear dimension $n$, these factorisations have additional structure: the maps $\mathbb M_k\rightarrow A$ decompose as the  sum of $(n+1)$ cpc maps which preserve orthogonality (order zero maps). The main technical lemma of Section \ref{Near} is a perturbation result for order zero maps $\mathbb M_k\rightarrow A$ producing a nearby cpc map $\mathbb M_k\rightarrow B$ whenever $A$ is nearly contained in $B$.  Combining this with the intertwining argument gives the theorem below, which is stated quantitatively as Theorem \ref{Near.Embedd}.

\begin{introtheorem}\label{Intro.Near}
Let $H$ be a Hilbert space and let $A$ and $B$ be ${\mathrm C}^*$-subalgebras of $\mathbb B(H)$. If $A$ is separable and has finite nuclear dimension and is nearly contained in $B$, then there is an embedding $A\hookrightarrow B$.
\end{introtheorem}

Our paper is organised as follows. In the next section we set out the notation used in the paper, establish some basic facts and recall a number of results from the literature for the reader's convenience.  In Section \ref{Avg} we discuss amenability for ${\mathrm C}^*$-algebras and give our point norm averaging results and Kaplansky density result.  Section \ref{Close} contains the intertwining argument (Lemma \ref{Close.Iso.Tech}) which combines the averaging results of Section \ref{Avg} to prove Theorem \ref{Intro.Iso} (Theorem \ref{thm4.2}, Theorem \ref{Close.Iso.Main}).  We also give two near inclusion results and some other consequences at this stage.  Section \ref{Unitary} contains the proof of Theorem \ref{Intro.Unitary}. We begin Section \ref{Near} with a review of order zero maps between ${\mathrm C}^*$-algebras and prove our perturbation theorem for these maps, using this to establish Theorem \ref{Intro.Near}. We also recall the salient facts about the nuclear dimension for the readers convenience.  We end the paper in Section \ref{Direct} with some sample applications of our techniques to other situations. Firstly we give a strengthened local characterisation of inductive limits of finitely presented weakly semiprojective nuclear ${\mathrm C}^*$-algebras. Secondly we revisit our perturbation theorem for order zero maps to show that the resulting map can also be taken of order zero, and we close by  presenting an improved characterisation of $\mathcal Z$-stability for nuclear ${\mathrm C}^*$-algebras.

\noindent\textbf{Acknowledgments. } The authors express their gratitude for the following sources of financial support:  NSF DMS-0401043 (R.R.S.) and EPSRC First Grant EP/G014019/1 (W.W.). A visit by E.C. to A.M.S. and S.A.W. was supported by a grant from the Edinburgh Mathematical Society, while R.R.S. and S.A.W. gratefully acknowledge the kind hospitality of the University of Copenhagen Mathematics Department while visiting E.C. Part of this work was undertaken while S.A.W was a participant at the Workshop in Analysis and Probability at Texas A\&{}M University and he would like to thank the organisers and NSF for their financial support. Finally, the authors record their gratitude to Bhishan Jacelon, Simon Wassermann and Joachim Zacharias for stimulating conversations concerning the paper.

\section{Preliminaries}
In this section we establish notation and recall some standard results. We begin with Kadison and Kastler's metric on operator algebras from 
\cite{Kadison.Kastler}.

\begin{definition}
Let $C$ be a ${\mathrm C}^*$-algebra.  We equip the collection of ${\mathrm C}^*$-subalgebras of $C$ 
with a metric $d$ by applying the Hausdorff metric to the unit balls of these 
subalgebras. That is $d(A,B)<\gamma$ if, and only if, for each $x$ in the unit ball of 
either $A$ or $B$, there exists $y$ in the unit ball of the other algebra with $\|x-y\|<
\gamma$.
\end{definition}

In the second half of the paper, we shall also use the notion of near containment 
introduced in \cite{Christensen.NearInclusions}.
\begin{definition}\label{Prelim.DefNear}
Suppose that $A$ and $B$ are ${\mathrm C}^*$-subalgebras of a ${\mathrm C}^*$-algebra $C$ 
and let $\gamma>0$.  Write $A\subseteq_\gamma B$ if given $x$ in the unit ball of 
$A$ there exists $y\in B$ with $\|x-y\|\leq\gamma$. Note that we do not require that 
$y$ lie in the unit ball of $B$.  Write $A\subset_\gamma B$ if there exists $
\gamma'<\gamma$ with $A\subseteq_{\gamma'}B$. For subsets $X$ and $Y$ 
of the unit ball of a ${\mathrm C}^*$-algebra with $X$ finite, the notation $X\subseteq_{\gamma}Y$ will mean that each $x\in X$ has a corresponding element $y\in Y$ satisfying $\|x-y\|\leq \gamma$ and $X\subset_{\gamma}Y$ will mean that each $x\in X$ has a corresponding element $y\in Y$ such that $\|x-y\|<\gamma$.   
\end{definition}

\begin{remark}\label{Prelim.Metric.Rem}
An equivalent notion of distance between ${\mathrm C}^*$-algebras was introduced in 
\cite{Christensen.NearInclusions} using near containments.  Define $d_0(A,B)$ to 
be the infimum of all $\gamma$ for which $A\subseteq_\gamma B$ and $B
\subseteq_\gamma A$. The difference between $d_0$ and $d$ arises from the fact 
that we do not require the $y$ in Definition \ref{Prelim.DefNear} to lie in the unit ball 
of $B$.  It is immediate that $d(A,B)\leq d_0(A,B)\leq 2d(A,B)$, but the function 
$d_0$ does not appear to satisfy the triangle inequality; for this reason we prefer to 
work with the metric $d$.
\end{remark}
The following proposition is folklore. We will use it to obtain surjectivity of our 
isomorphisms.
\begin{proposition}\label{EasyFact}
Let $A$ and $B$ be ${\mathrm C}^*$-algebras on a Hilbert space with $A\subseteq B$ and $B
\subset_1A$. Then $A=B$.
\end{proposition}

The next proposition records some standard estimates. The first statement follows 
from Lemma 2.7 of \cite{Christensen.PerturbationsType1} and the second can be 
found as \cite[Lemma 6.2.1]{Murphy.Book}.
\begin{proposition}\label{Prelim.UnitaryEquiv} 1.~~Let $x$ be an operator on a 
Hilbert space $H$ with $\|x-I_H\|<1$ and let $u$ be the unitary in the polar decomposition 
of $x$. Then $\|u-I_H\|\leq\sqrt{2}\|x-I_H\|$.

\noindent 2.~~Let $p$ and $q$ be projections in a unital ${\mathrm C}^*$-algebra $A$ with $\|
p-q\|<1$. Then there is a unitary $u\in A$ with $upu^*=q$ and $\|u-1\|\leq \sqrt{2}\|
p-q\|$.
\end{proposition}

On  occasion we will need to lift a near inclusion $A\subset_\gamma B$ to a near 
inclusion of a tensor product $A\otimes D\subset_\mu B\otimes D$. This can be 
done when $D$ is nuclear and $A$ has Kadison's similarity property \cite[Theorem 3.1]
{Christensen.NearInclusions}. The version of these facts below is taken from 
2.10-2.12 of \cite{Saw.PerturbLength} specialised to the case when $A$ is nuclear 
(and so has length $2$ with length constant at most $1$).  

\begin{proposition}[{\cite[Corollary 2.12]{Saw.PerturbLength}}]\label{TensorNuclear}
Let $A,B\subseteq\mathbb B(H)$ be ${\mathrm C}^*$-algebras with $A\subset_\gamma 
B$ for some $\gamma>0$ and $A$ nuclear. Given any nuclear ${\mathrm C}^*$-algebra $D$, 
we have $A\otimes D\subseteq_{2\gamma+\gamma^2} B\otimes D$ inside $
\mathbb B(H)\otimes D$.
\end{proposition}

\noindent We also need a version of the previous proposition for finite sets which we 
state in the context of amplification by matrix algebras, see 
\cite[Remark 2.11]{Saw.PerturbLength}.

\begin{proposition}\label{Near.Length}
Let $A$ be a nuclear ${\mathrm C}^*$-algebra on some Hilbert space $H$.   Then for each $n
\in\mathbb N$ and each finite subset $X$ of the unit ball of $A\otimes \mathbb M_n
$, there is a finite subset $Y$ of the unit ball of $A$ with the following property.  
Whenever  $B$ is another ${\mathrm C}^*$-algebra on $H$ with $Y
\subseteq_\gamma B$ for some $\gamma>0$,  then $X\subseteq_{2\gamma+
\gamma^2}B\otimes\mathbb M_n$. In particular if $A\subset_\gamma B$, then $A
\otimes\mathbb M_n\subset_{2\gamma+\gamma^2} B\otimes\mathbb  M_n$ for all 
$n\in\mathbb N$.
\end{proposition}
\begin{remark}\label{Near.Length.Rem}
Note that the same result holds for rectangular matrices, i.e. under the same 
hypotheses as the proposition, given a finite subset $X$ of the unit ball of $\mathbb 
M_{1\times r}(A\otimes\mathbb M_n)$ for some $r,n\in\mathbb N$, then there is a 
corresponding finite subset $Y$ of the unit ball of $A$ such that whenever $Y
\subseteq_\gamma B$, then $X\subseteq_{2\gamma+\gamma^2} \mathbb 
M_{1\times r}(B\otimes\mathbb M_n)$.  This follows by working in $\mathbb 
M_r(\mathbb B(H)\otimes \mathbb M_n)\cong \mathbb B(H)\otimes \mathbb M_{rn}
$ and then cutting down after the approximations have been performed.  
\end{remark}

In Sections \ref{Close} and \ref{Unitary} we need to transfer nuclearity and separabilty to close subalgebras, so that our main results only require hypotheses on one algebra. The next two results enable us to do this. The first is due to the first named author and 
requires the equivalence between nuclearity of $A$ and injectivity of the bidual 
$A^{**}$, see the account in \cite{BrownOzawa}. The second is folklore, appearing 
in the proof of \cite[Theorem 6.1]{Christensen.NearInclusions}, for example.  We 
give a proof of the latter statement for completeness.
\begin{proposition}[{\cite[Theorem 6.5]{Christensen.NearInclusions}}]
\label{Prelim.NuclearOpen}
Let $A$ and $B$ be ${\mathrm C}^*$-subalgebras of some ${\mathrm C}^*$-algebra $C$ with $d(A,B)<1/101$.  
Then $A$ is nuclear if, and only if, $B$ is nuclear.
\end{proposition}

\begin{proposition}\label{Prelim.SeparableOpen}
Let $A$ and $B$ be ${\mathrm C}^*$-subalgebras of some ${\mathrm C}^*$-algebra $C$ with $B\subset_{1/2}A$. 
If $A$ is separable, then $B$ is separable.
\end{proposition}
\begin{proof}
Suppose that $A$ is separable and suppose that $B\subset_{1/2}A$.  Fix $
0< \gamma<1/2$ with $B\subset_{\gamma}A$. Take $0<\vp<1-2\gamma$, and let  $\{b_i:i\in I \}$ be a maximal 
set of elements in the unit ball  of $B$ such that  $\|b_i-b_j\| \geq 1-\vp$ 
when $i\neq j$. For each $i\in I$, find an operator $a_i$ in the unit ball of  $A$ with 
$\|a_i-b_i\|<\gamma$.  For $i\neq j$, we have $\|a_i-a_j\|\geq 1-\vp-2\gamma
>0$,  so separability of $A$ implies that $I$ is countable. Let $C$ be the closed 
linear span of $\{b_i:i\in I\}$.  If $C\neq B$, take a unit vector $x\in B/C$ and write 
$x=y+C$ for some $y$ with $1\leq \|y\|<1+\vp$. Let $\tilde{y}=y/\|y\|$. If $\|
\tilde{y}-b_i\|<1-\vp$ for some $i$, then 
\begin{equation}
\|x\|_{B/C}\leq\|y-b_i\|\leq\|y-\tilde{y}\|+\|\tilde{y}-b_i\|<\vp+
1-\vp=1,
\end{equation}
which is a contradiction.  Thus we can adjoin $y$ to $\{b_i:i\in I\}$ contradicting 
maximality.  Hence $B=C$ and so $B$ is separable.
\end{proof}

In Section \ref{Unitary} we shall conjugate ${\mathrm C}^*$-algebras by unitaries to reduce to 
the situation of close separable nuclear ${\mathrm C}^*$-algebras which have the same  
ultraweak closure using the next two propositions.  The first is \cite[Proposition 3.2]
{Saw.PerturbLength}, while the second  was established in 
\cite{Christensen.NearInclusions}.  Note that the distance used in 
\cite{Christensen.NearInclusions} is the quantity $d_0$ described in Remark 
\ref{Prelim.Metric.Rem}. We restate the result we need in terms of $d$ and add a 
shared unit assumption which is implicit in the original version.

\begin{proposition}\label{Prelim.Unit}
Let $A$ and $B$ be ${\mathrm C}^*$-subalgebras of a unital ${\mathrm C}^*$-algebra $C$ satisfying 
$d(A,B)<\gamma<1/4$. Then $A$ is unital if and only if $B$ is unital. In this case 
there is a unitary $u\in C$ with $\|u-1_C\|<2\sqrt{2}\gamma$ and $u1_Au^*=1_B$.
\end{proposition}

\begin{proposition}[{\cite[Corollary 4.2 (c)]{Christensen.NearInclusions}}]
\label{Prelim.Inject}
Let $M$ and $N$ be injective von Neumann algebras on a Hilbert space $H$ which 
share the same unit.  If $d(M,N)<\gamma<1/8$, then there exists a unitary $u\in (M
\cup N)''$ with $uMu^*=N$ and $\|u-I_H\|\leq 12\gamma$.
\end{proposition}

Distances between maps restricted to finite sets will be a recurring theme in the 
paper, as will be maps which act almost as $*$-homomorphisms, so we formalise 
these with the following notational definitions.
\begin{definition}
Given two maps $\phi_1,\phi_2:A\rightarrow B$ between normed spaces, a subset 
$X\subseteq A$ and $\vp>0$, write $\phi_1\approx_{X,\vp}\phi_2$ 
if $\|\phi_1(x)-\phi_2(x)\|\leq \vp$ for $x\in X$.  When $A$ and $B$ are both subspaces of $\mathbb B(H)$, we use $\iota$ to denote the inclusion maps so $\phi_1\approx_{X,\vp}\iota$ means $\|\phi_1(x)-x\|\leq\vp$ for $x\in X$.
\end{definition}

\begin{definition}\label{Prelim.Def.Approxhm}
Let $A$ and $B$ be ${\mathrm C}^*$-algebras, $X$ a subset of $A$ and $\varepsilon>0$. A bounded 
linear map $\phi:A\rightarrow B$ is an \emph{$(X,\vp)$-approximate 
$*$-homomorphism} if it is a completely positive contractive map 
(cpc map) and the 
inequality
\begin{equation}
\|\phi(x)\phi(x^*)-\phi(xx^*)\|\leq \vp,\quad x\in X\cup X^*,
\end{equation}
is satisfied.
\end{definition}
The following well-known consequence of Stinespring's theorem, which can be 
found as  \cite[Lemma 7.11]
{Kirchberg.InfiniteAbsorbCuntz}, shows why we only consider pairs of the form 
$x,x^*$ in the previous definition.
\begin{proposition}\label{Prelim.ApproxhmEst}
Let $\phi:A\rightarrow B$ be a cpc map between ${\mathrm C}^*$-algebras. For $x,y\in A$, we 
have
\begin{equation}
\|\phi(xy)-\phi(x)\phi(y)\|\leq\|\phi(xx^*)-\phi(x)\phi(x^*)\|^{1/2}\|y\|.
\end{equation}
\end{proposition}

Given a near inclusion $A\subset_\gamma B$, the next two results give conditions 
under which we can find cpc maps $A\rightarrow B$ which approximate  on finite sets
the inclusion map 
of $A$ into the bounded operators on the underlying Hilbert space.  Our first result of this type is 
an application of Arveson's extension theorem \cite{Arveson.Subalgebras}. It uses 
the characterisation of nuclearity, due to Choi and Effros, \cite{Choi.EffrosCPAP}, in 
terms of approximate factorisations of the identity map by completely positive maps 
through matrix algebras.    The second is a version of \cite[Proposition 6.7]
{Christensen.NearInclusions}, for unital ${\mathrm C}^*$-algebras with approximately inner 
half flip. We include a proof here to remove the implicit assumption of a shared unit of the original version. A third result of this type will be obtained 
in Section \ref{Near} when $A$ has finite nuclear dimension.

\begin{proposition}\label{Prelim.PtArveson}
Let $A$ and $B$ be two ${\mathrm C}^*$-algebras on a Hilbert space $H$ with $B$ nuclear.  
Given  a finite set $X$ in the unit ball of $A$ with $X\subset_\gamma B$,  there 
exists a cpc map $\phi :  A\to B$ such that $\|\phi(x)-x\|\leq2\gamma$ for $x\in X
$.
\end{proposition}
\begin{proof}
Since $X\subset_\gamma B$, we may choose $\gamma' < \gamma$ so that $X
\subset_{\gamma'} B$. Let $\vp=2(\gamma-\gamma')>0$.
Find a finite subset $Y$ of $B$ such that $X\subseteq_{\gamma'}Y$.  Use the nuclearity of $B$ to find a finite dimensional 
${\mathrm C}^*$-algebra $F$ and cpc maps $\alpha :  B\rightarrow F$ and $\beta :  F
\rightarrow B$
so that $
\beta\circ\alpha\approx_{Y,\varepsilon}\id_B$.
Arveson's extension theorem, \cite{Arveson.Subalgebras}, allows us to extend $
\alpha$ to a cpc map $\tilde\alpha :  {\mathbb B}(H)\to F$. For a given $x\in X
$, choose $y\in Y$ so that $\|x-y\|\leq \gamma'$. Then
\begin{equation}
\|\beta(\tilde\alpha(x))-x\|\leq\|\beta(\tilde\alpha(x-y))\|+\|\beta(\tilde{\alpha}(y))-y\|+\|
x-y\|\leq 2\gamma'+\varepsilon=2\gamma,
\end{equation}
so the result holds with $\phi=\beta\circ\tilde\alpha$.
\end{proof}
Recall that a unital ${\mathrm C}^*$-algebra $A$ has  approximately inner half flip if the 
following condition is satisfied. Given a finite subset $F\subseteq A$ and $\vp > 0$, 
there exists a unitary $u$ in the spatial ${\mathrm C}^*$-tensor product $A\otimes A$  so that 
\begin{equation}
\|u(1\otimes x)u^*-x\otimes 1\|_{A\otimes A}< \vp,\quad  x\in F.
\end{equation}
Such ${\mathrm C}^*$-algebras are automatically nuclear by the proof of Proposition 2.8 of 
\cite{Effros.ApproxInnerFlip}.  The next result is established by reducing to the case 
where $A$ and $B$ are both unital with the same unit and then following 
\cite[Proposition 6.7]{Christensen.NearInclusions}.

\begin{proposition}\label{InnerFlip}
Let $A$ and $B$ be a ${\mathrm C}^*$-algebras on a Hilbert space $H$. Suppose that $A$ is 
unital with approximately inner half flip and $A\subset_{\gamma}B$ for $
\gamma<1/2$.  Then for all finite sets $X$ in the unit ball of $A$, there exists a cpc 
map $\phi:A\rightarrow B$ with 
\begin{equation}\label{InnerFlip.1}
\|\phi(x)-x\|\leq 8\alpha+4\alpha^2+4\sqrt{2}\gamma,\quad x\in X,
\end{equation}
where $\alpha=(4\sqrt{2}+1)\gamma+4\sqrt{2}\gamma^2$.
\end{proposition}
\begin{proof}
Fix $\gamma'<\gamma$ so that $A\subseteq_{\gamma'}B$.  By Lemma 2.1 of 
\cite{Christensen.PerturbationsType1}, find a projection $p\in B$ with $\|p-1_A\|\leq 
2\gamma'$. By Part 2 of Lemma \ref{Prelim.UnitaryEquiv} find a unitary $u$ on $H$ 
with $upu^*=1_A$ and $\|u-I_H\|\leq \sqrt{2}\|p-1_A\|$.  Define 
$B_0=u(pBp)u^*=1_AuBu^*1_A$ so that $A$ and $B_0$ are both unital and share the same unit. Given $x$ in the unit ball of $A$, there exists $y\in 
B$ with $\|x-y\|\leq \gamma'$.  Then 
\begin{align}
\|x-1_Auyu^*1_A\|&\leq\|x-uyu^*\|\leq \|x-y\|+2\|u-I_H\|\|y\|\notag\\
&\leq 
\gamma'+4\sqrt{2}\gamma'(1+\gamma').
\end{align}
Then $A\subseteq_{\alpha'}B_0$, where $\alpha'=\gamma'+4\sqrt{2}\gamma'(1+\gamma')$.  

Fix a finite subset $X$ of the unit ball of $A$ and set $\varepsilon=\alpha-
\alpha'>0$. As $A$ has  approximately inner half flip, find a unitary $v\in A\otimes A
$ with $\|v(1\otimes x)v^*-x\otimes 1\|<\varepsilon$ for $x\in X$.  Since $A$ is 
nuclear, Proposition \ref{TensorNuclear} gives some $w$ in the unit ball of $B_0
\otimes A$ with $\|v-w\|\leq 2(2\alpha'+\alpha'^2)$.  Given a state $\psi$ on $A$, let 
$R_\psi:{\mathrm C}^*(B_0, A)\rightarrow B_0$ be the cpc slice map induced by $
\psi$ and note that $R_{\psi}(B_0\otimes A)=B_0$.  Then $\phi(x)=u^*R_{\psi}(w(1\otimes x)w^*)u$ defines a cpc map $A
\rightarrow B$ since $u^*B_0u\subseteq B$. 
We have
\begin{align}
\|\phi(x)-x\|&\leq \|\phi(x)-R_\psi(w(1\otimes x)w^*)\|+\|R_\psi(w(1\otimes x)w^*)-R_\psi(v(1\otimes x)v^*)\|\notag\\
&\quad+\|R_\psi(v(1\otimes x)v^*)-x\|\notag\\
&\leq 2\|u-I_H\|+2\|w-v\|+\varepsilon\notag\\
&\leq4\sqrt{2}\gamma+8\alpha+4\alpha^2,\quad x\in X,
\end{align}
so (\ref{InnerFlip.1}) holds.
\end{proof}
As is often the case in perturbation theory, improved constants can be obtained in Proposition \ref{InnerFlip} under the assumption that $A$ and $B$ share the same unit. This is also true of many of our subsequent results. In order to provide a unified treatment of the unital and non-unital cases, we use two distinct unitizations in this paper, detailed below.
\begin{notation}\label{Prelim.Dagger}
Given a ${\mathrm C}^*$-algebra $A$, we write $A^\dagger$ for the following unitisation of $A$.  When 
$A$ is concretely represented on a Hilbert space $H$, the algebra $A^\dagger$ is 
obtained as ${\mathrm C}^*(A,e_A)$, where $e_A$ is the support projection of $A$.  Given a 
cpc map $\phi:A\rightarrow B$, with $A$ non-unital, it is well known that we can 
extend $\phi$ to a unital completely positive map (ucp map) $\tilde{\phi}:A^\dagger
\rightarrow B^\dagger$ by defining $\tilde{\phi}(1_{A^\dagger})=1_{B^\dagger}$  
(see \cite[Section 1.2]{Sinclair.CohomBook} or \cite[Section 2.2]{BrownOzawa}).  
However in the next section we will consider cpc maps $\phi:A\rightarrow B$ where 
$A$ may be unital and it is convenient to also convert these to ucp maps.  To this 
end we introduce another unitisation of $A$.   Take a faithful representation $\pi$ of 
$A$ on a Hilbert space $H$ and form the Hilbert space $\widetilde{H}=H\oplus 
\mathbb C$.  The representation $\pi$ extends to $\widetilde{H}$ by $x\mapsto 
\begin{pmatrix}\pi(x)&0\\0&0\end{pmatrix}$. We define $
\widetilde{A}={\mathrm C}^*(\pi(A),I_{\widetilde{H}})$. This construction is independent of the 
faithful representation $\pi$, and we note that $A$ is always a proper norm closed 
ideal in $\widetilde{A}$.  It follows that any cpc map $\phi:A\rightarrow B$ can be 
extended to a ucp map 
$\tilde{\phi} : \widetilde{A}\rightarrow B^\dagger$.$\hfill\square$
\end{notation}
Many of our results have a hypothesis of the form $d(A,B)<\gamma\leq K$ where $K$ is some absolute constant, for example $K=10^{-11}$ in Theorem \ref{thm10.4}. These constants are certainly not optimal, even for our methods, although we have attempted to make them as tight as possible. In the interests of readability, we have often rounded up multiple square roots to appropriate integer values. 

\section{Approximate averaging in nuclear ${\mathrm C}^*$-algebras}\label{Avg}

In this section we exploit the amenability of nuclear ${\mathrm C}^*$-algebras to establish point norm versions of the averaging results for injective von Neumann algebras in \cite{Christensen.Perturbations1}.  These techniques 
inevitably introduce small error terms not found in the von Neumann results since 
we are unable to take an ultraweak limit without possibly leaving the ${\mathrm C}^*$-algebra. We also use these methods to prove a  Kaplansky density result for approximate relative commutants for use in Section \ref{Unitary}. We begin with a review of the theory of amenability as it applies to ${\mathrm C}^*$-algebras.

Johnson defined amenability in the context of Banach algebras as a cohomological 
property \cite{Johnson.Memoir} and characterised amenable Banach algebras as 
those having a virtual diagonal \cite[Proposition 1.3]{Johnson.ApproxDiagonals}.  Given a Banach algebra $A$, let $A\ \potimes\ A$ be the projective tensor 
product.  A \emph{virtual diagonal} for $A$ is an element $\omega$ of $(A\ \potimes
\ A)^{**}$, with $a\omega=\omega a$  and $m^{**}(\omega)a=a$ for all $a\in A$, where $m:A\ \potimes\ A\rightarrow A$ is the multiplication map $x\otimes y\mapsto xy$.  When $A$ is a ${\mathrm C}^*$-algebra we can replace the latter condition by $m^{**}
(\omega)=1_{A^{**}}$.  As 
explained in \cite{Johnson.ApproxDiagonals}, virtual diagonals play the role in the 
context of Banach algebras that the invariant mean plays for amenable groups. 

Connes showed that any amenable ${\mathrm C}^*$-algebra is nuclear 
\cite{Connes.Cohomology} and Haagerup subsequently established the converse 
\cite{Haagerup.NuclearAmenable} using the non-commutative 
Grothendieck-Haagerup-Pisier inequality, 
\cite{Haagerup.Grothendieck,Pisier.Grothendieck}, to 
build on earlier partial results of Bunce and Paschke \cite{Bunce.Paschke}.  
Haagerup's proof not only obtains a virtual diagonal for a unital nuclear 
${\mathrm C}^*$-algebra but also shows that it can be taken in the 
ultraweakly closed convex hull of $\{x^*\otimes x:x\in A,\  \|x\|\leq 1\}$, \cite[Theorem 3.1]{Haagerup.NuclearAmenable}. 
This additional property will be of crucial importance subsequently.  In the unital 
setting, it is natural to ask whether the virtual diagonal can be chosen from the ultraweakly closed convex hull of $\{u^*\otimes u:u\in\mathcal U(A)\}$ as occurs for the group algebras ${\mathrm C}^*(G)$ of a discrete amenable group $G$.  This is not always 
the case; this property is equivalent to the concept of strong amenability for unital 
C$^*$-algebras \cite[Lemma 3.4]{Haagerup.NuclearAmenable}.  The Cuntz algebras 
$\mathcal O_n$ provide examples of nuclear ${\mathrm C}^*$-algebras which are not strongly 
amenable \cite{Rosenberg.AmenableCrossedProducts}.

An \emph{approximate diagonal} for a Banach algebra $A$ is a bounded net $(x_
\alpha)$ in $A\ \potimes\ A$ with $\|x_\alpha a-ax_\alpha\|_{A\potimes A}\rightarrow 
0$ and $\|m(x_\alpha)a-a\|_A\rightarrow 0$ for all $a\in A$. In the unital ${\mathrm C}^*$-algebra case we can replace the second condition 
by $\|m(x_\alpha)-1_{A}\|\rightarrow 0$. Approximate diagonals are closely 
connected to virtual diagonals since any ultraweak accumulation point of an approximate diagonal gives a virtual diagonal, while in \cite[Lemma 1.2]{Johnson.ApproxDiagonals} 
Johnson uses a Hahn-Banach argument to create an approximate diagonal from a 
virtual diagonal.  Just as a virtual diagonal plays the role of 
an invariant mean, the elements of an approximate diagonal play the role of 
F{\o}lner sets. To support this viewpoint, note that if $G$ is a countable discrete 
amenable group and $(F_n)_n$ an increasing exhaustive sequence of F{\o}lner 
sets, then $(\frac{1}{|F_n|}\sum_{g\in F_n}g^*\otimes g)_n$ is an approximate 
diagonal for ${\mathrm C}^*(G)$.  Combining Haagerup's work 
\cite{Haagerup.NuclearAmenable} on the location of a virtual diagonal in a nuclear 
${\mathrm C}^*$-algebra with \cite[Lemma 1.2]{Johnson.ApproxDiagonals} gives the following 
result.
\begin{lemma}\label{Avg.ApproxDiagonal}
Let $A$ be a unital nuclear ${\mathrm C}^*$-algebra.  Given a finite set $X\subseteq A$ and $
\varepsilon>0$, there exist contractions $\{a_i\}_{i=1}^n$ in $A$ and non-negative 
constants $\{\lambda_i\}_{i=1}^n$ summing to $1$ such that
\begin{equation}\label{Avg.Approx.Diagonal.1}
\left\|\sum_{i=1}^n\lambda_ia_i^*a_i-1_A\right\|<\varepsilon
\end{equation}
and
\begin{equation}
\left\|x\left(\sum_{i=1}^n\lambda_ia_i^*\otimes a_i\right)-\left(\sum_{i=1}^n
\lambda_ia_i^*\otimes a_i\right)x\right\|_{A\ \potimes\ A}<\varepsilon,\quad x\in X.
\end{equation}
If in addition $A$ is strongly amenable, then each of the $a_i$'s can be taken to be a 
unitary in $A$ and so (\ref{Avg.Approx.Diagonal.1}) can be replaced by $
\sum_{i=1}^n\lambda_ia_i^*a_i=1_A$.
\end{lemma}

Our next lemma is a point-norm version of Lemma 3.3 of \cite{Christensen.Perturbations1} for nuclear ${\mathrm C}^*$-algebras.
\begin{lemma}\label{Avg.1}
Let $A$ and $D$ be ${\mathrm C}^*$-algebras and suppose that $A$ is nuclear. Given a finite 
subset $X$ of the unit ball of $A$, $\varepsilon>0$, and $0<\mu<(25\sqrt{2})^{-1}$, 
there exists a finite subset $Y$ of the unit ball of $A$ (depending only on $X,A$, $
\varepsilon$ and $\mu$) with the following property:\ if $\phi :  A\to D$ is a $(Y,
\gamma)$-approximate $*$-homomorphism for some $\gamma\leq 1/17$, then 
there exists an $(X,\varepsilon)$-approximate $*$-homomorphism $\psi :   A\to 
D$ with $\|\phi-\psi\|\le 8\sqrt{2}\gamma^{1/2}+\mu$.
\end{lemma}

\begin{proof}
Let $\eta>0$ be such that if $0\le p_0\le 1$ is an operator with spectrum contained 
in $\Omega = [0,0.496] \cup [0.504,1]$ and $q$ is the spectral projection of $p_0$ 
for $[0.504,1]$, then the inequality $\|xp_0 - p_0x\|<\eta$ for a contractive operator 
$x$ implies that $\|xq-qx\|<\varepsilon$. The existence of $\eta$ (depending only on 
$\varepsilon$) follows from a standard functional calculus argument using uniform 
approximations of $\chi_{[0.504,1]}$ by polynomials on $\Omega$. See \cite[p. 332]
{Arveson.NotesExtensions}.

Nuclearity of $A$ implies nuclearity of the unital ${\mathrm C}^*$-algebra $\widetilde{A}$ 
obtained by adjoining a (possibly additional) unit (see Notation \ref{Prelim.Dagger}). 
By Lemma \ref{Avg.ApproxDiagonal} we may choose $\{\tilde a_i\}^n_{i=1}$ from 
the unit ball of $\widetilde{A}$ and non-negative constants $\{\lambda_i\}^n_{i=1}$ 
summing to 1 satisfying 
\begin{equation}\label{eq3.1}
 \left\|\sum^n_{i=1} \lambda_i \tilde a^*_i\tilde a_i-1_{\widetilde{A}}\right\| < (4\sqrt{2})^{-1}\mu
\end{equation}
and
\begin{equation}\label{eq3.2}
 \left\|x\left(\sum^n_{i=1} \lambda_i\,\tilde a^*_i\otimes \tilde a_i\right) - 
\left(\sum^n_{i=1} \lambda_i\,\tilde a^*_i \otimes \tilde a_i\right)x\right\|
_{\widetilde{A}\ \widehat\otimes\ \widetilde{A}} <\eta,\qquad x\in X.
\end{equation}
Each $\tilde a_i$ has the form $\alpha_i+z_i$ for $\alpha_i\in {\bb C}$ 
and $z_i\in A$. Projection to the quotient $\widetilde{A}/A$ shows that $|\alpha_i|\le 
1$ and so $\|z_i\|\le 2$. Define $a_i= z_i/2$, $1\le i\le n$, so that each $a_i$ lies in 
the unit ball of $A$ and $\tilde a_i = \alpha_i +2a_i$. We then take the set $Y$ to be 
$\{a_i,a_i^* :  1\leq i\leq n\}$, a finite subset of the unit ball of $A$. We now 
verify that $Y$ has the desired properties.

Consider a cpc map $\phi:A\rightarrow D$ satisfying the 
inequalities
\begin{equation}\label{eq3.3}
 \|\phi(a^*_ia_i) - \phi(a^*_i)\phi(a_i)\|\leq\gamma,\quad \|\phi(a_ia^*_i) - \phi(a_i)
\phi(a^*_i)\| \leq \gamma,\quad i=1,\dots,n,
\end{equation}
for some $\gamma\leq 1/17$. Let $\tilde{\phi}:\widetilde{A}\rightarrow D^\dagger$ 
be the canonical extension of $\phi$ to a ucp map. Assume that $D^\dagger$ is 
faithfully represented on a Hilbert space $H$ so that $I_H$ is the unit of $D^\dagger
$. Stinespring's representation theorem allows us to find a larger Hilbert space $K$ 
and a unital $*$-representation $\pi:\widetilde{A}\rightarrow {\mathbb B}(K)$ such 
that
\begin{equation}\label{eq3.4}
 \tilde\phi(x) = p\pi(x)p,\qquad x\in\widetilde A,
\end{equation}
where $p$ is the orthogonal projection of $K$ onto $H$. Define
\begin{equation}\label{eq3.5}
 p_0 = \sum^n_{i=1} \lambda_i\pi(\tilde a^*_i) p\pi(\tilde a_i).
\end{equation}
Then $0\le p_0\le 1$, and this operator lies in ${\mathrm C}^*(\pi(A),p)$ since this is an ideal in 
${\mathrm C}^*(\pi(\widetilde{A}),p)$ containing $p$. Using \eqref{eq3.1}, we obtain the 
estimate
\begin{align}
\|p_0-p\| &= \left\|\sum^n_{i=1} \lambda_i(\pi(\tilde a_i)^* p-p\pi(\tilde a^*_i)) \pi(\tilde 
a_i)
 + \sum^n_{i=1} \lambda_ip\pi(\tilde a^*_i\tilde a_i)-p\right\|\notag\\
&\le \sum^n_{i=1} \lambda_i\|\pi(\tilde a_i)^* p-p\pi(\tilde a_i)^*\| + \left\|p\pi\left(\sum^n_{i=1} \lambda_i \tilde a^*_i\tilde a_i-1_{\tilde A}\right)\right\|\notag\\
\label{eq3.6}
&\leq 2 \sum^n_{i=1} \lambda_i\|\pi(a_i)^* p-p\pi(a_i)^*\| + (4\sqrt{2})^{-1}\mu.
\end{align}

We are in position to closely follow \cite[Lemma 3.3]{Christensen.Perturbations1} for 
the remainder of the proof. Firstly
\begin{equation}\label{eq3.7}
 \|\pi(a_i^*)p-p\pi(a_i^*)\|^2 = \max\{\|p\pi(a_i)(1-p)\pi(a_i^*)p\|, \|p\pi(a^*_i)(1-p)
\pi(a_i)p\|\},
\end{equation}
from the estimate on \cite[p. 4]{Christensen.Perturbations1}. For $a\in A$, we have 
\begin{equation}
\phi(aa^*) - \phi(a)\phi(a^*) = p\pi(a) (1-p)\pi(a)^*p,
\end{equation} 
so taking $a=a_i$ and $a=a^*_i$ gives the estimate 
\begin{equation}\label{eq3.9}
 \|\pi(a_i^*)p-p\pi(a_i^*)\|^2\le \gamma,\qquad 1\le i\le n,
\end{equation}
from \eqref{eq3.3}. Using \eqref{eq3.9} at the end of \eqref{eq3.6} yields the 
inequality
\begin{equation}\label{eq3.10}
 \|p_0-p\|\le 2\gamma^{1/2}+(4\sqrt{2})^{-1}\mu<0.496,
\end{equation}
where the latter inequality follows from the bounds on $\gamma$ and $\mu$.
Define $\beta$ to be the constant $2\gamma^{1/2}+(4\sqrt{2})^{-1}\mu$. The spectrum of 
$p_0$ is contained in $[0,\beta]\cup [1-\beta,1]$. Letting $q$ denote the spectral 
projection of $p_0$ for $[1-\beta,1]$, (or equivalently $[0.504, 1]$), the functional 
calculus gives $\|p_0-q\| \le \beta$, and so $\|p-q\| \le 2\beta <1$ from 
\eqref{eq3.10}. Both $p$ and $q$ lie in ${\mathrm C}^*(\pi(\widetilde{A}),p)$, and so there is a 
unitary $w$ in this ${\mathrm C}^*$-algebra so that $wpw^*=q$ and $\|I_K-w\|\le 2\sqrt{2}\beta$, 
from Proposition \ref{Prelim.UnitaryEquiv} part 2.
Define a cpc map $\psi :   A\to {\mathbb B}(H)$ by
\begin{equation}\label{eq3.12}
 \psi(x) = pw^*\pi(x)wp,\qquad x\in A.
\end{equation}
Since $w^*\pi(A)w\subseteq {\mathrm C}^*(\pi(A),p)$, we observe that $\psi$ maps $A$ into 
$D$ because $p\pi(A)p$ is the range of $\phi:A\rightarrow D$. The estimate
\begin{equation}\label{eq3.13}
 \|\phi-\psi\|\le 2\|I_K-w\| \le 4\sqrt{2}\beta=8\sqrt{2}\gamma^{1/2}+\mu
\end{equation}
is immediate from \eqref{eq3.12}.

It remains to check that $\psi$ is an $(X,\varepsilon)$-approximate 
$*$-homomorphism. 
The definition of the projective tensor norm ensures that the map $y\otimes z
\mapsto \pi(y)p\pi(z)$ extends to a contraction from $\widetilde{A}
\ \widehat{\otimes}\ \widetilde{A}$ into $\mathbb B(K)$.   Applying this 
map to (\ref{eq3.2}) gives
\begin{equation}\label{eq3.14}
 \|\pi(x)p_0-p_0\pi(x)\|<\eta,\qquad x\in X,
\end{equation}
using the definition of $p_0$ from \eqref{eq3.5}. The choice of $\eta$ then gives
\begin{equation}\label{eq3.15}
 \|\pi(x)q-q\pi(x)\| <\varepsilon,\qquad x\in X.
\end{equation}
Thus, for $x\in X$,
\begin{equation}\label{eq3.16}
 \|\psi(xx^*)-\psi(x)\psi(x^*)\| = \|pw^*\pi(xx^*)wp-pw^*\pi(x)q\pi(x^*)wp\| <\varepsilon,
\end{equation}
using \eqref{eq3.15} and the relation $pw^*q=pw^*$. \end{proof}
\begin{remark}\label{rem3.4}
(i).~~If we choose to regard $\gamma$ as fixed in the previous lemma, then we 
could add the constraint $\mu \leq \nu\gamma^{1/2}$ (where $\nu$ is a fixed but 
arbitrary positive constant) to the hypotheses. This would change the concluding 
inequality to $\|\phi-\psi\|\leq (8\sqrt{2}+\nu)\gamma^{1/2}$, a form that is 
convenient for subsequent estimates.

\noindent (ii).~~With the hypothesis of Lemma \ref{Avg.1}, if $A$ is strongly 
amenable then so too is $\widetilde{A}$.  As such we can replace the estimate 
(\ref{eq3.1}) by the identity $\sum_{i=1}^n\lambda_i\tilde{a_i}^*
\tilde{a_i}=1_{\widetilde{A}}$.  It then follows that we can take $\mu=0$ and obtain 
an estimate $\|\phi-\psi\|\leq 8\sqrt{2}\gamma^{1/2}$. Without strong amenability, if 
we want $Y$ to be independent of $\gamma$ we are forced to introduce the 
constant $\mu$ upon which our estimates do not send $\|\phi-\psi\|$ to zero as $
\gamma\rightarrow 0$.$\hfill\square$ 
\end{remark}

We now turn to our second averaging lemma which is the analogue of Proposition 
4.2 of \cite{Christensen.Perturbations1}.  It is important to ensure that our 
point-norm version of this result handles not just $*$-homomorphisms but also 
$(Y,\delta)$-approximate $*$-homomorphisms for sufficiently large $Y$ and small $\delta$ as 
these are the outputs of Lemma \ref{Avg.1}.  Taking this into account gives the 
following lemma.
\begin{lemma}\label{Avg.2}
Let $A$ be a nuclear ${\mathrm C}^*$-algebra and let $D$ be a ${\mathrm C}^*$-algebra. Given a finite 
set $X$ in the unit ball of $A$ and $\varepsilon>0$, there exist a finite set $Y$ in the 
unit ball of $A$ and $\delta>0$, both depending only on $X,A$ and $\varepsilon$, 
with the following property.  If $\phi_1,\phi_2 :   A\to D$ are $(Y,\delta)$-
approximate $*$-homomorphisms such that $\phi_1\approx_{Y,\gamma} \phi_2$,
for some $\gamma\leq 13/150$, then there exists a unitary $u\in D^\dagger$ 
satisfying $\|u-I_{D^\dagger}\| < 2\sqrt{2}\gamma+5\sqrt{2}\delta$ and $\phi_1 \approx_{X,\varepsilon} \text{\rm Ad}(u)\circ \phi_2$.
\end{lemma}
\noindent Note that $\text{\rm Ad}(u)\circ\phi_2$ maps $A$ into $D$, since $D$ is 
an ideal in $D^\dagger$.
\begin{proof}
Fix a finite set $X$ in the unit ball of $A$ and an $\varepsilon>0$. By considering 
polynomial approximations to $t^{1/2}$ on $[0,1]$, we may find $\eta>0$, 
depending only on $\varepsilon$, so that the inequality $\|s^* sy-ys^*s\|<\eta$ for 
contractive operators $s$ and $y$ implies the relation $\|\,|s|y - y|s|\,\| < \varepsilon/
4 $ (see \cite[p. 332]{Arveson.NotesExtensions}).

Let $\delta$  satisfy
\begin{equation}\label{eq3.21a}
0< \delta\le \min\{\eta^2/100, \varepsilon^2/400, 1/200\}.
\end{equation}
By Lemma \ref{Avg.ApproxDiagonal}, we may find $\{\tilde a_i\}_{i=1}^n$ in the unit ball of $\widetilde A$ and non-negative constants $\{\lambda_i\}_{i=1}^n$ summing to 1, such that
\begin{equation}\label{eq3.22}
 \left\|\sum^n_{i=1} \lambda_i\tilde a^*_i\tilde a_i-1_{\widetilde{A}}\right\|<\delta,
\end{equation}
and
\begin{equation}\label{eq3.23}
\left\|\sum^n_{i=1} \lambda_i(x\tilde a^*_i \otimes \tilde a_i -\tilde a^*_i \otimes \tilde a_ix)\right\|_{\widetilde{A}\ \widehat\otimes\ \widetilde{A}} <
\delta^{1/2}, \qquad x\in X\cup X^*.
\end{equation}
As in Lemma \ref{Avg.1}, each $\tilde a_i$ can be written as $\alpha_i+2a_i$, where 
$\alpha_i$ is a constant and $a_i$ is in the unit ball of $A$. Let
$
 Y = \{a_1,\ldots, a_n\}.
$ We now show that $Y$ and $\delta$ have the desired property.

Suppose that $\phi_1,\phi_2:A\rightarrow D$ are $(Y,\delta)$-approximate 
$*$-homormorphisms satisfying $\phi_1\approx_{Y,\gamma}\phi_2$ for some $\gamma\leq 13/150$, and let $\tilde\phi_1,\tilde
\phi_2:\widetilde{A}\rightarrow D^\dagger$ be their unital completely positive 
extensions. Define
\begin{equation}
s=\sum_{i=1}^n\lambda_i\tilde{\phi_1}(\tilde{a_i}^*)\tilde{\phi_2}(\tilde{a_i}).
\end{equation}
Since $x\otimes y\mapsto \tilde\phi_1(x) \tilde\phi_2(y)$ is contractive on $
\widetilde{A}\ \potimes\ \widetilde{A}$, the inequality
\begin{equation}\label{eq3.24}
 \left\|\sum^n_{i=1} \lambda_i\tilde\phi_1(x\tilde a^*_i) \tilde\phi_2(\tilde a_i) - 
\sum^n_{i=1} \lambda_i \tilde\phi_1(\tilde a^*_i) \tilde \phi_2(\tilde a_ix)\right\|<
\delta^{1/2},\qquad x\in X\cup X^*,
\end{equation}
follows from \eqref{eq3.23}. Now $\tilde a_i = \alpha_i +2a_i$ and $\phi_1$ and $
\phi_2$ are $(Y,\delta)$-approximate $*$-homomorphisms so Proposition 
\ref{Prelim.ApproxhmEst} gives the inequalities
\begin{equation}\label{eq3.25}
 \|\tilde\phi_1(x\tilde a^*_i) - \tilde\phi_1(x) \tilde\phi_1(\tilde a^*_i)\|\leq 2\|{\phi_1}
(a_i^*a_i)-{\phi_1}(a_i^*){\phi_1}(a_i)\|^{1/2}<2\delta^{1/2},
  \end{equation}
and
\begin{equation}\label{Avg.2.eq1}
 \|\tilde\phi_2(\tilde a_ix) - \tilde\phi_2(\tilde{a_i}) \tilde\phi_2(x)\|\leq 2\|
\phi_2(a_ia_i^*)-\phi_2(a_i)\phi_2(a_i^*)\|<2\delta^{1/2},
 \end{equation}
for $x$ in the unit ball of $A$, and so in particular for $x\in X\cup X^*$, and $1\le i\le 
n$. Then \eqref{eq3.24}, \eqref{eq3.25} and \eqref{Avg.2.eq1} combine to yield
\begin{equation}\label{eq3.26}
 \|\phi_1(x)s-s\phi_2(x)\|< 5\delta^{1/2},\qquad x\in X\cup X^*.
\end{equation}
Taking adjoints gives
\begin{equation}\label{eq3.27}
\|s^*\phi_1(x)-\phi_2(x)s^*\|< 5\delta^{1/2},\qquad x\in X\cup X^*.
\end{equation}
Since
\begin{align}
s^*s\phi_2(x) - \phi_2(x)s^*s &= s^*\phi_1(x)s + s^*(s\phi_2(x)-\phi_1(x)s) - 
\phi_2(x)s^*s\notag\\
\label{eq3.28}
&= (s^*\phi_1(x) - \phi_2(x)s^*)s + s^*(s\phi_2(x)-\phi_1(x)s),
\end{align}
the inequality
\begin{equation}\label{eq3.29}
 \|s^*s\phi_2(x) - \phi_2(x)s^*s\| < 10\delta^{1/2},\qquad x\in X\cup X^*,
\end{equation}
follows from \eqref{eq3.26} and \eqref{eq3.27}. 

The choice of $\delta$ gives $10\delta^{1/2} \le \eta$ so, defining $z$ to be $|s|$,
\begin{equation}\label{eq3.30}
 \|z\phi_2(x) - \phi_2(x)z\| < \varepsilon/4,\qquad x\in X\cup X^*,
\end{equation}
by definition of $\eta$. Now
\begin{align}
 \|s-1_{D^\dagger}\| &= \left\|\sum^n_{i=1} \lambda_i\tilde\phi_1(\tilde a^*_i) \tilde\phi_2(\tilde a_i) 
-1_{D^\dagger}\right\|\notag\\
&\le \left\|\sum^n_{i=1} \lambda_i\tilde\phi_1(\tilde a^*_i) \tilde\phi_1(\tilde a_i) - 
1_{D^\dagger}\right\| + 2\gamma\notag\\
&\le \left\|\tilde\phi_1 \left(\sum^n_{i=1} \lambda_i \tilde a^*_i\tilde a_i -1_{\widetilde{A}}\right)\right\| 
+ \left\|\sum^n_{i=1} \lambda_i(\tilde\phi_1(\tilde a^*_i) \tilde\phi_1(\tilde a_i) - \tilde
\phi_1(\tilde a^*_i\tilde a_i))\right\| +2\gamma\notag\\
\label{eq3.31}
&\le \delta + 4\delta + 2\gamma = 5\delta +2\gamma,
\end{align}
using \eqref{eq3.22}, $\phi_1\approx_{Y,\gamma}\phi_2$, and that $\phi_1$ is a $(Y,\delta)$-approximate 
$*$-homomorphism. Since $5\delta + 2\gamma<1$, \eqref{eq3.31} gives invertibility 
of $s$, and the unitary in the polar decomposition $s=uz$ lies in $D^\dagger$ and satisfies 
$\|u-1_{D^\dagger}\|\leq 5\sqrt{2}\delta+2\sqrt{2}\gamma$ by part 1 of Proposition 
\ref{Prelim.UnitaryEquiv}. Then
\begin{align}
 \|z-1_{D^\dagger}\| &= \|u^*s-1_{D^\dagger}\| = \|s-u\|\notag\\
&\le \|s-1_{D^\dagger}\| + \|u-1_{D^\dagger}\|\notag\\
\label{eq3.32}
&\le  5(1+\sqrt{2})\delta+2(1+\sqrt{2})\gamma<1/2.
\end{align}
From this we obtain $\|z^{-1}\|\le 2$ so, for $x\in X$,
\begin{align}
 \|\phi_1(x) - u\phi_2(x)u^*\| &= \|\phi_1(x) u - u\phi_2(x)\|\notag\\
&\le \|\phi_1(x)uz - u\phi_2(x)z\| \|z^{-1}\|\notag\\
&\le 2\|\phi_1(x)s-s\phi_2(x)\| + 2\|uz\phi_2(x) - u\phi_2(x)z\|\notag\\
&\le 10\delta^{1/2} + 2\|z\phi_2(x) -\phi_2(x)z\|
< 10\delta^{1/2} + \varepsilon/2
\label{eq3.33}
\le \varepsilon,
\end{align}
using the definition of $\delta$ and equations \eqref{eq3.26} and \eqref{eq3.30}.
\end{proof}

\begin{remark}\label{Avg.2.Rem}
(i) \quad Having found one pair $(Y,\delta)$ for which  Lemma \ref{Avg.2} holds, it is 
clear that we may enlarge $Y$ and decrease $\delta$. Thus we may assume that 
$X\subseteq Y$, and we are at liberty to take $\delta$ as small as we wish.

\noindent (ii) \quad We have chosen to formulate Lemma \ref{Avg.2} so that $\delta
$ does not depend on $\gamma$. However, if we demand that $\gamma$ lies in the interval $[\gamma_0,13/150]$ for some $\gamma_0>0$, then we 
could add the extra constraint $\delta\leq (5\sqrt{2})^{-1}\nu\gamma_0$ to \eqref{eq3.21a}, 
where $\nu$ is a fixed but arbitrarily small positive number. This changes the 
estimate on $u$ from $\|1_{D^\dagger}-u\|<2\sqrt{2}\gamma+5\sqrt{2}\delta$ to $\|1_{D^\dagger}-u\|
<(2\sqrt{2}+\nu)\gamma$. In particular, we can take $\nu$ to be so small that  the 
estimate
$\|1_{D^\dagger}-u\|\leq 3\gamma$ holds. We will use this form of the lemma repeatedly in Section \ref{Unitary}, for a finite range of $\gamma$ simultaneously.

\noindent (iii)\quad If $A$, and hence $\widetilde{A}$ is strongly amenable, then we 
can replace the estimate \eqref{eq3.22} by $\sum_{i=1}^n\lambda_i\tilde
\phi_1(\tilde{a_i}^*)\tilde\phi_2(\tilde{a_i}^*)=1_{D^\dagger}$.  However this does not 
allow us to obtain a stronger version of the lemma for $(Y,\delta)$-approximate 
$*$-homomorphisms.  If we further insist that both $\phi_1$ and $\phi_2$ are  
$*$-homomorphisms (rather than just approximate ones), then we can take $\delta=0$ in Lemma \ref{Avg.2} in the strongly amenable case.$\hfill\square$
\end{remark}

Next we give two Kaplansky density style results for approximate 
relative commutants. Consider a unital ${\mathrm C}^*$-algebra $A$ acting non-degnerately 
on some Hilbert space $H$ and let $M=A''$. Given a finite dimensional ${\mathrm C}^*$-
subalgebra $F$ of $A$, one can average over the compact unitary group of $F$ to see that the relative comutant $F'\cap A$ is $^*$-strongly dense in the the relative commutant $F'\cap M$.   The next lemma is a version of this for approximate relative commutants, replacing averaging over a compact unitary group by an argument using approximate diagonals.

\begin{lemma}\label{Avg.Kap.Density}
Let $A$ be a unital ${\mathrm C}^*$-algebra faithfully and non-degenerately represented on a 
Hilbert space $H$ with strong closure $M=A''$. Let $X$ be a finite subset of the unit 
ball of $A$ which is contained within a nuclear ${\mathrm C}^*$-subalgebra $A_0$ of $A$ with 
$1_A\in A_0$. Given constants $\varepsilon,\mu>0$, there exists a finite set $Y$ in 
the unit ball of $A_0$ and $\delta>0$ with the following property.  Given a finite 
subset $S$ of the unit ball of $H$ and $m$ in the unit ball of $M$ such that
\begin{equation}\label{Avg.Kap.1}
\|my-ym\|\leq\delta,\quad y\in Y,
\end{equation}
there exists an element $a\in A$ with $\|a\|\leq\|m\|$ such that
\begin{equation}
\|ax-xa\|\leq\vp,\quad x\in X
\end{equation}
and
\begin{equation}
\|a\xi-m\xi\|<\mu,\ \|a^*\xi-m^*\xi\|<\mu \quad \xi\in S.
\end{equation}
If $m$ is self-adjoint, then $a$ can be taken self-adjoint.
\end{lemma}
\begin{proof}
Fix $0<\delta<\mu/4$.  By Lemma \ref{Avg.ApproxDiagonal} we can find positive 
scalars $\{\lambda_i\}_{i=1}^n$ summing to $1$ and elements $\{b_i\}_{i=1}^n$ in the 
unit ball of $A_0$ satisfying
\begin{equation}\label{Avg.Kap.2}
\left\|\sum_{i=1}^n\lambda_ib_i^*b_i-1_A\right\|\leq\delta
\end{equation}
and
\begin{equation}\label{Avg.Kap.3}
\left\|\sum_{i=1}^n\lambda_i(xb_i^*\otimes b_i-b_i^*\otimes b_ix)\right\|_{A_0\ 
\potimes\ A_0}< \vp,\quad x\in X.
\end{equation}
Now define $Y$ to be $\{b_1,\dots,b_n\}$. If $m\in M$ has $\|m\|\leq 1$ and 
satisfies (\ref{Avg.Kap.1}), then
\begin{align}
\left\|m-\sum_{i=1}^n\lambda_ib_i^*mb_i\right\|&\leq\left\|m-\sum_{i=1}^n
\lambda_ib_i^*b_im\right\|+\left\|\sum_{i=1}^n\lambda_ib_i^*(b_im-mb_i)\right\|
\nonumber\\
&\leq\delta+\delta=2\delta,\label{Avg.Kap.5}
\end{align}
from (\ref{Avg.Kap.1}) and (\ref{Avg.Kap.2}).

Given a finite set $S$ in the unit ball of $A$,  the Kaplansky density theorem gives $a_0\in A$ with $\|a_0\|\leq \|m\|$ 
satisfying
\begin{equation}\label{Avg.Kap.6}
\|(a_0-m)b_i\xi\|<\mu/2,\ \|(a_0-m)^*b_i\xi\|<\mu/2,\quad i=1,\dots,n,\quad \xi\in S.
\end{equation}
If $m=m^*$, then we can additionally insist that $a_0=a_0^*$. Define 
\begin{equation}\label{Avg.Kap.4}
a=\sum_{i=1}^n\lambda_ib_i^*a_0b_i
\end{equation}
which is self-adjoint if $m$, and hence $a_0$, is self-adjoint. Furthermore, $\|a\|\leq
\|a_0\|\leq\|m\|$. The definition of the projective tensor product norm shows that $y
\otimes z\mapsto ya_0z$ gives a contractive map of $A_0\ \potimes\ A_0$ into $A$.  
Using (\ref{Avg.Kap.3}) and (\ref{Avg.Kap.4}), we conclude that 
\begin{equation}
\|ax-xa\|<\vp,\quad x\in X.
\end{equation}
For $\xi\in S$,
\begin{align}
\|(a-m)\xi\|&\leq\left\|\big(\sum_{i=1}^n\lambda_ib_i^*(a_0-m)b_i\big)\xi\right\|+\left\|
\big(\sum_{i=1}^n\lambda_ib_imb_i^*-m\big)\xi\right\|\notag\\
&<\mu/2+2\delta,
\end{align}
from (\ref{Avg.Kap.5}) and (\ref{Avg.Kap.6}). The choice of $\delta$ gives $\|(a-m)
\xi\|<\mu$. Similarly, $\|(a-m)^*\xi\|<\mu$ for $\xi\in S$.
\end{proof}

In Section \ref{Unitary}, we will use the following version of Lemma \ref{Avg.Kap.Density} for unitaries. The hypothesis $\|u-I_H\|\le \alpha <2$ which we impose below is to ensure a gap in the spectrum, permitting us to take a continuous logarithm. This is essential for our methods.

\begin{lemma}\label{lem3.9}
Let $A$ be a unital ${\mathrm C}^*$-algebra faithfully and non-degenerately represented on a 
Hilbert space $H$ with strong closure $M=A''$. Let $X$ be a finite subset of the unit 
ball of $A$ which is contained within a nuclear ${\mathrm C}^*$-algebra $A_0$ of $A$ with 
$1_A\in A_0$.  Given constants $\vp_0,\mu_0>0$ and $0<\alpha<2$, there exists a 
finite set $Y$ in the unit ball of $A_0$ and $\delta_0>0$ with the following property. 
Given a finite set $S$ in the unit ball of $H$ and a unitary $u\in M$ satisfying $\|u-I_H\|
\le \alpha$ and
\begin{equation}\label{eq100.1}
 \|uy-yu\|\leq \delta_0,\qquad y\in Y,
\end{equation}
there exists a unitary $v\in A$ such that
\begin{gather}\label{eq100.2}
 \|vx-xv\|\leq \vp_0,\qquad x\in X,\\
\label{eq100.3}
\|v\xi-u\xi\|<\mu_0,\quad \|v^*\xi-u^*\xi\|<\mu_0,\qquad \xi\in S,
\end{gather}
and $\|v-I_H\|\le \|u-I_H\|\le\alpha$.
\end{lemma}

\begin{proof}
Fix $0<\alpha<2$. The result is obtained from Lemma \ref{Avg.Kap.Density} using some polynomial approximations and the following two observations.

 Firstly, there is an interval $[-c,c]$ with $0<c<\pi$ so that $|1-e^{i\theta}|
\le \alpha$ if and only if $\theta$ lies in this interval modulo $2\pi$. Any unitary $u$ 
satisfying $\|u-1\|\le\alpha$ has spectrum contained in the arc $\{e^{i\theta} \colon \ -
c\le \theta\le c\}$ on which there is a continuous logarithm. By approximating $\log z
$ on the arc by polynomials in the complex variables $z$ and $\bar z$, we obtain 
the following. Given $\delta>0$, there exists $\delta_0>0$ so that if $u\in M$ is a 
unitary satisfying $\|u-1\|\le \alpha$ and
\begin{equation}\label{eq100.4}
 \|uy-yu\| \leq \delta_0
\end{equation}
for some $y\in\mathbb B(H)$ with $\|y\|\leq 1$, then
\begin{equation}\label{eq100.5}
\Big\|\frac{\log u}{\pi}\,y - y\frac{\log u}{\pi}\Big\|\leq \delta,
\end{equation}
as in  \cite[p. 332]{Arveson.NotesExtensions}. Note that this deduction also requires \eqref{eq100.4} to hold with $u$ replaced by $u^*$, but this is immediate from the algebraic identity $u^*y-yu^*=u^*(yu-uy)u^*$.

Secondly, as we show below, the map $x\mapsto e^{i\pi 
x}$ is uniformly strong-operator continuous on the unit ball in $M_{\text{s.a.}}$ in 
the following sense.  Given $\mu_0>0$, there exists $\mu>0$ with the following property. For each finite subset $S$ of the unit ball of $H$ and $h\in M_{\text{s.a.}}$ with $\|h\|
\leq 1$, there exists a finite subset $S'$ of the unit ball of $H$ (depending only on $S$, $\mu_0$ and $h$) such that the inequalities 
\begin{equation}\label{Avg.Kap.Unitary.3}
\|(e^{i\pi h}-e^{i\pi k})\xi\|<\mu_0,\quad \|(e^{i\pi h}-e^{i\pi k})^*\xi\|<\mu_0,\quad \xi
\in S,
\end{equation}
are valid whenever an element $k$ in the unit ball of 
$
M_{\text{s.a.}}$ satisfies
\begin{equation}\label{Avg.Kap.Unitary.2}
\|(h-k)\xi\|<\mu,\quad \xi\in S'.
\end{equation}
This follows by considering polynomial approximations of $e^{i\pi t}$ for $t\in 
[-1,1]$. Given $\mu_0>0$, let $p(t)=\sum_{j=0}^r\lambda_jt^j$ be a polynomial in $t$ 
with $\sup_{-1\leq t\leq 1}|p(t)-e^{i\pi t}|<\mu_0/3$ and define
\begin{equation}
\mu=\frac{\mu_0}{3r\sum_{j=0}^r|\lambda_j|}.
\end{equation}
Given $h$ and $S$, define $S'=\{h^j\xi:\xi\in S,\ 0\leq j< r\}$ and suppose that 
(\ref{Avg.Kap.Unitary.2}) holds. For $\xi\in S$ and $0\leq j\leq r$, we compute
\begin{align}
\|(h^j-k^j)\xi\|&\leq \|(h^{j-1}-k^{j-1})h\xi\|+\|k^{j-1}(h-k)\xi\|\notag\\
&\leq\|(h^{j-2}-k^{j-2})h^2\xi\|+\|k^{j-2}(h-k)h\xi\|+\|k^{j-1}(h-k)\xi\|\notag\\
&\leq\dots\notag\\
&\leq\sum_{m=0}^{j-1}\|(h-k)h^{m}\xi\|<r\mu,
\end{align}
so that
\begin{equation}
\|(p(h)-p(k))\xi\|\leq \sum_{j=0}^r|\lambda_j|\|(h^j-k^j)\xi\|<\sum_{j=0}^r|\lambda_j|r
\mu=\mu_0/3,\quad \xi\in S,
\end{equation}
and, similarly,
\begin{equation}
\|(p(h)-p(k))^*\xi\|\leq\mu_0/3,\quad \xi\in S.
\end{equation}
The estimates in (\ref{Avg.Kap.Unitary.3}) follow.

We can now deduce the lemma from Lemma \ref{Avg.Kap.Density}. Assume then that $X,\vp_0$, and $\mu_0$ are given and let $A_0$ 
be a nuclear ${\mathrm C}^*$-subalgebra of $A$ containing $X$ and $1_A$. By means of another polynomial approximation argument choose $
\vp>0$ so that if $k\in M_{\text{s.a.}}$, $\|k\|\le 1$, and
\begin{equation}\label{eq100.8}
 \|xk-kx\|\leq\vp,\qquad x\in X,
\end{equation}
then
\begin{equation}\label{eq100.9}
 \|xe^{i\pi k}-e^{i\pi k}x\|\leq\vp_0, \qquad x\in X,
\end{equation}
as in \cite[p. 332]{Arveson.NotesExtensions}. Let $\mu>0$ be the constant corresponding via the second observation above to $\mu_0$ 
and apply Lemma \ref{Avg.Kap.Density} to $(X,\vp,\mu)$, producing a finite set $Y$ 
in the unit ball of $A$ and a constant $\delta>0$.  Let $\delta_0$ be the constant  
corresponding to $\delta$ given by our first observation so that (\ref{eq100.4}) 
implies (\ref{eq100.5}).

Suppose that we have a unitary $u\in M$ satisfying $\|u-I_H\|\leq\alpha$ and 
\begin{equation}
\|uy-yu\|\leq\delta_0,\qquad y\in Y.
\end{equation}
Define $h=\,-i\log u/\pi$ in $M_{\text{s.a.}}$. The definition of $\delta_0$ gives
\begin{equation}
\|hy-yh\|\leq\delta,\qquad y\in Y.
\end{equation}
Given a finite set $S$ in the unit ball of $H$, let $S'$ be the finite set corresponding 
to $h$, $S$ and $\mu_0$ from the uniform strong-operator continuity of $x\mapsto 
e^{i\pi x}$. Putting this set into Lemma \ref{Avg.Kap.Density}, we can find a 
self-adjoint operator $k$ in the unit ball of $A$ with $\|k\|\leq\|h\|$ which satisfies 
 (\ref{Avg.Kap.Unitary.2}) and (\ref{eq100.8}). Let $v=e^{i\pi k}$, which has $\|v-I_H\|
\leq \|u-I_H\|\leq \alpha$. By definition of $\vp$, (\ref{eq100.9}) holds and this gives \eqref{eq100.2}.
Similarly the definition of $\mu$ ensures that (\ref{Avg.Kap.Unitary.2}) implies 
(\ref{Avg.Kap.Unitary.3}) so that (\ref{eq100.3}) also holds. 
\end{proof}

\begin{remark}
In \cite[Section 4]{KOS} it is shown that Haagerup's proof that nuclear ${\mathrm C}^*$-algebras are amenable also implies that every ${\mathrm C}^*$-algebra $A$ has the following property: given an irreducible representation $\pi$ of $A$ on $H$, a finite subset $X$ of $A$, a finite subset $S$ of $H$ and $\varepsilon>0$, there exist $\{a_i\}_{i=1}^n$ in $A$ such that
\begin{enumerate}
\item $\|\sum_{i=1}^na_ia_i^*\|\leq 1$;
\item $\sum_{i=1}^n\pi(a_ia_i^*)\xi=\xi$ for $\xi\in S$;
\item $\|[\sum_{i=1}^na_iya_i^*,x]\|\leq\varepsilon$ for all $x\in X$ and all $y$ in the unit ball of $A$. 
\end{enumerate}
Although we do not need it in this paper,  using the property above in place of nuclearity in the proof of Lemma \ref{Avg.Kap.Density}, gives the result below. Comparing this with Lemma \ref{Avg.Kap.Density}, the key difference is that the set $S$ below must be specified with $X$ and $\varepsilon$, whereas in Lemma \ref{Avg.Kap.Density} the set $S$ can be chosen after $m$ is specified. When we use these results in Section \ref{Unitary}, being able to choose $S$ at this late stage is important to our argument.
\end{remark}
\begin{proposition}
Let $A$ be ${\mathrm C}^*$-algebra represented irreducibly on a Hilbert space $H$. Let $X$ be a finite subset of the unit ball of $A$, $S$ be a finite subset of the unit ball of $H$ and $\varepsilon>0$. Then there exists a finite subset $Y$ of the unit ball of $A$ and $\delta>0$ with the following property. Given any $m\in\mathbb B(H)$ with $\|my-ym\|\leq\delta$ for $y\in Y$, there exists $a\in A$ with $\|a\|\leq \|m\|$, $\|ax-xa\|\leq\varepsilon$ for $x\in X$ and $\|(m-a)\xi\|\leq\varepsilon$ for $\xi\in S$.
\end{proposition}

\section{Approximation on finite sets and isomorphisms}\label{Close}

In this section we establish the qualitative version of Theorem \ref{Intro.Iso}: that 
${\mathrm C}^*$-algebras sufficiently close to a separable nuclear ${\mathrm C}^*$-algebra $A$ must be 
isomorphic to $A$.  To do this, we use an approximation approach inspired by the 
intertwining arguments of \cite[Theorem 6.1]{Christensen.NearInclusions} and those 
in the classification programme
(see, for example, \cite{Elliott.RealRankZero}).  This is presented in Lemma 
\ref{Close.Iso.Tech}, where we have given a general formulation in terms of the 
existence of certain completely positive contractions. This is designed for 
 application in several contexts where this hypothesis can be verified, 
and so it forms the basis for all our subsequent near inclusion results as well as Theorem \ref{Intro.Iso}.

\begin{lemma}\label{Close.Iso.Tech}
Suppose that $A$ and $B$ are ${\mathrm C}^*$-algebras on some Hilbert space $H$ and that 
$A$ is separable and nuclear.  Suppose that there exists a constant $\eta>0$ 
satisfying
$\eta<1/210000$ such that for each finite subset $Z$ of the unit ball of $A$, there is 
a completely positive contraction $\phi:A\rightarrow B$ satisfying $\phi\approx_{Z,
\eta}\iota$.  Then, given any finite subset $X_A$ of the unit ball of $A$ and $0<
\mu<1/2000$,  there exists an injective $*$-homomorphism $\alpha:A\rightarrow B$ 
with $\alpha\approx_{X_A,8\sqrt{6}\eta^{1/2}+\eta+\mu}\iota$.  If, in addition,  $B\subset_{1/5}A$, then we can take $\alpha$ to be surjective.
\end{lemma}

\begin{proof}
Let $\{a_n\}_{n=0}^{\infty}$ be a dense sequence in the unit
ball of $A$, where $a_0=0$. Fix a finite subset $X_A$ of the
unit ball of $A$. Given $\mu<1/2000$, define $\nu=2\mu<1/4000$. We will 
construct inductively sequences
$\{X_n\}_{n=0}^{\infty}$, $\{Y_n\}_{n=0}^{\infty}$ of finite
subsets of the unit ball of $A$, a sequence
$\{\delta_n\}_{n=0}^{\infty}$ of positive constants, a
sequence $\{\theta_n:A\rightarrow B\}_{n=0}^{\infty}$ of
completely positive contractions, and a sequence of
unitaries $\{u_n\}_{n=1}^{\infty}$ in $B^\dagger$ which
satisfy the following conditions.
\begin{enumerate}[(a)]
\item The sets $\{X_n\}_{n=0}^{\infty}$ are increasing,
$a_n\in X_n$ for $n\geq 0$, and $X_A\subseteq X_1$.
\item $\delta_n\leq 2^{-n}$ for $n\geq 0$, and given any two
$(Y_n,\delta_n)$-approximate
$*$-homomorphisms $\phi_1,\phi_2  :  A\to B$ satisfying
$\phi_1\approx_{Y_n,2(8\sqrt{6}\eta^{1/2}+\eta+\nu)} \phi_2$, there exists a
unitary $u\in B^\dagger$ with ${\text{Ad}}(u)\circ
\phi_1\approx_{X_n, 2^{-n}\nu}\phi_2$. This unitary can be
chosen to satisfy $\|u-1_{B^\dagger}\|\leq 4\sqrt{2}(8\sqrt{6}\eta^{1/2}+\eta +\nu)+\nu$.
\item $X_n \subseteq Y_n$ and $\theta_n$ is a
$(Y_n,\delta_n)$-approximate $*$-homomorphism for $n\geq 0$.
\item $\theta_n \approx_{Y_n,8\sqrt{6}\eta^{1/2}+\eta+\nu} \iota$ for $n\geq
0$.
\item
${\text{Ad}}(u_n)\circ\theta_n\approx_{X_n,2^{-(n-1)}\nu}\theta_{n-1}$
and $\|u_n-1_{B^\dagger}\|\leq  4\sqrt{2}(8\sqrt{6}\eta^{1/2}+\eta+\nu)+\nu$ for $n\geq 1$.  When $n=1$, we  take $u_1=1$.
\end{enumerate}
If  $B\subset_{1/5}A$, then the separability of $A$ ensures the separability of $B$ by Proposition \ref{Prelim.SeparableOpen}. In this case fix  a 
dense sequence $\{b_n\}_{n=0}^\infty$   in the unit ball of $B$ with $b_0=0$,  and in this case we shall 
 require the following extra condition for our induction: 
\begin{enumerate}[(a)]
\setcounter{enumi}{5}
\item $d(u_{n-1}^*\dots u_1^*b_iu_1\dots u_{n-1},X_n)\leq 2/5$ for $0\leq i\leq n$.
\end{enumerate}

Assuming for the moment that the induction has been
accomplished, we first show that conditions (a)-(e) allow us to construct the 
embedding $\alpha:A\hookrightarrow B$. Define $\alpha_n={\text{Ad}}(u_1\ldots 
u_n)\circ \theta_n$
for $n\geq 1$ so $\alpha_1=\theta_1$ since $u_1=1$. For a fixed integer $k$ and an element $x\in
X_k$, we have
\begin{align}
\|\alpha_{n+1}(x)-\alpha_n(x)\|&=\|{\text{Ad}}(u_1\ldots
u_n)\circ
{\text{Ad}}(u_{n+1})\circ\theta_{n+1}(x)-{\text{Ad}}(u_1\ldots
u_n)\circ\theta_n(x)\| \notag\\
&=
\|{\text{Ad}}(u_{n+1})\circ\theta_{n+1}(x)-\theta_n(x)\|\notag
\\
&\leq 2^{-n}\nu,\quad n\geq k, \label{eq500.1}
\end{align}
using (e). Density of the $X_k$'s in the unit ball of $A$ then shows that  the 
sequence $\{\alpha_n\}_{n=1}^{\infty}$
converges in the point norm topology to a completely
positive contraction $\alpha:A\rightarrow B$.
Condition (c) implies that $\alpha$ is a
$*$-homomorphism since each $\alpha_n$ is a
$(Y_n,\delta_n)$-approximate $*$-homomorphism,
$\displaystyle{\lim_{n\to\infty}\delta_n=0}$, and $\cup_{n=0}^{\infty}Y_n$
is dense in the unit ball of $A$. For each $x\in X_n$, it
follows from \eqref{eq500.1} that
\begin{align}
\|\alpha_n(x)-\alpha (x)\|&\leq \sum_{m=n}^{\infty}
\|\alpha_{m+1}(x)-\alpha_m(x)\|\notag\\
&\leq \sum_{m=n}^{\infty} 2^{-m}\nu = 2^{-(n-1)}\nu,
\label{eq500.2}
\end{align}
and, in particular, that
\begin{equation}\label{eq500.3}
\|\alpha(x)-\theta_1(x)\|=\|\alpha (x)-\alpha_1(x)\|\leq \nu,\quad x\in
X_A\subseteq X_1.
\end{equation}
Thus, from \eqref{eq500.2} and (d),
\begin{align}
\|\alpha(x)\| &\geq \|\alpha_n(x)\|-2^{-n}\nu\notag\\
&= \|\theta_n(x)\|-2^{-n}\nu\notag\\
&\geq \|x\|-(8\sqrt{6}\eta^{1/2}+\eta+\nu)-2^{-n}\nu,\quad x\in
X_n.\label{eq500.4}
\end{align}
Letting $n\to \infty$ in \eqref{eq500.4}, it follows that
\begin{equation}\label{eq500.5}
\|\alpha(x)\|\geq \|x\|-(8\sqrt{6}\eta^{1/2}+\eta+\nu),
\end{equation}
for any $x$ in the unit ball of $A$, using the collective density of the $X_n$'s. This 
shows that $\alpha$
is injective since $\|\alpha(x)\|\geq 1-(8\sqrt{6}\eta^{1/2}+\eta+\nu)>0$ for $x$ 
in the unit sphere of $A$. Thus
$\alpha$ is a $*$-isomorphism of $A$ into $B$.
Then, for $x\in X_A\subseteq X_1$, the estimate
\begin{align}
\|\alpha(x)-x\|&\leq
\|\alpha(x)-\theta_1(x)\|
+\|\theta_1(x)-x\|\notag\\
&\leq \nu+8\sqrt{6}\eta^{1/2}+\eta+\nu=8\sqrt{6}\eta^{1/2}+\eta+\mu\label{eq500.6}
\end{align}
follows from \eqref{eq500.3} and hypothesis
(d). 

In the case that  $B\subset_{1/5}A$, we shall show that the 
additional assumption (f) gives $B\subset_1{\alpha}(A)$ and so ${\alpha}
$ is surjective by Proposition \ref{EasyFact}. Indeed, given 
$0\leq i\leq n$, find $x\in X_n$ with $\|x-u_{n-1}^*\dots u_1^*b_iu_1\dots u_{n-1}\|
\leq 2/5$. Then
\begin{align}
\|\alpha_n(x)-b_i\|&=\|u_1\dots u_n\theta_n(x)u_n^*\dots u_1^*-b_i\|\notag\\
&\leq\|u_n\theta_n(x)u_n^*-x\|+\|x-u_{n-1}^*\dots u_1^*b_iu_1\dots u_{n-1}\|\notag\\
&\leq 2\|u_n-1_{B^\dagger}\|+\|\theta_n(x)-x\|+\|x-u_{n-1}^*\dots u_1^*b_iu_1\dots u_{n-1}\|
\notag\\
&\leq 8\sqrt{2}(8\sqrt{6}\eta^{1/2}+\eta+\nu)+2\nu+(8\sqrt{6}\eta^{1/2}+\eta+\nu)
+2/5\notag\\
&\leq 0.94<1,
\end{align}
using (d), (e), and the upper bounds $\eta<1/210000$ and $\nu<1/4000$. The claim follows 
from the density of $\{b_i\}_{i=0}^\infty$ in the unit ball of $B$. It remains to 
complete the inductive construction.

We begin the induction trivially, setting $X_0=Y_0=\{a_0\}=\{0\}$, $u_0=1$, $
\delta_0=1$ and taking any completely positive contraction $\theta_0:A\rightarrow B
$. Suppose that the construction is complete up to the $n$-th stage.  Define $X_{n
+1}=X_A\cup X_n\cup\{a_{n+1}\}\cup Y_n$ so that condition (a) holds.  When $B\subset_{1/5} A$, we will have the same near containment 
for their unit balls with 2/5 replacing 1/5 (see Definition \ref{Prelim.DefNear} and 
Remark \ref{Prelim.Metric.Rem}). Thus we can extend $X_{n+1}$ so that condition 
(f) holds. Now use Lemma \ref{Avg.2} and Remark \ref{Avg.2.Rem} (i) to find a finite set $Y_{n+1}$ in the unit ball of 
$A$, containing $X_{n+1}$, and $0<\delta_{n+1}<\min\{\delta_n,2^{-(n+1)},(5\sqrt{2})^{-1}\nu
\}$ so that condition (b) holds.  This is possible because direct calculation 
shows that $2(8\sqrt{6}\eta^{1/2}+\eta+\nu)<13/150$ from the upper bounds on $
\eta$ and $\nu$, and the estimate $\|u-1_{B^\dagger}\|\leq 4\sqrt{2}(8\sqrt{6}\eta^{1/2}+\eta+\nu)
+\nu$ follows from $\delta_n\leq (5\sqrt{2})^{-1}\nu$. Now use Lemma \ref{Avg.1}  to find a finite set $Z\supset Y_{n+1}$ in the unit ball of $A$ so 
that given a $(Z,3\eta)$-approximate $*$-homomorphism $\phi:A\rightarrow B$, we 
can find a $(Y_{n+1},\delta_{n+1})$-approximate $*$-homomorphism $\psi:A
\rightarrow B$ with 
\begin{equation}
\|\phi-\psi\|\leq 8\sqrt{2}(3\eta)^{1/2}+\nu= 8\sqrt{6}\eta^{1/2}+
\nu.
\end{equation}
 This is possible because $3\eta<1/17$. Let 
\begin{equation}
Z'=Z\cup Z^*\cup\{zz^*:z\in Z\cup Z^*\}.
\end{equation}
From  the hypothesis, there is a completely positive contraction $\phi:A\rightarrow B
$ with $\phi\approx_{Z',\eta}\iota$, whereupon $\phi$ is a $(Z,3\eta)$-approximate $*$-homomorphism.  The definition of $Z$ then gives us a $(Y_{n+1},\delta_{n+1})$-approximate $*$-homomorphism $\theta_{n+1}:A\rightarrow B$ with $\|\phi-
\theta_{n+1}\|\leq 8\sqrt{6}\eta^{1/2}+\nu$ verifying (c). It follows that $\theta_{n+1}\approx_{Y_{n+1},
8\sqrt{6}\eta^{1/2}+\eta+\nu}\iota$ and so condition (d) holds.  Since $\theta_n$ and $
\theta_{n+1}$ are $(Y_n,\delta_n)$-approximate $*$-homomorphisms which satisfy 
$\theta_n\approx_{Y_n,2(8\sqrt{6}\eta^{1/2}+\eta+\nu)}\theta_{n+1}$, there exists a 
unitary $u_{n+1}$ in $B^\dagger$ with $\ad(u_{n+1})\circ\theta_{n+1}\approx_{X_n,
2^{-n}\nu}\theta_n$ and $\|u_{n+1}-1_{B^\dagger}\|\leq 4\sqrt{2}(8\sqrt{6}\eta^{1/2}+\eta+\nu)+
\nu$ from the inductive version of condition (b). This gives condition (e).  This last step is not required when $n=0$, 
as $X_0=\{0\}$ so we can take $u_1=1$ in this case.
\end{proof}

Using Proposition \ref{Prelim.PtArveson}, it follows that sufficiently close separable 
and nuclear ${\mathrm C}^*$-algebras are isomorphic. In fact we only need one near 
inclusion to be relatively small. 
\begin{theorem}\label{thm4.2}
Suppose that $A$ and $B$ are separable nuclear ${\mathrm C}^*$-algebras on some Hilbert 
space $H$ with $A\subset_\gamma B$ and $B\subset_\delta A$ for 
\begin{equation}
\gamma\leq 1/420000,\quad \delta\leq 1/5.
\end{equation}
Then $A$ and $B$ are isomorphic. 
\end{theorem}
\begin{proof}
Take $0<\gamma'<\gamma$ so that $A\subset_{\gamma'} B$.  Then Proposition 
\ref{Prelim.PtArveson} provides the cpc maps $A\rightarrow B$ required to use 
Lemma \ref{Close.Iso.Tech} when $\eta=2\gamma'$ and so the result follows.
\end{proof}

The qualitative version of Theorem \ref{Intro.Iso} also follows as algebras close to a 
separable and nuclear ${\mathrm C}^*$-algebra must again be separable and nuclear.  While 
the examples of \cite{Johnson.PerturbationExample} show that it is not possible in 
general to obtain isomorphisms which are uniformly close to the identity between 
close separable nuclear ${\mathrm C}^*$-algebras, Lemma \ref{Close.Iso.Tech} does enable 
us to control the behaviour of our isomorphisms on finite subsets of the unit ball.  

\begin{theorem}\label{Close.Iso.Main}
Let $A$ and $B$ be ${\mathrm C}^*$-algebras on some Hilbert space with 
\begin{equation}
d(A,B)<\gamma<1/420000.
\end{equation}
If $A$ is nuclear and separable, then $A$ and $B$ are isomorphic. Furthermore, 
given finite sets $X$ and $Y$ in the unit balls of $A$ and $B$ respectively,  there 
exists a surjective $*$-isomorphism $\theta:A\rightarrow B$ with
\begin{equation}
\|\theta(x)-x\|,\ \|\theta^{-1}(y)-y\|\leq 28\gamma^{1/2},\quad x\in X,\ y\in Y.
\end{equation}
\end{theorem}
\begin{proof}
As $d(A,B)<1/101$, Proposition \ref{Prelim.NuclearOpen}  ensures that $B$ is  nuclear. Choose 
$\gamma'$ so that $d(A,B)<\gamma'<\gamma$ and 
enlarge the set $X$ if necessary so that $Y\subset_{\gamma'} X$.   Proposition 
\ref{Prelim.PtArveson} shows that the hypotheses of Lemma \ref{Close.Iso.Tech} 
hold for $\eta=2\gamma'$ so $A$ and $B$ are isomorphic.  Furthermore, for a 
constant $0<\mu<1/2000$,  Lemma \ref{Close.Iso.Tech} provides a surjective $*$-isomorphism 
$\theta:A\rightarrow B$ with \begin{equation}
\|\theta(x)-x\|<8\sqrt{6}(2\gamma')^{1/2}+2\gamma'+\mu,\quad x\in X.
\end{equation}
For each $y\in Y$, there exists $x\in X$ with $\|x-y\|\leq \gamma'$. As $\|\theta^{-1}
(y)-x\|=\|y-\theta(x)\|$ it follows that
\begin{equation}
\|\theta^{-1}(y)-y\|\leq 2\|x-y\|+\|\theta(x)-y\|<8\sqrt{6}(2\gamma')^{1/2}+4\gamma'+
\mu.
\end{equation}
In these  inequalities, a suitably small choice of $\mu$ and the inequality $\gamma'<
\gamma<1/420000$ lead to upper estimates of $28\gamma^{1/2}$ in both cases. 
\end{proof}

Lemma \ref{Close.Iso.Tech} is also the technical tool behind our three near 
inclusion results.  We can deduce two of these results now, using Propositions 
\ref{Prelim.PtArveson} and \ref{InnerFlip} to obtain the cpc maps required to use 
Lemma \ref{Close.Iso.Tech}.  Our third, and most general result, which only 
assumes that $A$ has finite nuclear dimension is postponed until Section \ref{Near}.

\begin{corollary}\label{Close.Iso.NearNuclear}
Let $\gamma$ be a constant satisfying $0<\gamma< 1/420000$, and consider a 
near inclusion  $A\subset_\gamma B$  of ${\mathrm C}^*$-algebras on a Hilbert space $H$, 
where $A$ and $B$ are nuclear and $A$ is separable.  Then $A$ embeds into $B$ 
and, moreover, for each finite subset $X$ of the unit ball of $A$ there exists an 
injective $*$-homomorphism $\theta :  A\to B$ with $\theta\approx_{X,
28\gamma^{1/2}}\iota$.
\end{corollary}
\begin{proof}
Proposition \ref{Prelim.PtArveson} shows that $A$ and $B$ satisfy the hypotheses 
of Lemma \ref{Close.Iso.Tech} when $\eta=2\gamma$. The number $28$ appears 
from the inequality
\begin{equation}
8\sqrt{6}(2\gamma)^{1/2}+2\gamma+\mu<28\gamma^{1/2},
\end{equation}
which is valid for a sufficiently small choice of $\mu$.
\end{proof}

\begin{corollary}\label{half-flip-close}
Let $\gamma$ be a constant satisfying $0<\gamma<1/12600000\approx 7.9\times 10^{-8}$. Consider a near 
inclusion $A\subset_\gamma B$ where $A$ is unital, separable and has  
approximately inner half flip. Then $A$ embeds into $B$ and,  moreover, for each 
finite subset $X$ of the unit ball of $A$ there exists an injective $*$-homomorphism 
$\theta :  A\to B$ with $\theta\approx_{X,152\gamma^{1/2}}\iota$. 
\end{corollary}
\begin{proof}
Recall that the proof of Proposition 2.8 of \cite{Effros.ApproxInnerFlip} shows that 
$A$ is nuclear. Write $\alpha=(4\sqrt{2}+1)\gamma+4\sqrt{2}\gamma^2$.  
Proposition \ref{InnerFlip} shows that the hypotheses of Lemma 
\ref{Close.Iso.Tech} hold for $\eta=8\alpha+4\alpha^2+4\sqrt{2}\gamma$ and the bound on $\gamma
$ is chosen so that $\eta <60\gamma<1/210000$.   The number $152$ appears 
from the inequality
\begin{equation}
8\sqrt{6}\eta^{1/2}+\eta+\mu<152\gamma^{1/2},
\end{equation}
which is valid for a sufficiently small choice of $\mu$.
\end{proof}

We briefly examine stability under tensoring by the strongly self-absorbing algebras introduced in \cite{Toms.StrongSelfAbsorbing}.  Recall that a separable unital ${\mathrm C}^*$-algebra $D$ is \emph{strongly self-absorbing} if there is an isomorphism between $D$ and $D\otimes D$ which is approximately unitarily equivalent to the embedding $D\hookrightarrow D\otimes D,\ x\mapsto x\otimes 1_D$.  Such algebras automatically 
have  approximately inner half flip \cite[Proposition 1.5]
{Toms.StrongSelfAbsorbing} and so are simple and nuclear. A separable ${\mathrm C}^*$-algebra $A$ is \emph{$D$-stable} if $A
\otimes D\cong A$. When $A$ is unital, an equivalent formulation of  $D$-stability for $A$ is the following condition: given finite sets $X$ in $A$ and $Y$ in $D$ and $\vp>0$, there exists an 
embedding $\phi:D\hookrightarrow A$ such that $\|\phi(y)x-x\phi(y)\|<\vp$ for all $x
\in X$ and $y\in Y$ (\cite[Theorem 2.2]{Toms.StrongSelfAbsorbing}).

\begin{corollary}\label{ssa-closed}
Let $C$ and $D$ be ${\mathrm C}^*$-algebras with $D$ separable and strongly 
self-absorbing in the sense of \cite{Toms.StrongSelfAbsorbing}. Then, within the space 
of separable ${\mathrm C}^*$-subalgebras of $C$ containing $1_C$, the $D$-stable 
${\mathrm C}^*$-algebras form a closed subset. 
\end{corollary}
\begin{proof}
Suppose that $A \subseteq C$ is a unital C$^{*}$-subalgebra such that, for any $
\gamma >0$, there is a $D$-stable C$^{*}$-algebra $B \subseteq C$ with $d(A,B)<
\gamma$. We have to show that $A$ is $D$-stable. To this end, note that any finite 
subset $X$ of $A$ may be approximated arbitrarily well by a finite subset $\bar{X}$ 
in a nearby unital $D$-stable $B$. Now for a finite subset $Y$ of $D$, there is an 
embedding of $D$ into $B$ such that $Y$ almost commutes with $\bar{X}$.  From 
Corollary~\ref{half-flip-close} we obtain an embedding of $D$ into $A$ which sends 
$Y$ to a close subset $\bar{Y} \subseteq A$. Now if $D$ and $B$ were sufficiently 
close, $\bar{Y}$ and $X$ will almost commute. By \cite[Theorem~2.2]
{Toms.StrongSelfAbsorbing}, this is enough to ensure that $A$ is $D$-stable.
\end{proof}
We have not been able to decide whether the $D$-stable subalgebras also form an 
open subset. However, Corollary \ref{half-flip-close} immediately gives results for 
embeddings of strongly self-absorbing ${\mathrm C}^*$-algebras, since these have 
approximately inner half flip. 

\begin{corollary}\label{Z-Embedd}
Let $A\subset_\gamma B$ be a near inclusion of ${\mathrm C}^*$-algebras on a Hilbert space 
$H$, where $0<\gamma<1/12600000\approx 7.9\times 10^{-8}$, and let $D$ be a strongly self-absorbing C$^{*}$-algebra.  If $A$ 
admits an embedding of $D$, then so does $B$.  Moreover, on finite 
subsets of the unit ball of $D$ one may choose the embedding into $B$ to 
be within $152\gamma^{1/2}$ of the embedding into $A$. 
\end{corollary}
Recent progress in the structure theory of nuclear C$^{*}$-algebras suggests that 
the preceding corollaries are particularly interesting in the case where $D = \mathcal{Z}$, the Jiang--Su algebra. 
Moreover, $\mathcal{Z}$-stability is relevant also for non-nuclear C$^{*}$-algebras; for example it will be 
shown in \cite{Johanesova.ZStableSimilarity} that it implies finite length (hence 
Kadison's similarity property). We have not been able to establish whether  $
\mathcal{Z}$-stability is preserved under closeness, i.e., to answer the following 
question: If $A$ and $B$ are sufficiently close, and $A$ is $\mathcal{Z}$-stable, is 
$B$ $\mathcal{Z}$-stable as well? The previous corollary at least shows that the existence 
of embeddings of $\mathcal{Z}$ is preserved under close containment. The 
existence of such embeddings is highly nontrivial (even for otherwise well-behaved 
C$^{*}$-algebras, such as simple, unital AH algebras, cf.\ \cite{Dadarlat.ZEmbed}).

\section{Unitary Equivalence}\label{Unitary}

\indent 

In the previous section we have shown that two close separable nuclear ${\mathrm C}^*$-
subalgebras $A$ and $B$ of ${\mathbb B}(H)$ are $*$-isomorphic. We now show in 
Theorem \ref{thm10.4} that there is a unitary $u$ such that $uAu^*=B$ when $H$ is 
separable. This establishes Theorem \ref{Intro.Unitary} of the introduction and gives 
a complete answer to Kadison and Kastler's question from \cite{Kadison.Kastler} in this context.

The technicalities of the upcoming proofs warrant some additional overall explanation
of the methods employed. Before describing our approach, we  recall our unitisation conventions from Notation 
\ref{Prelim.Dagger}, so that if $A$ and $B$ have the same ultraweak closure, then $A^
\dagger$ and $B^\dagger$ have the same unit, $e_A$, where $e_A$ is the support 
projection of $A$ (and of $B$).

In Theorem \ref{thm10.3} below, we obtain the unitary that implements a $*$-isomorphism under the additional
assumption that $A$ and $B$ have the same ultraweak closure.  This
restriction is removed in Theorem \ref{thm10.4} by using known perturbation
results for injective von Neumann algebras to reduce to the situation
of Theorem \ref{thm10.3}.  The assumption that $A$ and $B$ have the same
ultraweak closure enables us to approximate $^*$-strongly unitaries in C$^*(A,B,1)$
with a uniform
spectral gap    by unitaries in $A$ or $B$
using the Kaplansky density result of Lemma \ref{lem3.9}.   The basic idea of
our proof is to construct a sequence of unitaries $\{u_n\}_{n=1}^\infty$ in
C$^*(A,B,1)$ using Lemma \ref{Avg.2} and Theorem \ref{Close.Iso.Main} so that
$\lim_{n\to\infty}\ad(u_n)$ exists in the point-norm topology and
defines a surjective $*$-isomorphism from $A$ onto $B$. If this sequence
converged $^*$-strongly to a unitary $u$, then this would implement
unitary equivalence. However, there is no reason to believe that this
happens.  Thus we produce a modified sequence $\{v_n\}_{n=1}^\infty$ of unitaries
which explicitly converges  $^*$-strongly, while maintaining
the requirement that $\lim_{n\to\infty}\ad(v_n)$ gives a
surjective $*$-isomorphism from $A$ onto $B$. Lemma \ref{lem3.9} enables us to
 approximate $^*$-strongly the unitaries obtained  in C$^*(A,B,1)$ by unitaries in  
$B^\dagger$. In principle, the idea is to multiply by unitaries
produced by Lemma \ref{lem3.9}, to ensure $^*$-strong convergence.  In order for  $\lim_{n\rightarrow\infty}\ad(v_n)$ (in the point-norm topology) to exist and define a $*$-isomorphism, it is essential that
these unitaries can be taken to approximately commute with suitable
finite sets.  In practice, we have further technical
hurdles to  overcome. One instance of this is the need  to ensure that
the unitaries to which we apply Lemma \ref{lem3.9}  have the required uniform spectral
gap.  This leads us to split off two technical lemmas (Lemmas \ref{lem10.1} and \ref{lem10.2})
which combine the results of Lemmas \ref{Avg.2}, \ref{lem3.9} and Theorem \ref{Close.Iso.Main} in
exactly the correct order for use in the inductive step of Theorem
\ref{thm10.3}. We advise the reader to begin this sequence of results with the latter one, referring back to the preceding two as needed.

The bounds on the constant $\gamma$ in the next three results are chosen so that 
whenever we wish to use Lemma \ref{Avg.2} or Theorem \ref{Close.Iso.Main}, it is 
legitimate to do so. In particular, the choice ensures that $392\gamma^{1/2}\le 
13/150$, so that Lemma \ref{Avg.2}  and Remark \ref{Avg.2.Rem} (ii) can be 
applied with $\gamma$ replaced by $\alpha\gamma^{1/2}$ for $\alpha\le 392$. It 
also guarantees the validity of \eqref{eq10.38}, an inequality that governs the overall bound 
on $\gamma$.
In part (VIII) of the following lemma, the inequality $1848\gamma^{1/2}\leq
1848\times 10^{-4}<1$ will allow us to use Lemma \ref{lem3.9}.

\begin{lemma}\label{lem10.1}
Suppose that $A$ and $B$ are separable and nuclear ${\mathrm C}^*$-algebras acting non-
degerately on a Hilbert space $H$. Suppose that $A''=B''=M$ and $d(A,B) <\gamma 
\le 10^{-8}$. Given finite subsets $X_0,Z_{0,B}$ of the unit ball of $B$, a finite 
subset $Z_{0,A}$ of the unit ball of $A$, and constants $\vp_0>0$, $\mu_0>0$, 
there exist finite subsets $Y_0,Z_0$ of the unit ball of $B$, $\delta_0>0$, a unitary 
$u_0\in M$ and a surjective $*$-isomorphism $\sigma :   B\to A$ with the following 
properties:
\begin{itemize}
 \item[\rm (I)] $X_0\subseteq Y_0$.
\item[\rm (II)] $\delta_0<\vp_0/2$.
\item[\rm (III)] $Z_{0,B} \subseteq Z_0$.
\item[\rm (IV)] $\sigma\approx_{Z_0,28\gamma^{1/2}}\iota$ and $\sigma^{-1} 
\approx_{Z_{0,A},28\gamma^{1/2}}\iota$.
\item[\rm (V)] $\sigma \approx_{Y_0,\delta_0/2} \ad(u_0)$.
\item[\rm (VI)] $\|u_0-1_M\|\le 84\gamma^{1/2}$.
\item[\rm (VII)] Given any $*$-homomorphism $\psi:B\rightarrow D$ for some unital 
${\mathrm C}^*$-subalgebra $D$ of $M$ with $\psi\approx_{Z_0, 364\gamma^{1/2}}\iota$, there 
exists a unitary $w_0\in {\mathrm C}^*(A,D,1_M)$ with $\|1_M-w_0\| \le 1176 \gamma^{1/2}$ such that
\begin{equation}
 \text{\rm Ad}(w_0)\circ \psi \approx_{Y_0,\delta_0/2}\sigma.
\end{equation}
\item[\rm (VIII)] Given a unitary $v_0\in M$ with $\|v_0-u_0\| \le 1848\gamma^{1/2}$ and $\text{\rm 
Ad}(v_0) \approx_{Y_0,\delta_0} \sigma$, and given a finite subset $S$ of the unit 
ball of $H$, there exists a unitary $v'_0 \in \Bu $ satisfying $\|v'_0-1_{B^\dagger}\|\le 1848\gamma^{1/2}$, $
\text{\rm Ad}(v_0v'_0) \approx_{X_0,\vp_0}\sigma$, and 
\begin{equation}\label{Unitary.10001}
\|(v_0v'_0-u_0)\xi\|,\quad \|(v_0v'_0-u_0)^*\xi\| < \mu_0,\qquad \xi\in S.
\end{equation}
\end{itemize}
\end{lemma}

\begin{proof}
By Lemma \ref{lem3.9} applied to the unitisation $\Bu $, there is a finite subset $\wt 
Y_0 \subseteq\Bu $ and $\delta_0>0$ such that if $u\in M$ is a unitary satisfying $\|
u-1_M\|\le 1848\gamma^{1/2}$ and
\begin{equation}\label{eq10.1}
 \|\tilde yu- u\tilde y\|\leq 3\delta_0,\qquad {\tilde y}\in \wt Y_0,
\end{equation}
and a finite subset $S_0$ of $H$ is given, then there exists a unitary $v\in \Bu $ 
satisfying $\|v-1_{B^\dagger}\|\le 1848\gamma^{1/2}$,
\begin{align}\label{eq10.2}
 &\|xv-vx\| \le \vp_0/2,\qquad x\in X_0,\\
\intertext{and}
\label{eq10.3}
&\|(v-u)\xi\|, \ \|(v-u)^*\xi\| < \mu_0,\qquad \xi\in S_0.
\end{align}
In applying this lemma, we have replaced $\vp_0$ by $\vp_0/2$ and $\delta_0$ by $3\delta_0$. We may replace $\delta_0$ by any smaller number, ensuring that condition (II) holds.

Each $\tilde y_i \in \wt Y_0$ can be written $\tilde y_i = \alpha_i +2y_i$ for scalars $
\alpha_i$ and elements $y_i$ in the unit ball of $B$, and let $Y_1$ denote the 
collection of these $y_i$'s. Define $Y_0=Y_1\cup X_0$, so that condition (I) is satisfied. If the inequality
\begin{equation}\label{eq10.4}
 \|yu-uy\| < 3\delta_0/2,\qquad y\in Y_0,
\end{equation}
holds for a particular unitary $u\in M$, then it also holds for $y\in Y_1\subseteq Y_0$,
implying \eqref{eq10.1} and hence \eqref{eq10.2} and \eqref{eq10.3}. 

Since $*$-homomorphisms are $(X,\delta)$-approximate $*$-homomorphisms for 
any set $X$ and any $\delta>0$, we may apply Lemma \ref{Avg.2} and Remark 
\ref{Avg.2.Rem} (ii) for any unital ${\mathrm C}^*$-subalgebra $D$ of $M$ to conclude that 
there exists a finite set $Z_0$ in the unit ball of $B$ with the following property. If $
\phi_1,\phi_2:B\rightarrow D$ are $*$-homomorphisms with $\phi_1 
\approx_{Z_0,\alpha\gamma^{1/2}} \phi_2$, where $\alpha\in \{28, \, 392\}$, then 
there exists a unitary $w_0\in D$ with $\|1_D-w_0\| \le 3\alpha\gamma^{1/2}$ and $
\text{Ad}(w_0)\circ \phi_1 \approx_{Y_0,\delta_0/2}\phi_2$. By increasing $Z_0$ to 
contain $Z_{0,B}$, condition (III) is satisfied. 

By Theorem \ref{Close.Iso.Main} we may now choose a surjective $*$-isomorphism 
$\sigma:B\rightarrow A$ so that condition (IV) holds. Since $\sigma$ and $\iota$ are 
$*$-homomorphisms of $B$ into ${\mathrm C}^*(A,B,1_M)\subseteq M$, there then exists a unitary $u_0\in M$ 
with $\|u_0-1_M\| \le 84\gamma^{1/2}$ so that $\sigma \approx_{Y_0,\delta_0/2} 
\text{Ad}(u_0)$, using $\alpha=28$ above. This establishes conditions (V) and (VI).

From condition (IV), we have $\sigma \approx_{Z_0,28\gamma^{1/2}}\iota$. If $\psi: B\rightarrow D\subseteq M$ is another $*$-homomorphism satisfying $\psi 
\approx_{Z_0,364\gamma^{1/2}}\iota$, then $\psi\approx_{Z_0, 392\gamma^{1/2}}
\sigma$. Taking $\alpha=392$ above,  the choice of $Z_0$ allows us to find a 
unitary $w_0\in M$ with $\|1_M-w_0\| \le 1176\gamma^{1/2}$ so that condition (VII) is 
satisfied.

It only remains to establish (VIII). Suppose that $v_0\in M$ satisfies $\|v_0-u_0\| \le 
1848\gamma^{1/2}$ and $\text{Ad}(v_0)\approx_{ Y_0,\delta_0}\sigma$, and that a finite subset $S$ 
of the unit ball of $H$ is given. Since $\sigma \approx_{Y_0,\delta_0/2} \text{Ad}
(u_0)$ from condition (V), we obtain
\begin{equation}\label{eq10.5}
\text{Ad}(v_0) \approx_{Y_0,3\delta_0/2} \text{Ad}(u_0),
\end{equation}
implying that
\begin{equation}\label{eq10.6}
\|v^*_0u_0y - yv^*_0u_0\| \le 3\delta_0/2,\qquad y\in Y_0.
\end{equation}
Since $\|1-v^*_0u_0\|\le 1848\gamma^{1/2}$, we can take $S_0=S\cup v_0^*S$ so the choice of $\widetilde Y_0$ at the start of the proof allows us to find a unitary  $v'_0\in B^\dagger $ with the following properties:
\begin{enumerate}
\item $\|1-v'_0\|\le 1848\gamma^{1/2}$,
\item $\|v'_0x - xv'_0\| < \varepsilon_0/2$, $\ \ x\in X_0$,
\item $\|(v'_0-v^*_0u_0)\xi\|<\mu$, $\|(v'_0-v^*_0u_0)^* v^*_0\xi\| < \mu$, $\ \ \xi\in S$.
\end{enumerate}  
 The third condition above gives \eqref{Unitary.10001}, while the second shows that $\ad(v'_0)\approx_{X_0,\varepsilon_0/2}\iota$.  
Since $X_0\subseteq Y_0$ and $\delta_0<\varepsilon_0/2$, we have $\text{Ad}
(v_0)\approx_{X_0,\varepsilon_0/2} \sigma$. It follows that  $\text{Ad}(v_0v'_0) 
\approx_{X_0,\varepsilon_0}\sigma$, and condition (VIII) is proved.
\end{proof}

\begin{lemma}\label{lem10.2}
Suppose that $A$ and $B$ are separable and nuclear ${\mathrm C}^*$-algebras acting non-
degerately on a Hilbert space $H$. Suppose that $A''=B''=M$ and $d(A,B) <\gamma 
\le 10^{-8}$. Given finite subsets $X$ in the unit ball of $A$ and $Z_B$ in the unit 
ball of $B$ and constants $\vp,\mu>0$, there exist finite subsets $Y$ of the unit ball 
of $A$ and $Z$ of the unit ball of $B$, a constant $\delta>0$, a unitary $u\in M$, and  a 
surjective $*$-isomorphism $\theta :  A\to B$ with the following properties.
\begin{itemize}
 \item[\rm (i)] $\delta<\vp$.
\item[\rm (ii)] $X\subseteq_\vp Y$.
\item[\rm (iii)] $\|u-1_M\| \le 252\gamma^{1/2}$.
\item[\rm (iv)] $\theta \approx_{Y,\delta} \text{\rm Ad}(u)$.
\item[\rm (v)] $\theta \approx_{X, 364\gamma^{1/2}}\iota,\  \theta^{-1} \approx_{Z_B, 
364 \gamma^{1/2}}\iota$.
\item[\rm (vi)] Given a surjective $*$-isomorphism $\phi:A\rightarrow B$ with $
\phi^{-1} \approx_{Z, 364\gamma^{1/2}}\iota$, there exists a unitary $w\in \Bu $ with $
\|w-u\| \le 1596 \gamma^{1/2}$ and $\text{\rm Ad}(w) \circ \phi \approx_{Y,\delta/2} 
\theta$. 
\item[\rm (vii)] Given a unitary $v\in M$  with $\|v-u\| \le 1848\gamma^{1/2}$ and $\text{\rm Ad}(v) 
\approx_{Y,\delta} \theta$, and given any finite subset $S$ of the unit ball of $H$, 
there exists a unitary $v'\in \Bu $ with $\|1_{B^\dagger}- v'\|\le 1848\gamma^{1/2}$, $\text{\rm Ad}(v'v) \approx_{X,
\vp} \theta$, and
\begin{equation}\label{eq10.9}
 \|( v'v-u)\xi\|, \ \|(v'v-u)^*\xi\| < \mu,\qquad \xi\in S.
\end{equation}
\end{itemize}
\end{lemma}

\begin{proof}
By Lemma \ref{Avg.2} and Remark \ref{Avg.2.Rem} (ii), there exists a finite 
subset $Z_1$ of the unit ball of $A$ with the following property. Given a unital C$^*
$-subalgebra $D$ of $M$, if $\phi_1,\phi_2 :  A\to D$ are $*$-homomorphisms 
such that $\phi_1 \approx_{Z_1,56\gamma^{1/2}} \phi_2$ then there is a unitary 
$w_1\in D$ satisfying $\|w_1-1_D\|\le 168\gamma^{1/2}$ and $\phi_1 \approx_{X,\vp/
3} \text{Ad}(w_1)\circ \phi_2$. 

By Theorem \ref{Close.Iso.Main}, we may choose a surjective $*$-isomorphism $\beta :   
A\to B$ with the property  that $\beta \approx_{Z_1,28 \gamma^{1/2}}\iota$. Then define a finite 
subset $X_0$ of the unit ball of $B$ by $X_0=\beta(X)$. Taking $\vp_0 = \vp/3$, $
\mu_0=\mu$, $Z_{0,A} = X$ and $Z_{0,B} = \beta(Z_1)\cup Z_B$, we may apply 
Lemma \ref{lem10.1} to obtain $Y_0, Z_0, \delta_0, \sigma$ and $u_0$ satisfying 
conditions (I)--(VIII) of this lemma. We then define $Z=Z_0 \subseteq B$ and $
\delta=\delta_0$. By Lemma \ref{lem10.1} (II), $\delta < \vp_0/2<\vp$ so condition 
(i) holds. By Lemma \ref{lem10.1} (IV), $\sigma \approx_{Z_0, 28\gamma^{1/2}}\iota$, 
so $\sigma\circ\beta \approx_{Z_1,56\gamma^{1/2}}\id_A$, since $\beta(Z_1) \subseteq 
Z_{0,B}\subseteq Z_0$ from Lemma \ref{lem10.1} (III). By the choice of $Z_1$ 
above, there exists a unitary $w_1\in\Au\subseteq M$ with $\|1_M-w_1\| \le 168\gamma^{1/2}$ so 
that 
\begin{equation}\label{eq10.10}
 \text{Ad}(w_1) \approx_{X,\vp/3} \sigma\circ \beta.
\end{equation}
Now define 
a surjective $*$-isomorphism
$\theta :  A\to B$ to be $\sigma^{-1}\circ \text{Ad}(w_1)$. Then $
\theta^{-1} = \text{Ad}(w^*_1)\circ\sigma$, and \eqref{eq10.10} can be rewritten as
\begin{equation}\label{eq10.10a}
 \theta \approx_{X,\vp/3}  \beta,
\end{equation}
which implies that
\begin{equation}\label{eq10.10b}
 \theta^{-1} \approx_{X_0,\vp/3}  \beta^{-1}.
\end{equation}
 If $z\in Z_0$ then, from Lemma \ref{lem10.1} (IV),
\begin{equation}\label{eq10.11}
 \|\sigma(z) - z\| \le 28\gamma^{1/2}.
\end{equation}
Thus
\begin{align}
 \|\theta^{-1}(z)-z\| &= \|\text{Ad}(w^*_1)(\sigma(z)-z) + \text{Ad}(w^*_1)(z)-z\|\notag\\
&\le 28\gamma^{1/2} + 2\|w^*_1-1_M\|\notag\\
\label{eq10.12}
&\le 364\gamma^{1/2},\qquad z\in Z_0.
\end{align}
Consequently $\theta^{-1} \approx_{Z_0,364\gamma^{1/2}}\iota$, so the second 
statement of condition (v) holds since $Z_B\subseteq Z_{0,B} \subseteq Z_0$. 
From Lemma \ref{lem10.1} (IV), $\sigma^{-1} \approx_{Z_{0,A}, 28\gamma^{1/2}}\iota
$, so a similar calculation leads to  $\theta\approx_{Z_{0,A}, 364\gamma^{1/2}} \iota$, 
establishing the first statement of condition (v), since $X=Z_{0,A}$.

We now define $u =u^*_0w_1$. By Lemma \ref{lem10.1} (VI), $\|u_0-1_M\| \le 84 
\gamma^{1/2}$ and so
\begin{align}
 \|u-1_M\| &= \|w_1-u_0\| \le \|w_1-1_M\| + \|1_M-u_0\|\notag\\
\label{eq10.13}
&\le 168 \gamma^{1/2} + 84\gamma^{1/2}  = 252 \gamma^{1/2},
\end{align}
giving condition (iii). Now let
\begin{equation}\label{eq10.13a} 
Y = w^*_1\sigma(Y_0)w_1=\theta^{-1}(Y_0),
\end{equation}
where the latter equality comes from the definition of $\theta$. From Lemma 
\ref{lem10.1} (V), $\sigma \approx_{Y_0,\delta/2} \text{Ad}(u_0)$, and so $
\sigma^{-1} \approx_{\sigma(Y_0),\delta/2} \text{ Ad}(u^*_0)$. The relation $\theta = 
\sigma^{-1}\circ \text{Ad}(w_1)$ then gives
\begin{equation}\label{eq10.14}
 \theta \approx_{w^*_1\sigma(Y_0)w_1,\delta/2} \text{ Ad}(u),
\end{equation}
since $u = u^*_0w_1$ and so condition (iv) holds.
By definition, $Y = \theta^{-1}(Y_0)$ and $\beta(X) = 
X_0$. Moreover, $X_0\subseteq Y_0$ from Lemma \ref{lem10.1} (I). Using \eqref{eq10.10b}, we see that
\begin{equation}\label{eq10.17}
X = \beta^{-1}(X_0) \subseteq_{\vp/3} \theta^{-1}(X_0) \subseteq \theta^{-1}(Y_0) = Y,
\end{equation}
verifying condition (ii).

Now consider a surjective $*$-isomorphism $\phi :   A\to B$ with $\phi^{-1} 
\approx_{Z,364\gamma^{1/2}}\iota$, which we extend  canonically to a unital 
$*$-isomorphism of $\Au $ onto $\Bu $, also denoted $\phi$. By Lemma \ref{lem10.1} 
(VII), there is a unitary $w_0\in {\mathrm C}^*(A,1_M)=\Au\subseteq M$ with $\|w_0-1_M\|\le 1176\gamma^{1/2}$
 such that
\begin{equation}\label{eq10.15}
\text{Ad}(w_0) \circ \phi^{-1} \approx_{Y_0,\delta/2} \sigma.
\end{equation}
Each $y\in Y$ can be written as $\text{Ad}(w^*_1) \circ\sigma(y_0) = \theta^{-1}
(y_0)$ for some $y_0\in Y_0$, so
\begin{align}
\|\theta(y) - \phi(w^*_0w_1yw^*_1w_0)\| &= \|y_0-(\phi\circ \text{Ad}(w^*_0) \circ
\sigma)(y_0)\|\notag\\
&= \|(\text{Ad}(w_0) \circ \phi^{-1})(y_0) - \sigma(y_0)\|\notag\\
&\le \delta/2,\label{eq10.15a}
\end{align}
from \eqref{eq10.15}. Thus 
\begin{equation}\label{eq10.15b}
\theta\approx_{Y,\delta/2}\phi\circ \text{Ad}(w^*_0w_1).
\end{equation}
Then
\begin{equation}\label{eq10.16}
\phi\circ \text{Ad}(w^*_0w_1) = \text{Ad}(\phi(w^*_0w_1))\circ\phi,
\end{equation}
since $w^*_0w_1\in \Au $. If we define $w=\phi(w^*_0w_1)\in \Bu $, then $\theta 
\approx_{Y,\delta/2} \text{Ad}(w)\circ\phi$ from \eqref{eq10.15b} and 
\eqref{eq10.16}. Moreover, the earlier estimates $\|w_0-1_M\| \le 1176\gamma^{1/2}$ 
and $\|w_1-1_M\| \le 168\gamma^{1/2}$ give $\|w_0^*w_1-1_M\| \le 1344\gamma^{1/2}$ and
so $\|w-1_M\| \le 1344\gamma^{1/2}$.
Recalling the estimate of \eqref{eq10.13}, 
\begin{align}
\|w-u\| &\le \|w-1_M\|+\|u-1_M\|\notag\\
&\leq 1344\gamma^{1/2}+252\gamma^{1/2}=1596\gamma^{1/2},\label{10.16a}
\end{align}
and condition (vi) is verified.

It only remains to verify condition (vii). Consider a unitary $v\in M$ with $\|v-u\|\le 
1848\gamma^{1/2}$, and $\text{Ad}(v) \approx_{Y,\delta} \theta$, and fix a finite subset $S$ of the 
unit ball of $H$. Then
\begin{equation}\label{eq10.18}
\text{Ad}(vw^*_1) \approx_{w_1Yw^*_1,\delta} \theta\circ \text{Ad}(w^*_1)=\sigma^{-1}
\end{equation}
so
\begin{equation}\label{eq10.19}
\text{Ad}(vw^*_1) \approx_{\sigma(Y_0),\delta} \sigma^{-1}
\end{equation}
and
\begin{equation}\label{eq10.20}
\text{Ad}(w_1v^*)\approx_{Y_0,\delta}\sigma,
\end{equation}
using $\theta = \sigma^{-1} \circ \text{ Ad}(w_1)$ and $Y = \theta^{-1}(Y_0)$. Since $u = 
u^*_0w_1$, we have the estimate 
\begin{equation}\label{eq10.20a}
\|w_1v^* - u_0\| = \|v^*-u^*\| \le 1848\gamma^{1/2}.
\end{equation}
 Let $S' = S\cup \{w_1\xi\colon \ \xi\in S\}$, a finite subset of the unit ball of $H$. 
Then Lemma \ref{lem10.1} (VIII), with $v_0=w_1v^*$, gives a unitary $v'_0 \in \Bu $ 
satisfying $\|v'_0-1_M\|\le 1848\gamma^{1/2}$, $\text{Ad}(w_1v^*v'_0) \approx_{X_0,\vp/3}\sigma$, and 
\begin{equation}\label{eq10.21}
 \|(w_1v^*v'_0-u_0)\xi\|, \ \|(w_1v^*v'_0-u_0)^* w_1\xi\| < \mu,\qquad\xi\in S.
\end{equation}
Set $v' = v^{\prime *}_0$. Then $\|1_M-v'\| \le 1848\gamma^{1/2}$, and
\begin{equation}\label{eq10.22}
 \|(v^*v^{\prime *}-u^*)\xi\|,\  \|(v'v-u)\xi\| < \mu,\qquad \xi\in S.
\end{equation}
Moreover, since $w_1v^*v^{\prime *} =w_1v^*v'_0$, we obtain
\begin{equation}\label{eq10.23}
 \text{Ad}(v^*v^{\prime *})\approx_{X_0,\vp/3} \text{Ad}(w^*_1)\circ\sigma = 
\theta^{-1},
\end{equation}
and
\begin{equation}\label{eq10.25}
 \text{Ad}(v^*v^{\prime *}) \approx_{X_0,2\vp/3} \beta^{-1}
\end{equation} 
from \eqref{eq10.23} and \eqref{eq10.10b}.
Thus
\begin{equation}\label{eq10.26}
 \text{Ad}(v'v) \approx_{X,2\vp/3} \beta,
\end{equation}
because $X=\beta^{-1}(X_0)$. Since $\theta\approx_{X,\vp/3}\beta$ from 
\eqref{eq10.10a}, we obtain $\text{Ad}(v'v) \approx_{X,\vp}\theta$. Thus condition 
(vii) holds, completing the proof.
\end{proof}

We are now in a position to prove  one of the main results of the paper, the unitary implementation of 
isomorphisms between separable nuclear close ${\mathrm C}^*$-algebras. We first prove this 
under the additional hypothesis that the two algebras have the same ultraweak closure, from which we will deduce the general case 
subsequently.

\begin{theorem}\label{thm10.3}
Suppose that $A$ and $B$ are ${\mathrm C}^*$-algebras acting non-degenerately on a separable Hilbert space $H$, and that $A$ is separable and nuclear. Suppose that $A''=B''=M$ and $d(A,B) <\gamma \le 10^{-8}$. Then there exists a unitary $u\in M$ such that $uAu^*=B$.
\end{theorem}

\begin{proof}
Since $\gamma < 1/101$, Propositions \ref{Prelim.NuclearOpen} and 
\ref{Prelim.SeparableOpen} show that $B$ is also separable and nuclear.
Fix dense sequences $\{a_n\}^\infty_{n=1}$ and $\{b_n\}^\infty_{n=1}$ in the unit 
balls of $A$ and $B$ respectively, and a dense sequence $\{\xi_n\}^\infty_{n=1}$ in 
the unit ball of $H$. We will construct inductively finite subsets $\{X_n\}_{n=0}^{\infty}$ and $\{Y_n\}_{n=0}^{\infty}$ of the unit 
ball of $A$, finite subsets $\{Z_n\}_{n=0}^{\infty}$ of the unit ball of $B$, positive constants $\{\delta_n\}_{n=0}^{\infty}$, 
surjective $*$-isomorphisms $\{\theta_n :  A\to B\}_{n=0}^{\infty}$, and unitaries $\{u_n\}_{n=0}^{\infty}$ in $M$ to 
satisfy the following conditions. 
\begin{itemize}
 \item[(1)] $a_1,\ldots, a_n\in X_n, \ n\ge 1$.
\item[(2)] $X_n \subseteq_{2^{-n}/3} Y_n$ and $\delta_n < 2^{-n}, \ n\ge 0$.
\item[(3)] $\theta_n \approx_{ X_{n-1},2^{-(n-1)}} \theta_{n-1}, \ n\ge 1$.
\item[(4)] $\theta_n \approx_{ Y_n,\delta_n} \text{ Ad}(u_n),\  n\ge 0$.
\item[(5)] $\|(u_n-u_{n-1})\xi_i\|,\  \|(u_n-u_{n-1})^*\xi_i\| < 2^{-n}$ for $1\le i\le n$.
\item[(6)] For each $1\le i \le n$, there exists $x\in X_n$ with $\|\theta_n(x)-b_i\| \le 
9/10$.
\item[(7)] Given a surjective $*$-isomorphism $\phi :  A\to B$ with $\phi^{-1} 
\approx_{Z_n,364\gamma^{1/2}}\iota$, there exists a unitary $w\in \Bu $ with $\|w-u_n
\| \le 1596\gamma^{1/2}$ and $\text{Ad } w\circ\phi \approx_{Y_n,\delta_n/
2}\theta_n$.
\item[(8)] Given a unitary $v\in M$ with $\|v-u_n\|\le 1848\gamma ^{1/2}$ and $
\text{Ad}(v) \approx_{Y_n,\delta_n} \theta_n$, and given any finite subset $S$ of the 
unit ball of $H$, there exists a unitary $v'\in \Bu $ with $\|v'-1_M\| \le 
1848\gamma^{1/2}$, $\text{Ad}(v'v) \approx_{X_n,2^{-(n+1)}} \theta_n$, and
\begin{equation}\label{eq10.27}
 \|(v'v-u_n)\xi\|, \ \|(v'v-u_n)^*\xi\| < 2^{-(n+1)},\qquad \xi\in S.
\end{equation}
\item[(9)] There exists a unitary $z\in \Bu $ with $\|z-u_n\| \le 252\gamma^{1/2}$.
\end{itemize}

Conditions (7)--(9) are not needed to derive unitary equivalence but are used in the inductive step. Assuming that the induction has been accomplished, we first show how  conditions 
(1)--(6) establish unitary implementation.

Conditions (1) and (3) imply that the sequence $\{\theta_n\}^\infty_{n=1}$ 
converges in the point norm topology to a $*$-isomorphism $\theta$ of $A$ into $B
$. Fix  $i\geq 1$. For a given integer $n\ge i$, condition (6) allows 
us to choose $x\in X_n$ so that $\|\theta_n(x)-b_i\| \le 9/10$. By condition (3), $\|
\theta_{m+1}(x) - \theta_m(x)\| \le 2^{-m}$ for $m\ge n$. Thus
\begin{align}
 \|\theta(x) - b_i\| &\le \|\theta_n(x) - b_i\| + \sum^\infty_{m=n} 2^{-m}\notag\\
\label{eq10.27a}
&\le 9/10 + 2^{-(n-1)}.
\end{align}
Since $n\ge i$ was arbitrary, \eqref{eq10.27a} and the density of $\{b_i\}_{i=1}^{\infty}$ in the unit ball of $B$ show that $d(B,\theta(A)) \le 9/10$, 
and so $\theta(A) = B$ by Proposition \ref{EasyFact}. Thus $\theta :  A\to B$ is a surjective $*$-isomorphism.
Since the unitary group of $M$ is closed in the $^*$-strong topology, condition (5) 
ensures that the sequence $\{u_n\}^\infty_{n=1}$ converges $^*$-strongly to a unitary $u\in M$. 
Conditions (1), (2) and (4) then show that $\theta = \text{Ad}(u)$, and so $B=uAu^*
$, proving the result.

We start the induction by taking $X_0 =Y_0 =\emptyset$, $Z_0=\emptyset$,
$\delta_0=1/2$, 
$u_0=1$, and $\theta_0$ any $*$-isomorphism of $A$ onto $B$, possible by 
Theorem \ref{Close.Iso.Main}. At this initial level, conditions (1), (3), (5) and (6) do not have meaning, but these will not be used in the inductive step.  Conditions (2) and (4) are trivial (as $X_n=Y_n=\emptyset$), while conditions (7) and (9) are satisfied by taking $w=1_M$ and $z=1_M$ respectively.  In condition (8) given a unitary $v\in M$ with $\|v-u_0\|\leq 1848\gamma^{1/2}$ and any finite subset $S$ of the unit ball of $H$, take $v'=v^*$ so that $\|v'-1_M\|=\|v-u_0\|$ and the left hand side of \eqref{eq10.27} vanishes.  It remains to carry out the inductive step. In order to distinguish the conditions that we are assuming at level $n$ from those we are proving at the next level, we employ the notations $(\cdot)_n$ and $(\cdot)_{n+1}$ as appropriate. 

Now suppose that the various objects have been constructed to satisfy (1)$_n$--
(9)$_n$. Let $z\in \Bu $ be the unitary of condition (9)$_n$ satisfying $\|z-u_n\| \le 
252\gamma^{1/2}$. Then $z^*b_iz$ lies in the unit ball of $B$ for $1\le i\le n+1$, so we may choose 
elements $x_i$ in the unit ball of $A$ such that $\|x_i-z^*b_iz\| <\gamma$, for $1\le i
\le n+1$. Now define
\begin{equation}
 X_{n+1} = X_n \cup Y_n \cup \{a_1,\ldots ,a_{n+1}\}\cup \{x_1, \ldots ,x_{n+1}\}  
\end{equation}
so that condition (1)$_{n+1}$ is satisfied.

In Lemma \ref{lem10.2}, let $X  = X_{n+1}$, $\vp =   \delta_n/6$, $\mu = 2^{-(n+2)}$ 
and $Z_B = Z_n$, and let $Y_{n+1} \subseteq A$, $\delta_{n+1}>0$, $Z_{n
+1}\subseteq B$, $u\in M$ and $\theta :  A\to B$ be the resulting objects which 
satisfy conditions (i)--(vii) of that lemma. By $(2)_n$ and Lemma \ref{lem10.2} (i) $\delta_{n+1} < \vp= \delta_n/6 < 
2^{-(n+1)}/3$ and so the inequality $\delta_{n+1} < 2^{-(n+1)}$ holds. Also $X_{n+1} 
\subseteq_{2^{-(n+1)}/3} Y_{n+1}$ since $(2)_n$ ensures that $\vp \le 2^{-(n+1)}/3$. Thus condition 
(2)$_{n+1}$ is satisfied. By Lemma \ref{lem10.2} (v), $\theta^{-1} \approx_{Z_n, 
364\gamma^{1/2}}\iota$, so we may apply condition (7)$_n$ to find a unitary $w\in \Bu 
$ with $\|w-u_n\|\le 1596\gamma^{1/2}$ such that $\text{Ad}(w)\circ \theta 
\approx_{Y_n,\delta_n/2}\theta_n$. From Lemma \ref{lem10.2} (iv), $\text{Ad}(u) 
\approx_{Y_{n+1},\delta_{n+1}} \theta$. Since $Y_n \subseteq X_{n+1}\subset_\vp 
Y_{n+1}$ and $\delta_{n+1} \le \vp=\delta_n/6$, a simple triangle inequality 
argument gives $\text{Ad}(u) \approx_{Y_n,\delta_n/2} \theta$. It follows that
\begin{equation}\label{eq10.28}
 \text{Ad}(wu) \approx_{Y_n,\delta_n} \theta_n.
\end{equation}
By Lemma \ref{lem10.2} (iii), $\|u-1_M\| \le 252\gamma^{1/2}$, so
\begin{equation}\label{eq10.29}
 \|wu-u_n\|\le \|w-u_n\|+\|w(u-1_M)\|= \|w-u_n\| + \|u-1_M\| \le 1848\gamma^{1/2}.
\end{equation}
Thus the unitary $wu$ satisfies the initial hypotheses of condition (8)$_n$ which we 
can now apply to the set $S = \{\xi_1,\ldots, \xi_{n+1}\}$. 
Consequently there exists a unitary $v'\in\Bu $ with $\|v'-1_M\| \le 1848\gamma^{1/2}$,
\begin{equation}\label{eq10.30}
 \text{Ad}(v'wu) \approx_{X_n,2^{-(n+1)}} \theta_n,
\end{equation}
and
\begin{equation}\label{eq10.31}
 \|(v'wu-u_n)\xi_i\|, \ \|(v'wu-u_n)^*\xi_i\| < 2^{-(n+1)}, \qquad 1\le i \le n+1.
\end{equation}
Defining $u_{n+1} = v'wu$, we see that condition (5)$_{n+1}$ follows from \eqref{eq10.31}. Now $v'w\in \Bu $, 
so we may take $z$ in condition (9)$_{n+1}$  to be this unitary, since
\begin{equation}\label{eq10.32}
\|v'w-u_{n+1}\| = \|v'w-v'wu\| = \|1-u\|\le 252\gamma^{1/2}.
\end{equation}
Next, define $\theta_{n+1} = \text{Ad}(v'w)\circ \theta$, which maps $A$ onto $B$ 
because $B$ is an ideal in $\Bu $. Lemma \ref{lem10.2} (iv) gives $\text{Ad}(u) 
\approx_{Y_{n+1},\delta_{n+1}}\theta$, so applying $\text{Ad}(v'w)$ results in $
\text{Ad}(u_{n+1}) \approx_{Y_{n+1},\delta_{n+1}} \theta_{n+1}$, proving condition 
(4)$_{n+1}$.

We turn now to condition (3)$_{n+1}$. The choices of $X_{n+1}$, $Y_{n+1}$ and $
\delta_{n+1}$ give 
the inclusions $X_{n+1} \subseteq_{2^{-(n+1)}/3} Y_{n+1}$, and $X_n\subseteq 
X_{n+1}$, and also the inequality $\delta_{n+1} < 2^{-(n+1)}/3$. Thus
\begin{equation}\label{eq10.33}
 \theta_{n+1} =\text{Ad}(v'w)\circ \theta\approx_{X_n,2^{-(n+1)}} \text{ Ad}(u_{n+1}).
\end{equation}
Combined with \eqref{eq10.30}, we obtain
\begin{equation}\label{eq10.34}
 \theta_{n+1} \approx_{X_n,2^{-n}} \theta_n,
\end{equation}
and condition (3)$_{n+1}$ is proved.

The choice of $\theta$ from Lemma \ref{lem10.2} (v) entailed $\theta \approx_{X_{n
+1}, 364\gamma^{1/2}}\iota$. Applying $\text{Ad}(v'w)$ to this gives
\begin{equation}\label{eq10.35}
 \theta_{n+1} \approx_{X_{n+1},364\gamma^{1/2}} \text{ Ad}(v'w).
\end{equation}
Recall that, by construction, there are elements $x_i\in X_{n+1}$ such that
\begin{equation}\label{eq10.36}
 \|z^*b_iz-x_i\|\le \gamma,\qquad 1\le i\le n+1,
\end{equation}
and also that $\|z-u_n\|\le 252\gamma^{1/2}$, $\|v'-1_M\|\le 1848\gamma^{1/2}$, and 
$\|w-u_n\|\le 1596\gamma^{1/2}$. From these inequalities, it follows that
\begin{align}
 \|zx_iz^* - v'wx_i(v'w)^*\| &\le \|u_nx_iu^*_n - (v'w)x_i(v'w)^*\| + 504 
\gamma^{1/2}\notag\\
&\le \|wx_iw^* - (v'w)x_i(v'w)^*\| + 3696 \gamma^{1/2}\notag\\
&\le 2\|v'-1_M\| + 3696\gamma^{1/2}
\label{eq10.37}
\le 7392 \gamma^{1/2}.
\end{align}
Thus, for $1\le i \le n+1$, it follows from \eqref{eq10.35}, \eqref{eq10.36} and \eqref{eq10.37} that 
\begin{align}
 \|\theta_{n+1}(x_i)-b_i\| &\le \|\theta_{n+1}(x_i) - (v'w) x_i(v'w)^*\| + \|(v'w)x_i(v'w)^* - 
zx_iz^*\| + \|zx_iz^*-b_i\|\notag\\
&\le 364\gamma^{1/2} + 7392\gamma^{1/2} + \gamma
\label{eq10.38}
\le 7757\gamma^{1/2} \le 9/10,
\end{align}
since $\gamma \leq 10^{-8}$. This proves condition (6)$_{n+1}$.

We now prove condition (7)$_{n+1}$, so take a surjective $*$-isomorphism $\phi
 :  A\to B$ with $\phi^{-1} \approx_{Z_{n+1},364\gamma^{1/2}}\iota$.  By Lemma 
\ref{lem10.2} (vi), there exists a unitary $w'\in \Bu $ with $\|w'-u\|\le 1596 
\gamma^{1/2}$ and $\text{Ad}(w') \circ \phi\approx_{Y_{n+1},\delta_{n+1}/2} \theta
$. Apply $\text{Ad}(v'w)$ to this to obtain $\text{Ad}(v'ww') \circ\phi \approx_{Y_{n
+1},\delta_{n+1}/2} \text{ Ad}(v'w)\circ\theta$. Since $u_{n+1} = v'wu$, we have $\|
v'ww'-u_{n+1}\| \le 1596\gamma^{1/2}$ and also $\text{Ad}(v'ww') \circ\phi 
\approx_{Y_{n+1}, \delta_{n+1}/2} \theta_{n+1}$ since $\theta_{n+1} = \text{Ad}(v'w) 
\circ\theta$. Then $v'ww'$ is the required unitary in condition (7)$_{n+1}$ which 
now holds.

The last remaining condition is (8)$_{n+1}$. Now $\theta$ and $u$ were chosen to 
satisfy Lemma \ref{lem10.2} (vii) for $Y_{n+1}$, $\delta_{n+1}$, $\vp=\delta_n/6
\leq 2^{-(n+1)}/3$, and $\mu = 2^{-(n+2)}
$, and we now show that the same is true for $\theta_{n+1} = \text{Ad}(v'w) \circ
\theta$ and $u_{n+1}=v'wu$. Given a unitary $v\in M$ with $\|v-v'wu\|\le 1848\gamma^{1/2}$ and $\text{Ad}(v) 
\approx_{Y_{n+1},\delta_{n+1}} \text{Ad}(v'w)\circ\theta$, and given a finite subset 
$S$ of the unit ball of $H$, the unitary $w^*v^{\prime *}v$ satisfies
\begin{equation}\label{eq10.38a}
 \|w^*v^{\prime *}v-u\| \le 1848\gamma^{1/2}\quad\text{and}\quad \text{Ad}(w^*v^{\prime *}v) 
\approx_{Y_{n+1}, \delta_{n+1}} \theta.
\end{equation}
Let
\begin{equation}
 S' = S \cup \{w^*v^{\prime *}\xi \colon  \ \xi\in S\}.
\end{equation}
Applying Lemma \ref{lem10.2} (vii) to $S'$, there is a unitary $\tilde v\in \Bu $ 
satisfying $\|\tilde v-1_M\| \le 1848\gamma^{1/2}$,
\begin{equation}\label{eq10.39}
 \text{Ad}(\tilde vw^* v^{\prime *}v) \approx_{X_{n+1},\vp}\theta,
\end{equation}
and
\begin{equation}\label{eq10.40}
 \|(\tilde vw^*v^{\prime *}v-u)\eta\|, \ \|(\tilde vw^*v^{\prime *}v-u)^* \eta\| < 2^{-(n
+2)},\qquad \eta\in S'.
\end{equation}
Applying $\text{Ad}(v'w)$ to \eqref{eq10.39} gives
\begin{equation}\label{eq10.41}
 \text{Ad}(v'w\tilde vw^* v^{\prime *}v) \approx_{X_{n+1},\vp} \text{ Ad}(v'w) 
\circ \theta= \theta_{n+1}.
\end{equation}
From the first inequality of \eqref{eq10.40} we obtain
\begin{equation}\label{eq10.42}
 \|((v'w\tilde vw^*v^{\prime *})v-v'wu)\xi \| < 2^{-(n+2)},\qquad \xi\in S.
\end{equation}
For $\xi\in S$, put $\eta = w^*v^{\prime *}\eta\in S'$ into the second inequality of 
\eqref{eq10.40}, to yield
\begin{equation}\label{eq10.43}
 \|((v'w\tilde vw^*v^{\prime *})v-v'wu)^*\xi\| < 2^{-(n+2)},\qquad \xi\in S.
\end{equation}
The unitary $v'w\tilde vw^*v^{\prime *}\in B^\dagger$ then satisfies the requirements of condition 
(8)$_{n+1}$ since
\begin{equation}\label{eq10.44}
 \|v'w\tilde vw^*v^{\prime *}-1_M\| = \|\tilde v-1_M\|\le 1848\gamma^{1/2},
\end{equation}
and $\vp\leq 2^{-(n+2)}$. Thus condition (8)$_{n+1}$ holds and the proof is complete.
\end{proof}

In the next result we remove the hypothesis that $A$ and $B$ have the same ultraweak closure from Theorem \ref{thm10.3} and so establish Theorem \ref{Intro.Unitary}.  In the theorem below, the algebras $A$ and $B$ do not necessarily act non-degenerately so we use $\overline{A}^w$ and $\overline{B}^w$ for the ultraweak closures of $A$ and $B$ rather than $A''$ and $B''$.

\begin{theorem}\label{thm10.4}
Let $A$ and $B$ be  ${\mathrm C}^*$-algebras acting
on a separable Hilbert space $H$. Suppose that $A$ is separable and nuclear, and 
that $d(A,B)< 10^{-11}$. Then there exists a unitary $u\in (A\cup B)''$ such that $uAu^*=B$.
\end{theorem}

\begin{proof}
Since $d(A,B) < 1/101$, Propositions \ref{Prelim.NuclearOpen} and 
\ref{Prelim.SeparableOpen} show that $B$ is also separable and nuclear.
Choose $\eta$ so that $d(A,B)<\eta <10^{-11}$, and denote
the support projections of $A$ and $B$ by $e_A$ and $e_B$
respectively. These are the respective units of $\ovl{A}^w$
and $\ovl{B}^w$. By \cite[Lemma 5]{Kadison.Kastler} we have $d(\ovl{A}^w,
\ovl{B}^w)\leq
d(A,B)$, so from Proposition \ref{Prelim.Unit} there is a unitary $u_0\in {\mathrm W}^*(A,B,I_H)
$
such that $u_0e_Au_0^*=e_B$ and $\|1-u_0\|\leq
2\sqrt{2}\eta$. Let $A_0=u_0Au_0^*$. Then
\begin{equation}
d(A_0,A)\leq 2\|1-u_0\|\leq 4\sqrt{2}\eta,
\end{equation}
so  
\begin{equation}   
d(A_0,B)\leq (4\sqrt{2}+1)\eta< 1/8.
\end{equation}
Another use of \cite[Lemma 5]{Kadison.Kastler} shows that the same estimate holds 
for
$d(\ovl{A}^w_0,\ovl{B}^w)$. These injective von Neumann
algebras have the same unit $e_B$, so we can use Proposition \ref{Prelim.Inject} to
obtain a unitary $v\in {\mathrm W}^*(A,B,I_H)$ so that
$v\ovl{A}^w_0v^*=\ovl{B}^w$ and
\begin{equation}
\|I_H-v\|\leq 12d(\ovl{A}_0^w,\ovl{B}^w)\leq
(48\sqrt{2}+12)\eta.
\end{equation}
If we define $w=vu_0$, then $wAw^*$ and $B$ have identical
ultraweak closures, and
\begin{equation}
\|I_H-w\|=\|v^*-u_0\|\leq \|I_H-v^*\|+\|I_H-u_0\|\leq
(50\sqrt{2}+12)\eta.
\end{equation}
Now define $A_1=wAw^*$. Then
\begin{equation}
d(A_1,B)\leq d(A_1,A)+d(A,B)< 2\|I_H-w\| +\eta
<(100\sqrt{2}+25)\eta <10^{-8},
\end{equation}
since $\eta < 10^{-11}$.

Now let $K$ be the range of $e_B$ and restrict $A_1$ and $B$
to this Hilbert space. The hypotheses of Theorem \ref{thm10.3}
are now met, so there exists a unitary $u_1\in {\mathrm W}^*(A_1\cup B)={\mathrm W}^*(B)\subseteq \mathbb B(K)$
so that $u_1A_1u_1^*=B$. We extend $u_1$ to a unitary $u_2\in
B''$ by $u_2=u_1+(1-e_B)$, so that $u_2A_1u_2^*=B$. Then
$u_2wAw^*u_2^*=B$, and the proof is completed by defining
$u=u_2w\in {\mathrm W}^*(A,B,I_H)$.
\end{proof}

\begin{corollary}
Let $A$ be a separable nuclear ${\mathrm C}^*$-algebra on a separable Hilbert space $H$. 
Then the connected component of $A$ for the metric $d(\cdot,\cdot)$ is
$$
V=\{B\colon B=uAu^*,\ u\in {\mathcal U}({\mathbb B}(H))\}.
$$
\end{corollary}
\begin{proof}
Each unitary $u$ may be written $u=e^{ih}$ for a self-adjoint operator $h$, so $A$ 
is connected to $uAu^*$ by the path $t\mapsto e^{ith}Ae^{-ith}$ for $0\leq t\leq 1$. 
Thus $V$ is contained in the connected component.

If $D \in V^c$, then the open ball of radius $10^{-11}$ centred at $D$ must lie in $V^c$, 
otherwise there exists $B\in V$ with $d(B,D)<10^{-11}$. If this were the case then, 
by Theorem \ref{thm10.4}, $D$ would be unitarily equivalent to $B$ and thus to $A
$, placing $D\in V$ and giving a contradiction. Thus $V$ is closed, and it is also 
open by another application of Theorem \ref{thm10.4}. The result follows.
\end{proof}

\section{Near inclusions and nuclear dimension}\label{Near}

In this section we return to near inclusions
$A\subseteq_\gamma B$, where $A$ is nuclear and separable,
and we study the problem of whether $A$ embeds into $B$ for
sufficiently small values of $\gamma$. For this we will use
Lemma \ref{Close.Iso.Tech}, so the question reduces to finding cpc maps
from $A$ to $B$ which closely approximate the inclusion map of $A$ into the underlying ${\mathbb B}(H)$
on finite subsets of the unit ball of $A$. Making use of the
nuclearity of $A$ to approximately factorise id$_A$ through
matrix algebras, we see that the core question is this: do
cpc maps $\theta :  \mathbb M_n\to A$ perturb to nearby
cpc maps $\tilde\theta  :  \mathbb M_n\to B$? The obvious
approach is to use the well known identification of $\theta$
with a positive element of the ball of radius $n$ of $\mathbb M_n(A)$, approximate
this by a positive element of $\mathbb M_n(B)$ and take the
associated cp map $\tilde\theta  :  \mathbb M_n\to B$.
However, we will lose control of $\|\theta-\tilde\theta\|$
which will depend on $n$, forcing us to employ other
methods. We do not know the answer in full generality, but
will be able to give a positive solution for the class of
order zero maps, defined in Definition \ref{Near.DefOrderZero}.  This will
enable us to pass from a near inclusion $A\subseteq_\gamma
B$ to an embedding of $A$ into $B$ whenever $A$ has finite
nuclear dimension (see Theorem \ref{Near.Embedd} which is the
quantitative version of Theorem \ref{Intro.Near} in the introduction).
  The nuclear dimension (see Definition 
\ref{Near.Prelim.ND} below) of a ${\mathrm C}^*$-algebra, like its forerunner the 
decomposition rank \cite{Winter.CoveringDimension}, is defined by requiring the 
existence of suitable cpc approximate point-norm factorisations $\phi :  A\to F$ and $\psi :  F\to A$ for $\id_A$ through finite dimensional ${\mathrm C}^*$-algebras $F$ with the map $\psi$  splitting 
as a finite sum of  order zero maps.  We note that in all classes of nuclear separable
${\mathrm C}^*$-algebras
which have so far been classified, the constituent algebras
have finite nuclear dimension. We begin by
defining order zero maps; all such maps in this paper will
be contractions so we absorb this into the definition.

\begin{definition}\label{Near.DefOrderZero}
Let $F$ and $A$ be ${\mathrm C}^*$-algebras.  An \emph{order zero map} is a cpc map $\phi:F
\rightarrow A$ which preserves orthogonality in the following sense: if $e,f\in F^+$ and $ef=0$, then $\phi(e)\phi(f)=0$.
\end{definition}

Order zero maps are more general than $*$-homomorphisms, but nevertheless have many 
pleasant structural properties.  In particular, we will use the 
following two facts from \cite{Winter.NuclearCoveringDimension} and 
\cite{Winter.NuclearCoveringDimension2} (see also \cite{Winter.OrderZero}, where 
these assertions are established when $F$ is not finite dimensional).
\begin{proposition}\label{Near.OZ.Hm}
Let $A$ and $F$ be ${\mathrm C}^*$-algebras with $F$ finite dimensional and let $\phi:F\rightarrow 
A$ be an order zero map.  Suppose that $A$ is faithfully represented on $H$. Then 
there exists a unique $*$-homomorphism $\pi:F\rightarrow \overline{\phi(F)}^w  
\subseteq \overline{A}^w$ such that
\begin{equation}
\phi(x)=\pi(x)\phi(1_F)=\phi(1_F)\pi(x),\quad x\in F.
\end{equation}
\end{proposition}

\begin{proposition}\label{Near.OZ.Cone}
Let $A$ and $F$ be ${\mathrm C}^*$-algebras with $F$ finite dimensional.  Given an order zero map 
$\phi:F\rightarrow A$, the map $\id_{(0,1]}\otimes x\mapsto\phi(x)$ induces a 
$*$-homomorphism $\rho_\phi:C_0(0,1]\otimes F\rightarrow A$. Conversely, given 
a $*$-homomorphism $\rho:C_0(0,1]\otimes F\rightarrow A$, there is an order zero
map $\rho_\phi :  F\rightarrow A$ defined by  $x\mapsto\rho(\id_{(0,1]}\otimes x)$. 
\end{proposition}

We can now perturb order zero maps.  In the theorem below we do not obtain that 
$\psi$ is automatically order zero, but will address this point in Theorem 
\ref{OZ.Nose2}.
\begin{theorem}\label{Near.Perturb}
Let $A$ be a ${\mathrm C}^*$-algebra on a Hilbert space $H$.  Given a finite dimensional 
${\mathrm C}^*$-algebra $F$ and an order zero map $\phi:F\rightarrow A$, there exists a finite set 
$Y$ in the unit ball of $A$ with the following property.  If $B$ is another 
${\mathrm C}^*$-algebra on $H$ with $Y\subseteq_\gamma B$ for some $\gamma>0$, then 
there exists a cp map $\psi:F\rightarrow B$ with
\begin{equation}
\|\phi-\psi\|_{\mathrm{cb}}\leq (2\gamma+\gamma^2)(2+2\gamma+\gamma^2).
\end{equation}
\end{theorem}

\begin{proof}
Let $\phi:F\rightarrow A$ be cpc and order zero.  By replacing $A$ by the image of 
the $*$-homomorphism $\rho_\phi:C_0(0,1]\otimes F\rightarrow A$ from 
Proposition \ref{Near.OZ.Cone} we may assume that $A$ is nuclear. Let $\pi:F
\rightarrow \overline{\phi(F)}^w\subseteq \overline{A}^w$ be the unique $*$-homomorphism with
\begin{equation}\label{Near.Perturb.1}
\phi(x)=\pi(x)\phi(1_F)=\phi(1_F)\pi(x),\quad x\in F,
\end{equation}
given by Proposition \ref{Near.OZ.Hm}.  Write $F=\mathbb M_{n_1}\oplus\dots
\oplus\mathbb M_{n_r}$. Without loss of generality, we may assume that $\pi$ is 
injective, as if $M_{n_k}\subseteq\ker(\pi)$, then (\ref{Near.Perturb.1}) ensures that $
\phi|_{\mathbb M_{n_k}}=0$ allowing us to remove the $\mathbb M_{n_k}$ 
summand from $F$.  For each $1\leq k\leq r$, let ${\big{\{}}e^{(k)}_{i,j}\big\}_{i,j=1}^{n_k}$ be a 
system of matrix units for $\mathbb M_{n_k}$.  Let $m=\max\{n_1,\dots,n_r\}$ and 
let $\{f_{i,j}\}_{i,j=1}^m$ be a system of matrix units for $\mathbb M_m$.  For each 
$1\leq k\leq r$, define a non-unital $*$-homomorphism $\theta_k:\mathbb M_{n_k}
\rightarrow\mathbb M_m$ by $\theta_k\left(e_{i,j}^{(k)}\right)=f_{i,j}$.  Define a non-unital $*
$-homomorphism $\theta:F\rightarrow\mathbb M_r(\mathbb B(H)\otimes\mathbb 
M_m)$ by
\begin{equation}
\theta(x_1\oplus\dots\oplus x_r)=\begin{pmatrix}I_H\otimes\theta_1(x_1)&0&\cdots&0\\0&I_H\otimes\theta_2(x_2)&\ddots&\vdots\\\vdots&\ddots&\ddots&0\\0&
\cdots&0&I_H\otimes\theta_r(x_r)\end{pmatrix}.
\end{equation}

For each $1\leq k\leq r$, define self-adjoint partial isometries 
\begin{equation}
s_k=\sum_{i,j=1}^{n_k}\pi\big(e^{(k)}_{i,j}\big)\otimes f_{j,i}\in \pi(F)\otimes \mathbb M_m.
\end{equation}
These satisfy
\begin{equation}\label{Near.Perturb.2}
s_k(I_H\otimes\theta_k(x_k))s_k^*=\pi(x_k)\otimes \sum_{i=1}^{n_k} f_{i,i},\quad x_k\in 
\mathbb M_{n_k}.
\end{equation}
By \eqref{Near.Perturb.1}
\begin{equation}
(\phi(1_F)\otimes 1_{\mathbb M_m})s_k=\sum_{i,j=1}^{n_k}\phi\big(e_{i,j}^{(k)}\big)\otimes 
f_{j,i}=(\phi(1_{\mathbb M_{n_k}})\otimes 1_{\mathbb M_m})s_k\in A\otimes \mathbb M_m.
\end{equation}
The continuous functional calculus then shows that $f(\phi(1_{M_{n_k}})\otimes 1_{\mathbb M_m})s_k\in A\otimes\mathbb M_m$, whenever $f$ is a 
continuous function on $[0,1]$ with $f(0)=0$.  In particular, $t_k=(\phi(1_{\mathbb 
M_{n_k}})^{1/2}\otimes 1_{\mathbb M_m})s_k\in A\otimes \mathbb M_m$ for all $k$.  Let $t\in
\mathbb M_{1\times r}(A\otimes\mathbb M_m)$ be the row matrix $(t_1,\dots, t_r)$.  
We can then define a completely positive map $\phi_0:F\rightarrow A\otimes
\mathbb M_m$ by
\begin{equation}
\phi_0(x)=t\theta(x)t^*,\quad x\in F.
\end{equation}
For $x=x_1\oplus\dots\oplus x_r\in F$, use (\ref{Near.Perturb.2}) to compute
\begin{align}
\phi_0(x)=\sum_{k=1}^rt_k(I_H\otimes \theta_k(x_k))t_k^*=&
\sum_{k=1}^r(\phi(1_{\mathbb M_{n_k}})^{1/2}\otimes 1)\big(\pi(x_k)\otimes 
\sum_{i=1}^{n_k}f_{i,i}\big)(\phi(1_{\mathbb M_{n_k}})^{1/2}\otimes 1)\notag\\
=&\sum_{k=1}^r\big(\phi(x_k)\otimes \sum_{i=1}^{n_k}f_{i,i}\big).
\end{align}
In particular, under the identification $\mathbb B(H)\cong f_{1,1}(\mathbb B(H)
\otimes \mathbb M_m)f_{1,1}$, we can recover $\phi$ by
\begin{equation}
\phi(x)=(I_H\otimes f_{1,1})\phi_0(x)(I_H\otimes f_{1,1}),\quad x\in F.
\end{equation}

Now
\begin{align}
\|t\|^2=&\|tt^*\|=\|\sum_{k=1}^r(\phi(1_{\mathbb M_{n_k}})^{1/2}\otimes 
1_{\mathbb M_m})s_ks_k^*(\phi(1_{\mathbb M_{n_k}})^{1/2}\otimes 1_{\mathbb M_m})\|\notag\\
\leq&\|\sum_{k=1}^n\phi(1_{\mathbb M_{n_k}})\|=\|\phi(1_F)\|=\|\phi\|\leq 1
\end{align}
so $\{t\}$ is a finite subset of the unit ball of $\mathbb M_{1\times r}(A\otimes 
\mathbb M_m)$.  Accordingly Proposition \ref{Near.Length} gives a finite subset $Y
$ of the unit ball of $A$ with the property that whenever $B$ is another ${\mathrm C}^*$-
algebra on $H$ and $Y\subseteq_\gamma B$, then $\{t\}\subseteq_{\mu}\mathbb 
M_{1\times r}(B\otimes\mathbb M_m)$, where $\mu=2\gamma+\gamma^2$.  
Assume we are given such a ${\mathrm C}^*$-algebra $B$ so that we can find some 
$u=(u_1,\dots,u_r)\in \mathbb M_{1\times r}(B\otimes \mathbb M_m)$ with $\|t-u\|
\leq\mu$ in $\mathbb M_{1\times r}(\mathbb B(H)\otimes \mathbb M_m)$.

Define a cp map $\psi_0:F\rightarrow B\otimes \mathbb M_m$ by $\psi_0(x)=u
\theta(x)u^*$.  Use the identification $\mathbb B(H)\cong f_{1,1}(\mathbb B(H)
\otimes \mathbb M_m)f_{1,1}$ to define a cp map $\psi:F\rightarrow B$ by 
\begin{equation}
\psi(x)=f_{1,1}\psi_0(x)f_{1,1},\quad x\in F.
\end{equation}
Finally
\begin{equation}
\|\phi-\psi\|_{\text{cb}}\leq \|\phi_0-\psi_0\|_{\text{cb}}\leq\|t-u\|\|t\|+\|t-u\|\|u\|\leq \mu +\mu(1+\mu),
\end{equation}
exactly as claimed.
\end{proof}

\begin{corollary}
Let $A$ and $B$ be ${\mathrm C}^*$-algebras on a Hilbert space $H$ with  $A
\subseteq_{\gamma}B$ for some $\gamma>0$. Then for each finite dimensional C$^*
$-algebra $F$ and order zero  map $\phi:F\rightarrow A$, there exists a cp map 
$\psi:F\rightarrow B$ satisfying 
\begin{equation}
\|\phi-\psi\|_{\mathrm{cb}}\leq (2\gamma+\gamma^2)
(2+2\gamma+\gamma^2).
\end{equation}
\end{corollary}

We now work towards our embedding result for a near containment of a separable 
${\mathrm C}^*$-algebra of finite nuclear dimension. First let us recall the definition of nuclear 
dimension from \cite{Zacharias.NuclearDimension}.
\begin{definition}\label{Near.Prelim.ND}
Let $A$ be a ${\mathrm C}^*$-algebra and $n\geq 0$.  Say that $A$ has \emph{nuclear 
dimension} at most $n$, written $\dim_{{\mathrm{nuc}}}(A)\leq n$, if for each finite subset 
$X$ of $A$ and $\vp>0$, there exists a finite dimensional ${\mathrm C}^*$-algebra $F$ 
which decomposes as  a direct sum $F=F_0\oplus\dots\oplus F_n$ and maps $
\phi:A\rightarrow F$ and $\psi:F\rightarrow A$ such that $\psi\circ\phi\approx_{X,
\vp}\id_A$, $\phi$ is cpc and $\psi$ decomposes as $\psi=\sum_{i=0}^n\psi_i$, where each $\psi_i: F_i\to A$ is  order zero.
\end{definition}

The definition of nuclear dimension is a modification of the \emph{decomposition 
rank} from \cite{Winter.CoveringDimension}. The decomposition rank $\text{dr}(A)$ 
of $A$ is defined in the same way as the nuclear dimension, but with the additional 
requirement that the map $\psi$ in the definition above is also cpc.  Suprisingly the 
small change in the definition from decomposition rank to nuclear dimension considerably 
enlarges the class of ${\mathrm C}^*$-algebras with finite dimension (while retaining the 
permanence properties). Indeed in \cite{Winter.CoveringDimension} it is shown that 
a separable ${\mathrm C}^*$-algebra with finite decomposition rank is necessarily 
quasidiagonal and so stably finite, while in \cite{Zacharias.NuclearDimension} it is 
shown that the Cuntz algebras $\mathcal O_n$ (and all classifiable Kirchberg 
algebras) have finite nuclear dimension.  We need one final structural property of 
the cp approximations defining nuclear dimension (see \cite[Remark~2.2 (iv)]
{Zacharias.NuclearDimension}); this  is immediate in the unital case.  

\begin{proposition}\label{Near.NDProp}
Suppose that $A$ is a ${\mathrm C}^*$-algebra with $\dim_{{\mathrm{nuc}}}(A)\leq n$.  Given a 
finite set $X\subseteq A$ and $\varepsilon>0$, there exist $F$ and maps $\phi$ and $
\psi$ as in Definition \ref{Near.Prelim.ND} with the additional property that $\psi\circ
\phi$ is cpc.
\end{proposition}

Given a near inclusion $A\subseteq_\gamma B$, where $A$ has finite nuclear 
dimension, we can now  approximate in the point-norm topology the inclusion of A into the underlying $\mathbb B(H)$ by cpc maps $A\rightarrow B$.  
\begin{lemma}\label{Near.DirectLem}
Let $n\geq 0$.  Let $D$ be a ${\mathrm C}^*$-algebra with $\dim_{{\mathrm{nuc}}}(D)\leq n$, let $A
$ be a ${\mathrm C}^*$-algebra represented on the Hilbert space $H$ and let $\theta:D
\rightarrow A$ be an order zero map. Given a finite subset $X$ of the unit ball of $D
$ and $\vp>0$, there exists another finite subset $Y$ of the unit ball of $A$ with 
the following property.  If $B$ is another ${\mathrm C}^*$-algebra on $H$ with $Y
\subseteq_\gamma B$ for some $\gamma>0$,  then there exists a cpc map $\phi:D
\rightarrow B$ with 
\begin{equation}\label{Near.DirectLem.1}
\|\phi(x)-\theta(x)\|\leq2(n+1)(2\gamma+\gamma^2)(2+2\gamma+\gamma^2)+
\vp,\quad x\in X.
\end{equation}
\end{lemma}
\begin{proof}
Given a finite subset $X$ of the unit ball of $D$ and $\varepsilon>0$, we first use the 
definition of nuclear dimension in conjuction with Proposition \ref{Near.NDProp} to 
find a finite dimensional ${\mathrm C}^*$-algebra $F$ and cp maps $\psi_1:D\rightarrow F$ 
and $\psi_2:F\rightarrow D$ such that
\begin{itemize}
\item[\rm (i)] $\psi_1$ is cpc;
\item[\rm (ii)] $F$ decomposes as $F_0\oplus\dots\oplus F_n$ and $\psi_2$ 
decomposes as  $\psi_2=\sum_{i=0}^n\psi_{2,i}$, where each $\psi_{2,i}: F_i\to D$ is 
cpc and order zero;
\item[\rm (iii)] $\psi_2\circ\psi_1$ is contractive and $\psi_2\circ\psi_1\approx_{X,\varepsilon}\id_D$.
\end{itemize}
Theorem \ref{Near.Perturb} enables us to find a finite set $Y$ in the unit ball of $A$ 
such that whenever $B$ is another ${\mathrm C}^*$-algebra on $H$ with $Y\subseteq_
\gamma B$, we can find cp maps $\tilde{\psi}_{2,i}:F_i\rightarrow B$ with 
\begin{equation}
\|\theta\circ\psi_{2,i}-\tilde{\psi}_{2,i}\|_{\text{cb}}\leq (2\gamma+\gamma^2)
(2+2\gamma+\gamma^2),
\end{equation}
for $i=0,\dots,n$.  This is the set $Y$ required by the lemma as given such maps, 
we can define a cp map $\tilde{\psi}_2:F\rightarrow B$ by $\sum_{i=0}^n\tilde{\psi}_{2,i}$ and this satisfies
\begin{equation}
\|\theta\circ\psi_2-\tilde{\psi}_2\|_\text{cb}\leq (n+1)(2\gamma+\gamma^2)
(2+2\gamma+\gamma^2).
\end{equation}
We define a cp map $\tilde{\phi}:D\rightarrow B$ by $\tilde{\phi}=\tilde{\psi}_2\circ
\psi_1$.  Now
\begin{equation}
\|\tilde{\phi}-\theta\circ\psi_2\circ\psi_1\|_\text{cb}\leq (n+1)(2\gamma+\gamma^2)
(2+2\gamma+\gamma^2)
\end{equation}
so
\begin{equation}
\|\tilde{\phi}\|_\text{cb}\leq 1+(n+1)(2\gamma+\gamma^2)(2+2\gamma+\gamma^2).
\end{equation}
If $\tilde{\phi}$ is already cpc, we define $\phi=\tilde{\phi}$.  Otherwise define $\phi=
\tilde{\phi}/\|\tilde{\phi}\|_\text{cb}$.  Thus
\begin{equation}
\|\phi-\tilde{\phi}\|_\text{cb}\leq\|\tilde{\phi}\|_\text{cb}-1\leq(n+1)(2\gamma+
\gamma^2)(2+2\gamma+\gamma^2).
\end{equation}
For $x\in X$,
\begin{align}
\|\phi(x)-\theta(x)\|&\leq\|\phi(x)-\tilde{\phi}(x)\|+\|\tilde{\phi}(x)-(\theta\circ\psi_2\circ
\psi_1)(x)\|+\|(\theta\circ\psi_2\circ\psi_1)(x)-\theta(x)\|\nonumber\\
&\leq 2(n+1)(2\gamma+\gamma^2)(2+2\gamma+\gamma^2)+\vp,
\end{align}
establishing the result.
\end{proof}

Taking $\theta$ to be the identity map on $A$, the next corollary is immediate, since if 
$A\subset_\gamma B$, then we can find $\gamma'<\gamma$ with $A
\subseteq_{\gamma'}B$ and take 
\begin{equation}
\vp=2(n+1)[(2\gamma+\gamma^2)(2+2\gamma+\gamma^2)-(2\gamma'+\gamma'^2)(2+2\gamma'+\gamma'^2)]
\end{equation}
 in the previous 
lemma.
\begin{corollary}\label{Near.Cpc}
Let $A$ and $B$ be ${\mathrm C}^*$-algebras on a Hilbert space $H$ with $A\subset_\gamma B$ for 
some $\gamma>0$ and suppose that $\dim_{\mathrm{nuc}}(A)\leq n$ for some $n\geq 
0$.  For any finite subset $X$ of the unit ball of $A$, there exists a cpc map $\phi:A
\rightarrow B$ with
\begin{equation}
\|\phi(x)-x\|\leq2(n+1)(2\gamma+\gamma^2)(2+2\gamma+\gamma^2),\quad x\in X.
\end{equation}
\end{corollary}

The intertwining argument of Lemma \ref{Close.Iso.Tech} combines with the previous corollary to give immediately  the following quantitative version of Theorem \ref{Intro.Near}. The number 20 appearing in \eqref{eq6.2000} is an integer estimate for $8\sqrt{6}$ from Lemma \ref{Close.Iso.Tech}.
\begin{theorem}\label{Near.Embedd}
Let $n\geq 0$. Let $A\subseteq_\gamma B$ be a near inclusion of ${\mathrm C}^*$-algebras on a Hilbert 
space $H$, let $\eta=2(n+1)(2\gamma+\gamma^2)(2+2\gamma+
\gamma^2)$, and suppose that $A$ is separable with nuclear dimension at most 
$n$.  If $\eta<1/210000$, then $A$ embeds into $B$.  Moreover, 
for each finite subset $X$ of the unit ball of $A$, there exists an embedding $
\theta:A\rightarrow B$ with
\begin{equation}\label{eq6.2000}
\|\theta(x)-x\|\leq20\eta^{1/2}.
\end{equation}
\end{theorem}

\begin{remark}
The hypotheses on $A$ in the previous theorem  are, in particular, satisfied for all 
separable simple nuclear C$^{*}$-algebras presently covered by known 
classification theorems. This includes Kirchberg algebras satisfying the UCT, simple 
unital  C$^{*}$-algebras with finite decomposition rank for which projections 
separate traces (and  also satisfying the UCT), and transformation group 
C$^{*}$-algebras associated to compact minimal uniquely ergodic dynamical systems (see 
\cite{Zacharias.NuclearDimension} and \cite{Toms.MinimalDynamicsPNAS}).  
\end{remark}

\section{Applications}\label{Direct}

In \cite{Bratteli.AF} Bratteli initiated the study of separable approximately finite 
dimensional (AF) ${\mathrm C}^*$-algebras, namely those separable ${\mathrm C}^*$-algebras arising 
as direct limits of finite dimensional ${\mathrm C}^*$-algebras.  He gave a local 
characterisation of these algebras, showing that a separable ${\mathrm C}^*$-algebra is AF if, 
and only if, for each  $\varepsilon>0$ and each finite set $X\subseteq A$, there exists a 
finite dimensional ${\mathrm C}^*$-subalgebra $A_0$ of $A$ such that $X
\subset_{\varepsilon}A_0$. By changing $\varepsilon$ if necessary, we can scale 
the set $X$ above so that it lies in the unit ball of $A$. This characterisation can be 
weakened: it is not necessary to be able to approximate a finite set of the unit ball 
of $A$ \emph{arbitrarily} closely by a finite dimensional ${\mathrm C}^*$-algebra; an approximation  up to a fixed small tolerance is sufficient to imply that $A$ is AF. The 
proposition below states this precisely and is implicit in the proof of \cite[Theorem 
6.1]{Christensen.NearInclusions}.  
\begin{proposition}
There exists a constant $\gamma_0>0$ with the following property. Let $A$ be a 
separable ${\mathrm C}^*$-algebra and suppose that for all finite sets $X$ in the unit ball of 
$A$, there exists a finite dimensional ${\mathrm C}^*$-subalgebra $A_0$ of $A$ such that $X
\subset_{\gamma_0}A_0$. Then $A$ is AF.
\end{proposition}

Our first objective in this section is to generalise this last result to other inductive 
limits, which admit a local characterisation.  In \cite{Elliott.RealRankZero}, Elliott 
gave a local characterisation of the separable $A\mathbb T$-algebras (those C$^*
$-algebras arising as direct limits of algebras of the form $C(\mathbb T)\otimes F$, 
where $F$ is finite dimensional).  Loring developed a theory of finitely presented 
${\mathrm C}^*$-algebras and showed that local characterisations are possible for inductive 
limits of finitely presented weakly semiprojective ${\mathrm C}^*$-algebras. There are many 
examples of such algebras, including the  dimension drop intervals used in 
\cite{Jiang.Z}.  We refer to Loring's monograph \cite{Loring.LiftingBook} for more 
examples and background information on these concepts (see also \cite{EK}).  The proposition below is 
Lemma 15.2.2 of \cite{Loring.LiftingBook}, scaling the finite sets involved into the 
unit ball.
\begin{proposition}[Loring]
Suppose that $A$ is a ${\mathrm C}^*$-algebra containing a (not necessarily nested) sequence of 
 ${\mathrm C}^*$-subalgebras $A_n$ with the property that for each finite set $X$ of the unit 
ball of $A$ and each $\varepsilon>0$, there exists $n\in\mathbb N$ with $X\subseteq_
\varepsilon A_n$.  If each $A_n$ is weakly semiprojective and finitely presented, 
then $A$ is isomorphic to a direct limit $\displaystyle{\lim_{\rightarrow}}(A_{k_n},
\phi_n)$ for some subsequence $\{k_n\}_{n=1}^\infty$ and some connecting 
$*$-homomorphisms $\phi_n:A_{k_n}\rightarrow A_{k_{n+1}}$.
\end{proposition}
Provided that the building blocks $A_n$ are all nuclear, the arbitrary tolerance $
\varepsilon>0$ appearing above can be replaced by the fixed quantity $1/120000$. The only 
fact we require about finitely presented weakly semiprojective ${\mathrm C}^*$-algebras is the 
following easy proposition, which is immediate from the definition of weak stability of 
a finite presentation of a ${\mathrm C}^*$-algebra.
\begin{proposition}\label{Direct.WSP.ApproxHm}
Let $A$ be a finitely presented, weakly semiprojective ${\mathrm C}^*$-algebra. Then, for each
finite subset $X$ of the unit ball of $A$ and each $\varepsilon>0$, there exists a finite 
subset $Y$ of the unit ball of $A$ and $\delta>0$ with the following property. If $\phi:A\rightarrow B$ is 
a $(Y,\delta)$-approximate $*$-homomorphism, then there is a 
$*$-homomorphism $\psi:A\rightarrow B$ with $\psi\approx_{X,\varepsilon}\phi$.
\end{proposition}

In the proposition above, the choice of $\delta$  depends on both $X$ and $
\varepsilon$.  Indeed, one proves the proposition by replacing $X$ by a weakly stable generating set for $A$, reducing $\varepsilon$ if necessary, and taking $Y=X$. The result follows since the image of $Y$ under a $(Y,\delta)$-approximate $*$-homomorphism is an $\eta$-representation for the presentation $Y$ (where $\eta\rightarrow 0$ as $\delta\rightarrow 0$).  When $A$ is nuclear,  Lemma \ref{Avg.1} can be used to show that the 
$\delta$ appearing in Proposition \ref{Direct.WSP.ApproxHm} only depends on $\varepsilon$ and not on $X$ or $A$. This approach results in  enlarging the set $Y$.
\begin{lemma}\label{Direct.Easy}
Fix $\varepsilon>0$ and let $\delta<\min(1/17,\varepsilon^2/128)$.   Suppose 
that $A$ is a finitely presented, weakly semiprojective nuclear ${\mathrm C}^*$-algebra. Then 
for each finite subset $X$ of the unit ball of $A$, there exists a finite subset $Y$ of 
the unit ball of $A$ with the following property. If $\phi:A\rightarrow B$ is a $(Y,\delta)$-approximate 
$*$-homomorphism, then there is a $*$-homomorphism $\psi:A\rightarrow B$ 
with $\psi\approx_{X,\varepsilon}\phi$.
\end{lemma}
\begin{proof}
Fix $\varepsilon>0$ and take $0<\delta<1/17$ such that $8\sqrt{2}\delta^{1/2}<\varepsilon
$. Fix $\mu>0$ so that $8\sqrt{2}\delta^{1/2}+\mu<\varepsilon$ and write $\varepsilon'=
\varepsilon-8\sqrt{2}\delta^{1/2}-\mu>0$. Given a finitely presented, weakly 
semiprojective nuclear ${\mathrm C}^*$-algebra $A$ and a finite set $X$ in the unit ball of $A
$, use Proposition \ref{Direct.WSP.ApproxHm} to find $\delta'>0$ and a finite set $Z
$ in the unit ball of $A$ such that if $\phi_1:A\rightarrow B$ is a 
$(Z,\delta')$-approximate $*$-homomorphism, then there exists a $*$-homomorphism $\psi:A
\rightarrow B$ with $\psi\approx_{X,\varepsilon'}\phi_1$.  By Lemma \ref{Avg.1}, 
there is a finite set $Y$ in the unit ball of $A$ such that given any 
$(Y,\delta)$-approximate $*$-homomorphism $\phi:A\rightarrow B$, there is a 
$(Z,\delta')$-approximate $*$-homomorphism $\phi_1:A\rightarrow B$ with $\|\phi-\phi_1\|\leq 
8\sqrt{2}\delta^{1/2}+\mu$. It follows that $\delta$ has the property claimed in the 
lemma.
\end{proof}

We now show that we do not need  arbitrarily close approximations in 
order to detect direct limits of finitely presented weakly semiprojective nuclear C$^*
$-algebras. 
\begin{theorem}\label{Direct.Main}
Let $A$ be a separable ${\mathrm C}^*$-algebra and suppose that there is a (not necessarily) 
nested sequence $\{A_k\}_{k=1}^\infty$ of finitely presented, weakly semiprojective 
nuclear ${\mathrm C}^*$-subalgebras of $A$ with the following property. For each finite subset $X$ of the unit ball 
of $A$, there exists $k\in\mathbb N$ such that $X\subset_{\eta} A_k$ for some $\eta$ 
satisfying $\eta<1/120000$. Then $A$ is isomorphic to a direct limit $\displaystyle{\lim_{\rightarrow}}(A_{k_n},
\phi_{n+1})$ for some subsequence $\{k_n\}$ and some connecting 
$*$-homomorphisms $\phi_{n+1}:A_{k_n}\rightarrow A_{k_{n+1}}$.
\end{theorem}
\begin{proof}
Fix a dense sequence $\{a_i\}_{i=1}^\infty$ of the unit ball of $A$.  For each $n\geq 
1$ we will construct $k_n\in\mathbb N$, a dense sequence $\big\{a_i^{(n)}\big\}_{i=1}^
\infty$ in the unit ball of $A_{k_n}$ and a unitary $u_n$ in the unitisation ${A^{\dagger}}$. For $n>1$, we 
will define connecting $*$-homorphisms $\phi_n:A_{k_{n-1}}\rightarrow A_{k_n}$. 
These objects will satisfy the following properties:
\begin{enumerate}
\item $\|a_i-u_1\dots u_{n-1}a_i^{(n)}u_{n-1}^*\dots u_1^*\|<1/10$ for all $n\geq 1$ 
and $i=1,\dots,n$.
\item For $n>1$, $\|u_n(\phi_n(x))u_n^*-x\|<2^{-n}$ whenever $x\in A_{k_{n-1}}$ is 
of the form 
\begin{equation}
(\phi_{n-1}\circ\dots\circ\phi_j)(a_i^{(j-1)}),\ \ 1\leq i\leq n,\ \ 2\leq j \leq n-1.
\end{equation} 
\item $\|u_n-1_{{A}^{\dagger}}\|\leq 2/5$.
\end{enumerate}

Once the induction is complete, the second condition above ensures that the 
following diagram gives an approximate intertwining in the sense of \cite[2.3]
{Elliott.RealRankZero}, where $\iota_n:A_n\hookrightarrow A$ is the inclusion map.  
$$
\xymatrix{A_{k_1}\ar[rr]_{\phi_2}\ar[dd]_{\ad(u_1)\circ\iota_1=
\iota_{k_1}}&&A_{k_2}\ar[rr]_{\phi_3}\ar[dd]_{\ad(u_1u_2)\circ
\iota_{k_2}}&&A_{k_3}\ar[rr]\ar[dd]_{\ad(u_1u_2u_3)\circ\iota_{k_3}}&&{\dots}\\\\A
\ar[rr]_{\text{id}_A}&&A\ar[rr]_{\text{id}_A}&&A\ar[rr]&&{\dots}}
$$
In particular \cite[2.3]{Elliott.RealRankZero} produces a $*$-homomorphism $
\theta:\displaystyle{\lim_{\rightarrow}}(A_{k_n},\phi_{n+1})\rightarrow A$ as a 
point-norm limit. This map is injective since each of the vertical maps is injective.  For 
surjectivity, fix $n\in\mathbb N$.  Let $b_n$ denote the image of $a_n^{(n)}$ in $
\displaystyle{\lim_{\rightarrow}}(A_{k_n},\phi_{n+1})$ so that 
\begin{equation}
\theta(b_n)=\lim_{m\rightarrow\infty}u_1\dots u_m(\phi_m\circ\dots\circ\phi_{n+1})
(a^{(n)}_n)u_m^*\dots u_1^*
\end{equation}
in norm.  Repeatedly using the second condition, we have
\begin{equation}
\|\theta(b_n)-u_1\dots u_na_n^{(n)}u_n^*\dots u_1^*\|\leq \sum_{m>n}2^{-m},
\end{equation}
so that
\begin{align}
\|\theta(b_n)-a_n\|&\leq\sum_{m>n}2^{-m}+\|u_1\dots u_na^{(n)}_nu_n^*\dots 
u_1^*-a_n\|\notag\\
&\leq\sum_{m>n}2^{-m}+2\|u_n-1_{{A}^{\dagger}}\|+\|u_1\dots u_{n-1}a^{(n)}_nu_{n-1}^*
\dots u_1^*-a_n\|\notag\\
&\leq\sum_{m>n}2^{-m}+4/5+1/10
\end{align}
from conditions 1 and 3.  Hence $d(A,\theta(\displaystyle{\lim_{\rightarrow}}
(A_{k_n},\phi_{n+1})))\leq 9/10<1$ so that $\theta$ is surjective by Proposition 
\ref{EasyFact}.

We start the construction by using the hypothesis to find $k_1$ and a dense 
sequence $\big\{a^{(1)}_i\big\}_{i=1}^\infty$ so that condition $1$ holds.  We take 
$u_1=1_{{A}^{\dagger}}$ so condition 3 holds and at this first stage condition 2 is empty.  
Suppose that all objects have been constructed up to and including stage $n-1$ for 
some $n>1$.  Let $X$ be the finite subset of the unit ball of $A_{k_{n-1}}$ 
consisting of the elements $(\phi_{n-1}\circ\dots\circ\phi_j)\big(a_i^{(j-1)}\big)$ for $i\leq n$ 
and $j=2,\dots,n-1$.   By Lemma \ref{Avg.2} (with $\gamma=13/150$, $
\delta=1/50$ and $B$ any ${\mathrm C}^*$-algebra), there exists a finite subset $Y$ of the unit ball of $A_{k_{n-1}}$ such 
that, given any two $*$-homomorphisms $\psi_1,\psi_2:A_{k_{n-1}}\rightarrow B$ 
with $\psi_1\approx_{Y,13/150}\psi_2$, there is a unitary $u\in {B}^{\dagger}$ with $\|
u-1_{{B}^{\dagger}}\|\leq 2/5$ and $\ad(u)\circ\psi_1\approx_{X,2^{-n}}\psi_2$.  By enlarging $Y$, we may 
assume that $Y\supseteq X$. By Proposition \ref{Direct.Easy}, there is a finite 
subset $Z$ of the unit ball of $A_{k_{n-1}}$ such that if $\phi:A_{k_{n-1}}\rightarrow 
B$ is a $(Z,6\eta)$-approximate $*$-homomorphism for some $\eta$ satisfying 
$\eta<1/120000=\frac{1}{6}\left(\frac{2}{25}\right)^2\frac{1}{128}$, then there is a $*$-homomorphism $\psi:A_{k_{n-1}}
\rightarrow B$ with $\psi\approx_{Y,2/25}\phi$. Now use the hypothesis to find 
$k_n\in\mathbb N$ such that
\begin{equation}
Z\cup\{u_{n-1}^*\dots u_1^*a_ju_1\dots u_{n-1}\colon 1\leq j \leq n\}\subset_{\eta} A_{k_n}.
\end{equation}
Since $A_{k_n}$ is nuclear, Proposition \ref{Prelim.PtArveson} gives a cpc map $
\phi:A_{k_{n-1}}\rightarrow A_{k_n}$ with $\|\phi(z)-z\|\leq 2\eta$ for $z\in Y\cup Z
\cup Z^*\cup \{zz^*:z\in Z\cup Z^*\}$.  In particular such a map is a $(Z,6\eta)$-
approximate $*$-homomorphism and hence there is a $*$-homomorphism $
\phi_n:A_{k_{n-1}}\rightarrow A_{k_n}$ with $\phi_n\approx_{Y,2/25}\phi$. In 
particular
\begin{equation}
\|\phi_n(y)-y\|\leq\frac{2}{25}+2\eta\leq\frac{13}{150},\quad y\in Y,
\end{equation}
so our choice of $Y$ gives us a unitary $u_n\in {A}^{\dagger}$ with $\|u_n-1_{{A}^{\dagger}}\|
\leq2/5$ and 
\begin{equation}
\|(\ad(u_n)\circ\phi_n)(x)-x\|\leq 2^{-n},\quad x\in X,
\end{equation}
and conditions 2 and 3 hold.  Since $\{u_{n-1}^*\dots u_1^*a_ju_1\dots u_{n-1}\colon 1\leq j \leq n\}\subset_\eta A_{k_n}$ for some $
\eta<1/120000$, we may choose  a dense sequence $
\big\{a_i^{(n)}\big\}_{i=1}^\infty$ to fulfill condition 1. This completes the induction.
\end{proof}

As an example, Elliott's local characterisation of $A\mathbb T$-algebras from 
\cite{Elliott.RealRankZero} can be weakened to give the following statement. Note that the algebras $A_0$ below are all semiprojective by combining 
\cite[14.1.7, 14.1.8, 14.2.1, 14.2.2]{Loring.LiftingBook}. 
\begin{corollary}
Let $A$ be a separable ${\mathrm C}^*$-algebra.  Suppose that, for any finite set $X$ in the 
unit ball of $A$, there exists a ${\mathrm C}^*$-subalgebra $A_0$ of $A$ with 
$X\subset_{1/120000}A_0$, where $A_0$ has the form 
$C(\mathbb T)\otimes F_1\oplus C[0,1]\otimes F_2\oplus F_3$ for finite 
dimensional ${\mathrm C}^*$-algebras $F_1,F_2$ and $F_3$. 
Then $A$ is $A\mathbb T$.
\end{corollary}

In a similar fashion, we obtain a  generalised characterisation of the Jiang--Su 
algebra $\mathcal{Z}$, (see \cite[Theorems~2.9 and 6.2]{Jiang.Z}). The algebras $Z_{p,q}$
in the statement are semiprojective from \cite{ELP}. 
\begin{corollary}
Let $A$ be a unital, simple, separable ${\mathrm C}^*$-algebra with a unique tracial state.  
Suppose that for each finite subset $X$ of the unit ball of $A$, there exists a prime 
dimension drop C$^{*}$-algebra $Z_{p,q}$ with $X\subset_{1/120000} Z_{p,q}$. 
Then $A\cong \mathcal{Z}$.
\end{corollary}

We now return to our perturbation result for order zero maps (Theorem 
\ref{Near.Perturb}) and show that the resulting maps can also be taken of order 
zero. For simplicity, we establish this result for near inclusions rather than for the context of finite 
sets used in Theorem \ref{Near.Perturb}. 
In the next theorem we will make use of the  fact that a map $\phi$ whose domain is a finite dimensional  operator space $E$ is completely bounded with $\|\phi\|_{{\mathrm{cb}}}\leq ({\mathrm{dim}}\,E)\,\|\phi\|$, (see \cite{EH}).  
\begin{theorem}\label{OZ.Nose2}
Let $A\subset_\gamma B$ be a near inclusion of ${\mathrm C}^*$-algebras, where $\gamma$ satisfies
\begin{equation}\label{OZ.100}
0<\gamma<10^{-7}.
\end{equation}
Given a finite dimensional ${\mathrm C}^*$-algebra $F$ and an order zero 
map $\phi:F\rightarrow A$, there exists an order zero map $\psi:F\rightarrow B$ 
satisfying
\begin{equation}
\|\phi-\psi\|_{{\mathrm{cb}}}<493\gamma^{1/2}.
\end{equation}
\end{theorem}
\begin{proof}
Let $\gamma'<\gamma$ be such that $A\subseteq_{\gamma'} B$.
Denote the linear dimension of $F$ by $m$ and choose
$\beta>0$ so that $3m\beta<\gamma^{1/2}$. Let $X_0$ be a
$\beta$-net for the unit ball of $F$, and let
\begin{equation}
X=\{{\mathrm{id}}_{(0,1]}\otimes x \colon x\in
X_0\}\subseteq C_0(0,1]\otimes F.
\end{equation}
Now apply Lemma \ref{Avg.2} and Remark \ref{Avg.2.Rem} (ii) to the
nuclear ${\mathrm C}^*$-algebra $C_0(0,1]\otimes F$ with $\gamma$
replaced by $82 \gamma^{1/2}<13/150$, $\vp>0$ replaced by
$\beta >0$, and $D$ replaced by C$^*(A,B)$. Then there is a finite
subset $Z$ of the unit ball of  $C_0(0,1]\otimes F$ such that if
$\phi_1,\phi_2  :  C_0(0,1]\otimes F\to {\mathrm C}^*(A,B)$ are
$*$-homomorphisms satisfying $\phi_1
\approx_{Z,82\gamma^{1/2}}\phi_2$, then there exists a
unitary $u\in {\mathrm C}^*(A,B)^{\dagger}$ with
$\|u-1\|<246\gamma^{1/2}$ and
$\phi\approx_{X,\beta}{\mathrm{Ad}}(u)\circ \phi_2$.
Choose $\vp >0$ so that
\begin{equation}
(17\gamma'+\vp)^{1/2} +2\vp<\sqrt{17}\gamma^{1/2},
\end{equation}
and then define
\begin{align}
\eta&=4(2\gamma'+\gamma'^{2})(2+2\gamma'+\gamma'^2)+\vp<17\gamma'
+\vp,
\end{align}
since $\gamma'<10^{-7}$. The ${\mathrm C}^*$-algebra $C_0(0,1]\otimes
F$ is finitely presented and weakly semiprojective,
\cite[Chapter 14]{Loring.LiftingBook}, and has nuclear dimension 1, \cite{Zacharias.NuclearDimension}. The
choice of $\eta$ gives
\begin{equation}
3\eta<\min\, \{1/17,(81\gamma^{1/2})^2/128\},
\end{equation}
since $(81\gamma^{1/2})^2/128>51\gamma>3\eta$. We may now
apply Lemma \ref{Direct.Easy} with $X$ replaced by $Z$ to conclude
that there is a finite subset $X_1$ of the unit ball of $C_0(0,1]\otimes F$
with the following property. If $\theta :  C_0(0,1]\otimes F\to B$ is an
$(X_1,3\eta)$-approximate $*$-homomorphism, then there is a
$*$-homomorphism $\pi :  C_0(0,1]\otimes F\to B$ with
$\pi\approx_{Z,81\gamma^{1/2}}\theta$. We may enlarge $X_1$
if necessary so that $Z\subseteq X_1$.

Given an order zero map $\phi :  F\to A$, Proposition
\ref{Near.OZ.Cone} gives a $*$-homomorphism $\rho_\phi  : 
C_0(0,1]\otimes F\to A$ defined by
$\rho_\phi({\mathrm{id}}_{(0,1]}\otimes x)=\phi(x)$ for
$x\in F$. The nuclear dimension of $C_0(0,1]\otimes F$ is 1, so Lemma \ref{Near.DirectLem} gives a cpc map
$\theta :  C_0(0,1]\otimes F\to B$ with
\begin{equation}
\|\rho_\phi(x)-\theta(x)\|\leq
4(2\gamma'+\gamma'^2)(2+2\gamma'+\gamma'^2)+\vp=\eta
\end{equation}
for $x\in X_1\cup X_1^*\cup\{xx^*\colon x\in
X_1\cup X_1^*\}$. Since $\rho_\phi$ is a $*$-homomorphism,
$\theta$ is an $(X_1,3\eta)$-approximate $*$-homomorphism.
By choice of $X_1$, there is a $*$-homomorphism $\pi : 
C_0(0,1]\otimes F\to B$ with $\pi\approx_{Z,81\gamma^{1/2}}\theta$, and the
inequality $\eta<17 \gamma$ leads to
$\pi\approx_{Z,82\gamma^{1/2}}\rho_\phi$. Thus the choice of
$Z$ ensures the existence of a unitary $u$ with $\|u-1\|\leq
246\gamma^{1/2}$ such that
$\rho_\phi\approx_{X,\beta}{\mathrm{Ad}}(u)\circ \pi$. Define
an order zero map $\psi :  F\to B$ by
$\psi(x)=\pi({\mathrm{id}}_{(0,1]}\otimes x)$ for $x\in F$.
Then $\phi\approx_{X_0,\beta}{\mathrm{Ad}}(u)\circ \psi$.
Since $X_0$ is a $\beta$-net for the unit ball of $A$, a
simple approximation argument gives
\begin{equation}
\|\phi(x)-({\mathrm{Ad}}(u)\circ\psi)(x)\|\leq
3\beta\|x\|,\qquad
x\in F.
\end{equation}
Recalling that $F$ has dimension $m$, we find that
\begin{equation}
\|\phi-{\mathrm{Ad}}(u)\circ \psi\|_{{\mathrm{cb}}}\leq
m\|\phi-{\mathrm{Ad}}(u)\circ \psi\|\leq
3m\beta<\gamma^{1/2}.
\end{equation}
We also have the estimate
\begin{equation}
\|\psi-{\mathrm{Ad}}(u)\circ \psi\|_{{\mathrm{cb}}}\leq
2\|u-1\|< 492 \gamma^{1/2},
\end{equation}
and the desired conclusion
$\|\phi-\psi\|_{{\mathrm{cb}}}\leq 493\gamma^{1/2}$ follows from the previous two inequalities.
\end{proof}

We end by   using our methods to give a new characterisation of when a
separable nuclear C$^{*}$-algebra is  $D$-stable, where $D$ is any
separable strongly self-absorbing C$^{*}$-algebra. For simplicity we only state
and prove a unital version, but it seems clear that, with some extra
effort, one can use \cite[Theorem~2.3]{Toms.StrongSelfAbsorbing} to
give a non-unital version as well. We first establish some notation. For a 
${\mathrm C}^*$-algebra $A$, $\prod_{n=1}^{\infty} A$ will denote the space of bounded sequences with entries from $A$ while $\sum_{n=1}^{\infty} A$ is the ideal of sequences $\{a_n\}_{n=1}^{\infty}$ for which $\lim_{n\to\infty}\|a_n\|=0$. We write $A_{\infty}$ for the quotient space and $\pi$ for the quotient map of $\prod_{n=1}^{\infty} A$ onto $A_{\infty}$. We identify $A$ with  a subalgebra of $A_{\infty}$ by first regarding $A$ as the algebra of constant sequences in 
$\prod_{n=1}^{\infty} A$ and then applying $\pi$. The relative commutant $A'\cap A_{\infty}$ is the algebra of central sequences.  
\begin{theorem}
Let $D$ be a  separable unital strongly self-absorbing ${\mathrm C}^*$-algebra  and let $\gamma$ satisfy  $0 \leq
\gamma\leq1/169$. Suppose that $A$ is a separable unital nuclear
${\mathrm C}^*$-algebra and that, for any finite subsets $X$ and $Y$ of the unit balls
of $A$ and $D$ respectively, there exists a ucp $(Y,\gamma)$-approximate
$*$-homomorphism  $\theta:D\rightarrow A$ with
$\|\theta(y)x-x\theta(y)\|\leq\gamma$ for $x\in X$ and $y\in Y$.  Then
$A$ is $D$-stable.
\end{theorem}
\begin{proof}
Suppose that we are given a finite set $Z$ (containing $1_{D}$) in the unit
ball of $D$ and some $\varepsilon>0$. Since $4\gamma\leq1/17$ and $D$
is nuclear, Lemma \ref{Avg.1} shows that there exists a finite set $Y$
(also containing $1_{D}$) in the unit ball of $D$ so that, if
$\phi_{Y}:D \rightarrow A_{\infty}\cap A'$ is a
$(Y,4\gamma)$-approximate $*$-homomorphism,   then there is a
$(Z,\varepsilon)$-approximate $*$-homomorphism
$\psi_{Z,\varepsilon}:D\rightarrow A_\infty\cap A'$ near to
$\phi_{Y}$.

Let $\{a_n\}_{n=1}^{\infty}$ be dense in the unit ball of $A$. Fix $n$ and use
Lemma~\ref{Avg.ApproxDiagonal} to find positive real numbers
$(\lambda_i)_{i=1}^m$ summing  to $1$ and contractions
$\{b_i\}_{i=1}^m$ in $A$ such that
$\|\sum_{i=1}^m\lambda_ib_i^*b_i-1_A\|<1/n$ and
\begin{equation}\label{thm5.9e1}
\left\|\sum_{i=1}^m\lambda_i(a_jb_i^*\otimes b_i-b_i^*\otimes b_ia_j)\right\|_{A\
\potimes\ A}<1/n,\qquad 1\leq j\leq n.
\end{equation}
By our hypotheses, there is a ucp $(Y,\gamma)$-approximate
$*$-homomorphism $\theta:D\rightarrow A$ with
$\|\theta(y)b_i-b_i\theta(y)\|\leq \gamma$ for $y\in Y\cup Y^* \cup
\{yy^*:y\in Y\cup Y^*\}$ and $i=1,\dots,m$.  Define a cpc map
$\phi_n:D\rightarrow A$ by
$\phi_n(x)=\sum_{i=1}^m\lambda_ib_i^*\theta(x)b_i$. For $y\in Y\cup
Y^* \cup\{yy^*:y\in Y\cup Y^*\}$, we have
\begin{align}
\|\phi_n(y)-\theta(y)\|&=\left\|\sum_{i=1}^m\lambda_ib_i^*\theta(y)b_i-\sum_{i=1}^
m
\lambda_{i} b_i^*b_i\theta(y)\right\|+\left\|(\sum_{i=1}^m \lambda_{i}
b_i^*b_i-1_A) \theta(y)\right\|\notag\\
&<\gamma+1/n.
\end{align}
As $\theta$ is a $(Y,\gamma)$-approximate $*$-homomorphism, $\phi_n$
is a $(Y,4\gamma+3/n)$-approximate $*$-homo\-morphism.     Therefore
we can define a $(Y,4\gamma)$-approximate $*$-homo\-morphism
$\phi_Y:D\rightarrow A_\infty$ by $\phi_Y(x)=\pi((\phi_n(x)))$, where
$\pi:\prod_{\mathbb{N}} A\rightarrow A_\infty$ is the quotient map.
For each $d$ in the unit ball of $D$, the map $x_1\otimes x_2\mapsto
x_1\theta(d)x_2$ extends to contractive linear map from $A\ \potimes\
A$ into $A$.   The estimate (\ref{thm5.9e1}) then gives
$\|a_j\phi_n(d)-\phi_n(d)a_j\|<1/n$ for $j\leq n$ and so $\phi_Y$
takes values in $A_\infty\cap A'$. Moreover, as $\theta$ is unital, we
have $\|\phi_{Y}(1_{D}) - 1_{A_{\infty}}\|\leq \gamma$. Now, by our
choice of $Y$, Lemma \ref{Avg.1} yields a
$(Z,\varepsilon)$-approximate $*$-homomorphism
$\psi_{Z,\varepsilon}:D\rightarrow A_\infty\cap A'$ such that $\|
\psi_{Z,\varepsilon} - \phi_Y\| \leq 12 \gamma^{1/2}$.

Using separability of $D$, upon increasing $Z$ and decreasing
$\varepsilon$ we obtain a $*$-homomorphism $
\psi:D \rightarrow (A_{\infty} \cap A')_{\infty}$ such that $\|
\psi(1_{D}) - 1_{(A_{\infty})_{\infty}}\| \leq 12 \gamma^{1/2} +
\gamma < 1$. The latter in particular implies that $\psi$ is unital.
As $D$ is nuclear, we may use the Choi--Effros lifting theorem to
obtain a ucp lift $\bar{\psi}:D \rightarrow \prod_{\mathbb{N} \times
\mathbb{N}} A$. Now a standard diagonal argument yields a unital
$*$-homomorphism $\tilde{\psi}:D \to A_{\infty} \cap A'$. By
\cite[Theorem~2.2]{Toms.StrongSelfAbsorbing} this shows that $A$ is
$D$-stable.
 \end{proof}

\end{document}